\documentclass[12pt, reqno]{amsart}

\usepackage[utf8]{inputenc}
\usepackage{float}
 \usepackage[hyphens]{url}
\usepackage{microtype}
\usepackage{algorithm}
\usepackage{algpseudocode}
\usepackage{mdframed}
\usepackage{needspace}
\usepackage{caption}
\usepackage{aliascnt}
\usepackage{enumitem}
\usepackage{mathtools}
\usepackage{mathrsfs}
\usepackage{bbm}
\usepackage{amsmath,amssymb,amsthm,bm}
\usepackage[top    = 26mm,  
            bottom = 26mm, 
            right  = 25mm, 
            left   = 25mm]{geometry}
\usepackage{graphicx}
\usepackage{subcaption}
\usepackage{url}
\usepackage{tikz}
\usepackage{pgfplots} 
\pgfplotsset{compat=newest} 
\usetikzlibrary{intersections, pgfplots.fillbetween}
\pgfdeclarelayer{pre main}
\usepackage[colorlinks=true,
            linkcolor=blue,
            filecolor=blue,
            citecolor=red,
            urlcolor=blue, 
            colorlinks=true]{hyperref}
\usepackage[capitalize,noabbrev]{cleveref}
\usepackage{extpfeil}
\usepackage{tikz}
\usepackage{tikz-cd}
\usetikzlibrary{patterns}

\newtheorem{theorem}{Theorem}[section]

\newaliascnt{proposition}{theorem}
\newtheorem{proposition}[proposition]{Proposition}
\aliascntresetthe{proposition}

\newaliascnt{lemma}{theorem}
\newtheorem{lemma}[lemma]{Lemma}
\aliascntresetthe{lemma}

\newaliascnt{corollary}{theorem}

\aliascntresetthe{corollary}

\theoremstyle{definition}

\newaliascnt{definition}{theorem}
\newtheorem{definition}[definition]{Definition}
\aliascntresetthe{definition}

\newaliascnt{problem}{theorem}
\newtheorem{problem}[problem]{Problem}
\aliascntresetthe{problem}

\newaliascnt{example}{theorem}
\newtheorem{example}[example]{Example}
\aliascntresetthe{example}

\newaliascnt{remark}{theorem}
\newtheorem{remark}[remark]{Remark}
\aliascntresetthe{remark}

\newcommand{\Span}{\mathrm{span}}

\newcommand{\Mat}{\mathrm{Mat}}

\newcommand{\PP}{\mathbb{P}}

\newcommand{\R}{\mathbb{R}}

\newcommand{\OGr}{\mathrm{OGr}}
\newcommand{\Gr}{\mathrm{Gr}}
\newcommand{\sech}{\mathrm{sech}}

\newcommand{\disk}{\mathbb{D}}

\newcommand{\wt}{{\mathrm{wt}}}
\newcommand{\Rpos}{\mathbb{R}_{>0}}

\newcommand{\rt}{\vec{\tau}_{\mathscr{M}}}
\newcommand{\lvec}[1]{\reflectbox{$\vec{\reflectbox{$#1$}}$}}
\newcommand{\lt}{\lvec{\tau}_{\mathscr{M}}}

\tikzset{qvert/.style={draw,black,circle,fill=gray,minimum size=5pt,inner sep=0pt}  } 
\tikzset{bvert/.style={draw,circle,fill=black,minimum size=5pt,inner sep=0pt}  }  
\tikzset{gbvert/.style={draw, gray, circle,fill=gray,minimum size=5pt,inner sep=0pt}  } 
\tikzset{gvert/.style={draw,gray,circle,fill=white,minimum size=5pt,inner sep=0pt}  } 
\tikzset{wvert/.style={draw,circle,fill=white,minimum size=5pt,inner sep=0pt}  } 
\tikzset{fvert/.style={text=MidnightBlue}  } 
\tikzset{sqvert/.style={draw,black,rectangle,fill=black,minimum size=5pt,inner sep=0pt}  } 
\tikzset{lvert/.style={draw,circle,fill=black,minimum size=4pt,inner sep=0pt}  }  
\usetikzlibrary{arrows}
\tikzcdset{arrow style=tikz, diagrams={>={Stealth[round,length=4pt,width=4.95pt,inset=2.75pt]}}}
\tikzset{nvert/.style={draw,circle,fill=black,minimum size=3pt,inner sep=0pt}} 
\usetikzlibrary{decorations.markings}
\tikzset{>=latex}
\def\ep{0.16}
\def\rc{7}
\def\lw{1.1pt}

\title{$C$-networks and the planar Ising inverse problem}

\author[Yassine Elmaazouz]{Yassine {Elmaazouz}} 

\address{Yassine {ElMaazouz} (Caltech)}
\email{maazouz@caltech.edu}

\author[Terrence George]{Terrence {George}} 

\address{Terrence {George} (TIFR-CAM)}
\email{terrence@tifrbng.res.in}

\date{\today}

\begin{document}

    \begin{abstract}
    We solve the inverse problem for Ising models on reduced planar graphs in a disk, i.e., recovering the edge coupling constants from the boundary spin correlations. A recursive solution to this problem was provided by Galashin--Pylyavskyy. Our solution is nonrecursive. It is based on an Ising analog of the chamber ansatz which asserts that the inverse map factors through variables on the graph that transform under the Ising Y-$\Delta$ move according to the discrete CKP equation.
    \end{abstract}

    \maketitle

    \setcounter{tocdepth}{1}
    \tableofcontents

\section{Introduction}

An \emph{Ising model} in a disk $\disk$ consists of a planar graph $G=(V,E,F)$ embedded in $\disk$ with $n$ boundary vertices labeled $b_1^\partial,\dots,b_n^\partial$ in clockwise cyclic order together with a function $J:E \rightarrow \R_{>0}$ called the \emph{coupling constant}. A \emph{spin configuration} is a function 
\(
\sigma : V \rightarrow \{ -1,1\}
\)
assigning to each vertex $v \in V$ a spin $\sigma_v$. The probability of a spin configuration is defined as
\[
\mathbf{P}(\sigma):= \frac{1}{Z} \prod_{e = uv \in E} e^{J_e \sigma_u \sigma_v},
 \quad
\text{where}
\quad
Z := \sum_{\sigma: V \rightarrow \{-1,1\} } \prod_{e = uv \in E} e^{J_e \sigma_u \sigma_v} 
\] is the \emph{partition function}. We denote by $[n]$ the finite set $[n]:=\{1, \dots, n\}$ and by $\binom{[n]}{k}$ the set of all subsets of $[n]$ of size $k$. Given $i,j \in [n]$, we define the \emph{boundary correlation}  
\[
\langle \sigma_i \sigma_j \rangle : = \sum_{\sigma : V \rightarrow \{-1,1\}} \mathbf{P}(\sigma) \sigma_{b_i^\partial} \sigma_{b_j^\partial}.
\]
These correlations are organized into an $n \times n$ symmetric real matrix with $1$s on the diagonal
\[
M(G,J):=(\langle \sigma_i \sigma_j \rangle)_{ i,j \in [n]},
\]
called the \emph{boundary correlation matrix}. We write
\[
\operatorname{Corr}_G:
\mathcal I_G^{>0}
\rightarrow
\Mat_n^{\mathrm{sym}}(\R,1),
\qquad
J\mapsto M(G,J),
\]
where
\[
\mathcal I_G^{>0}
:= \Big\{ J:E(G)\to\Rpos \Big\}
\]
is the {space of Ising models on the graph $G$} and \(\Mat_n^{\mathrm{sym}}(\R,1)\) denotes the space of real symmetric \(n\times n\) matrices with \(1\)s on the diagonal.

{
We begin with planar graphs in a disk for simplicity. Below we briefly
explain the extension to cactus graphs, which will be the natural setting
for the results of the paper.
}

{
The \emph{inverse problem for the Ising model} considered in this paper is the following. Given a specified reduced graph $G$ and its boundary correlation matrix $M(G,J)$, recover the coupling constants
\[
J:E(G)\rightarrow \R_{>0}.
\]
The underlying Ising graph $G$ is part of the input. The problem is to recover the weights $J$.
}

In order for this inverse problem to have a unique solution, one must
impose a reducedness assumption on $G$. This is analogous to the
situation for electrical networks. Lam~\cite{Lam} showed that reduced
electrical networks in a disk, modulo certain local transformations of
the graph called Y-$\Delta$ moves, are classified by matchings of
$[2n]$. In the Ising setting, the same combinatorial object appears:
to a reduced graph $G$ one associates a matching $\pi$ of
$[2n]$. 

{
For technical reasons, it is convenient to work throughout with the
slightly more general class of \emph{cactus graphs}. Briefly, if $P$
is a noncrossing partition of the boundary vertices of a disk, one
forms a cactus by identifying boundary vertices belonging to the same
part of $P$. A cactus graph is a graph embedded in the resulting
quotient space. Ordinary planar graphs in a disk are the special case
in which $P$ is trivial partition into $n$ singleton blocks. We give the precise definition in Section~\ref{sec:cactus_graphs}. From that point onward,
and unless stated otherwise, $G$ will denote a reduced cactus graph and all our constructions will be compatible with Y-$\Delta$ moves.
}

Galashin--Pylyavskyy~\cite{GalashinPylyavskyy} gave a recursive solution to
the Ising inverse problem using operations called adjoining a boundary edge
and adjoining a boundary spike. These operations are the Ising analogs of
operations for electrical networks introduced by
Curtis--Ingerman--Morrow~\cite{CIM98}. The purpose of this paper is to give a direct and
nonrecursive solution to the Ising inverse problem using ideas that originate in the chamber ansatz of Berenstein--Fomin--Zelevinsky~\cite{BFZ}, later developed for the positive Grassmannian by Marsh--Scott~\cite{MarshScott}, extended to arbitrary positroid cells by Muller--Speyer~\cite{MullerSpeyer}, and applied to
electrical networks in the big cell by George~\cite{GeorgeElectrical}. In this
paper, we develop an Ising analog of this chamber-ansatz framework. 

\bigskip

We define two auxiliary functions $s, c : E(G) \rightarrow (0,1)$ as follows: 
\[
s_e := \operatorname{sech}(2J_e),\qquad c_e:= \operatorname{tanh}(2J_e). 
\]
Rather than solving directly for the coupling constant \(J\), we solve for the functions \(s\) and \(c\) from which we can recover $J$ as $\frac12\operatorname{arctanh}(c)$. The chamber ansatz for the Ising model asks that these functions factor through a collection of positive variables attached to the vertices and faces of \(G\). More precisely, a \emph{\(C\)-network} on \(G\) is a function
\[
C:V(G) \sqcup F(G) \rightarrow \R_{>0}
\]
defined modulo rescaling by $\Rpos$. Let \[ \mathcal C_G^{>0} \cong \Rpos^{|V(G)| + |F(G)|-1} \]
be the space of $C$-networks with underlying graph $G$. In terms of the $C$-network, the functions $s$ and $c$ are given by
\begin{equation}\label{eq:intro_s_and_c}
s_e := \sqrt{\frac{C_fC_g}{C_uC_v+C_fC_g}}, \qquad      c_e := \sqrt{\frac{C_uC_v}{C_uC_v+C_fC_g}},  
\end{equation}
where $e=uv$ and $f,g$ are the two faces adjacent to $e$. 
In other words,~\eqref{eq:intro_s_and_c} defines a map 
\[
q_{G}:\mathcal C_G^{>0} \rightarrow \mathcal I_G^{>0}.
\]
There is a positive torus
\begin{equation}\label{pos:torus}
T_\pi^{>0}
\subset
\Rpos^{2n},
\end{equation}
depending only on the matching \(\pi\), that acts on
\( C_G^{>0} \). The map \(q_G\) is invariant under this action, and therefore descends to a map
\[
q_G:\mathcal C_G^{>0}/T_\pi^{>0}\rightarrow \mathcal I_G^{>0}.
\]
In this framework, solving the inverse problem reduces to constructing the dotted map making the following diagram commute:
\begin{equation}\label{eq:dotted}
\begin{tikzcd}
\mathcal C_{G}^{>0}/T_\pi^{>0}
  \arrow[rr, "q_G"]  
& &
\mathcal I_{G}^{>0} 
\\
& {\{\substack{\mathrm{Boundary~correlation} \\ \mathrm{matrices~on~} G}\}}  \subset \Mat_n^{\mathrm{sym}}(\R,1) \arrow[ur, "\operatorname{Corr}_{G}^{-1}"'] \arrow[ul, dotted ,"?"]  &
\end{tikzcd}
\end{equation}

The terminology $C$-network is motivated by the transformation rule for the variables $C$
under the Y-$\Delta$ move. This transformation is the discrete CKP equation
in the form introduced by Kashaev in his study of the Ising Y-$\Delta$ move~\cite{Kashaev}; see also Schief~\cite{Schief}. It is Cayley's $2\times 2\times 2$
hyperdeterminant~\cite{GKZ} and appears in the work of 
Holtz--Sturmfels~\cite{HoltzSturmfels} on the principal minor assignment problem for symmetric matrices. Related incarnations of the same equation appear in the work of Kenyon--Pemantle \cite{KenyonPemantle1,KenyonPemantle2}, Leaf~\cite{Leaf}, Melotti~\cite{Melotti}, and Arthamonov--Harnad--Hurtubise~\cite{AHH}, where it is connected to statistical mechanics, cluster algebras, and Grassmannians. It also appears in the work of Doliwa on discrete integrable systems~\cite{Doliwa} and in discrete differential geometry in the work of Bobenko--Schief~\cite{BS2,BobenkoSchief}. The chamber ansatz formula~\eqref{eq:intro_s_and_c} appears in Melotti~\cite[Remark~3]{Melotti} and, up to a sign, in the formula (10) for the partial correlations in Sturmfels--Tsukerman--Williams~\cite{STW}.

We will construct the dotted map in~\eqref{eq:dotted} via a modification of the twist map of Muller--Speyer~\cite{MullerSpeyer} for positroid cells. We first recall how boundary correlation matrices are related to the totally nonnegative Grassmannian. The \emph{Grassmannian} $\Gr_{n,2n}$ is the space of $n$-dimensional subspaces of $\R^{2n}$. Such a subspace $V\subset \R^{2n}$ can be represented as the rowspan of a full-rank $n\times 2n$ matrix. Its maximal minors
\[
\Delta_I(V),
\qquad
I\in \binom{[2n]}{n},
\]
are homogeneous coordinates on $\Gr_{n,2n}$, called \emph{Pl\"ucker coordinates}. The \emph{totally nonnegative Grassmannian} is the subset
\[
\Gr_{n,2n}^{\geq 0}
:=
\left\{
V\in \Gr_{n,2n}:
\Delta_I(V)\geq 0
\text{ for all } I\in \binom{[2n]}{n}
\right\}.
\]
Postnikov~\cite{Postnikov2006} showed that $\Gr_{n,2n}^{\geq 0}$ decomposes into \emph{positroid cells}
\[
\Gr_{n,2n}^{\geq 0}
=
\bigsqcup_{\mathscr M}
\Pi_{\mathscr M}^{>0},
\quad
\text{where}
\quad
\Pi_{\mathscr M}^{>0}
:=
\left\{
V\in \Gr_{n,2n}^{\geq 0}:
\Delta_I(V)>0
\text{ if and only if }
I\in \mathscr M
\right\}.
\]
The indexing sets $\mathscr M\subset \binom{[2n]}{n}$ that arise in this way are called \emph{positroids}. Equivalently, positroid cells are indexed by (decorated) permutations of $[2n]$; we denote the positroid corresponding to a permutation $\pi$ by $\mathscr{M}_\pi$.

\medskip

Galashin--Pylyavskyy~\cite[Section 2.3]{GalashinPylyavskyy} define a map
\[
\phi_n:
\Mat_n^{\mathrm{sym}}(\R,1)
\rightarrow
\Gr_{n,2n},
\]
called the \emph{doubling map}, whose image lies in the orthogonal Grassmannian
\[
\OGr^{+}_{n,2n} :=
\left\{
V\in \Gr_{n,2n}:
\Delta_I(V)=\Delta_{[2n]\setminus I}(V)
\text{ for all } I\in \binom{[2n]}{n}
\right\}.
\]
Moreover, if $G$ is reduced and $\pi$ is the associated matching of $[2n]$, then $\pi$ determines a positroid $\mathscr M_{\pi}$ (viewing the matching $\pi$ as a permutation of $[2n]$), and 
\[
\phi_n(M(G,J))
\in
\OGr_{n,2n}
\cap
\Pi_{\mathscr M_{\pi}}^{>0}.
\]

Next, we relate $C$-networks to the totally nonnegative Grassmannian. $C$-networks on graphs $G$ with the maximal matching (i.e., $i \mapsto i+n ~(\text{mod}~2n))$ were first defined by Kenyon--Pemantle~\cite{KenyonPemantle2}, who used them to parameterize the space of positive-definite symmetric matrices. Here, the $C$-variables are principal and almost-principal minors of the matrix. Equivalently, this space can be identified with the positive part ${\rm LG}^{> 0}_{n,2n}$ of a Lagrangian Grassmannian $\text{LG}_{n,2n} \subset \Gr_{n,2n}$ with respect to the symplectic form 
\[
\omega (x,y) := \sum_{i=1}^{n} (-1)^{i-1} (x_{i}y_{i+n} - x_{i+n}y_{i}), \quad \text{for } x,y \in \R^{2n}.
\]
This semialgebraic set appears in work of~Shevchenko~\cite{Shevchenko}, and differs from the totally positive Lagrangian Grassmannian (in Lusztig’s sense~\cite{Lusztig}) studied by Karpman~\cite{Karpman} and the totally positive Lagrangian Grassmannian arising in the theory of electrical networks in work of Lam~\cite{Lam}, Chepuri--George--Speyer~\cite{CGS} and Bychkov--Gorbounov--Kazakov--Talalaev~\cite{BGKT}, which use other symplectic forms. Under this identification with ${\rm LG}^{>0}_{n,2n}$, the $C$-variables, i.e., the principal and almost-principal minors, are Pl\"ucker coordinates of the point in $\text{LG}^{>0}_{n,2n}$ and this construction can be viewed as a specialization of a construction of Scott~\cite{scott} for the totally nonnegative Grassmannian. Generalizing this, we define subsets 
\[
\Lambda_{\pi}^{>0} \subset \Pi_{\mathscr M_{\pi}}^{>0},
\] 
which we call \emph{locally Lagrangian positroid cells}, together with maps
\[
\bm H_G:
\Lambda_\pi^{>0}
\rightarrow
\mathcal C_G^{>0},
\]
where each $C$-variable is given by a Pl\"ucker coordinate of the point in $\Lambda_{\pi}^{>0}$. 
We show that the conditions defining $\Lambda_{\pi}^{>0}$ inside $\Pi_{\mathscr M_{\pi}}^{>0}$ may be viewed as local Lagrangian conditions, thereby justifying the terminology. We then prove that $\bm H_G$ is a bijection and that these bijections are compatible
with Y-$\Delta$ moves.

\begin{remark}
    A different collection of subsets of the Grassmannian generalizing the big cell of the Lagrangian Grassmannian appears in the study of Gaussoids by Boege--D'Al\`i--Kahle--Sturmfels \cite{gaussoids}.
\end{remark}

The final ingredient is a modification of the twist map. For an
arbitrary positroid \(\mathscr M\), Muller--Speyer~\cite{MullerSpeyer} define
right and left twist maps
\[
\rt,\lt:
\Pi_{\mathscr M}^{>0}
\rightarrow
\Pi_{\mathscr M}^{>0},
\]
which are inverse to each other and which relate the \(A\)- and \(X\)-cluster
coordinates of the positroid cell. In the Ising setting, the twist maps must be modified by the natural action of $\Rpos^{2n}$ rescaling the columns of a matrix representative to be compatible with $q_G$; a similar modification was required for electrical networks in George~\cite{GeorgeElectrical}. 

The positive torus~\eqref{pos:torus} acts on $\Lambda_{\pi}^{>0}$, and the map $\bm H_G$ is equivariant and therefore descends to a map
\[
\bm H_G:
\Lambda_{\pi}^{>0}/T_{\pi}^{>0}
\rightarrow
\mathcal C_G^{>0}/T_{\pi}^{>0}.
\]
The \emph{Ising twist and inverse twist maps} 
\[
\hat{\tau}_{\mathscr M_\pi}:
\Lambda_{\pi}^{>0}/T_{\pi}^{>0}
\to
\OGr_{n,2n}
\cap
\Pi_{\mathscr M_{\pi}}^{>0}, \qquad
\hat{\tau}_{\mathscr M_{\pi}}^{-1}:
\OGr_{n,2n}
\cap
\Pi_{\mathscr M_{\pi}}^{>0}
\to
\Lambda_{\pi}^{>0}/T_{\pi}^{>0},
\]
are defined after passing to the quotient \(\Lambda_\pi^{>0}/T_\pi^{>0}\). We can now state our main theorem.

\begin{theorem}\label{thm:intro_main_thm_inverse}
    The Ising twist map $\hat \tau_{{\mathscr{M}}_{\pi}}:\Lambda_{{\pi}}^{>0}/T_{\pi}^{>0} \rightarrow \OGr^{\ge 0}_{n,2n} \cap \Pi_{\mathscr{M}_{\pi}}^{>0} $ is a bijection whose inverse is $\hat \tau^{-1}_{\mathscr{M}_{\pi}}$. They fit into the following commutative diagram in which every map is a bijection:
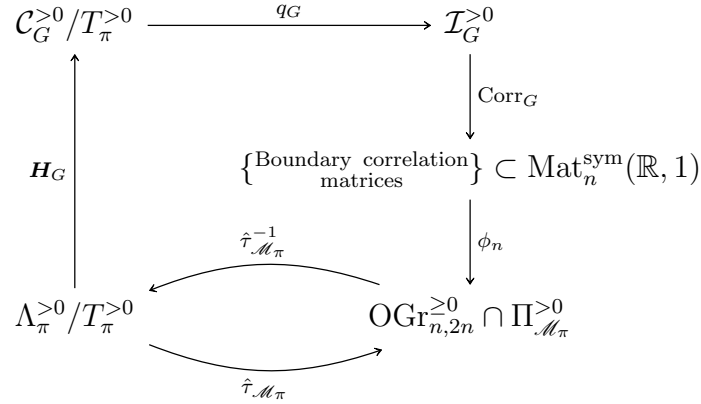
\[
\begin{tikzcd}[row sep = large]
\mathcal C_{G}^{>0}/T_{\pi}^{>0} 
  \arrow[r, "q_G"] 
& 
\mathcal I_{G}^{>0} 
  \arrow[d, "\operatorname{Corr}_{G}"] 
\\
& {\{\substack{\mathrm{Boundary~correlation} \\ \mathrm{matrices~on~} G}\}} \subset \Mat_n^{\mathrm{sym}}(\R,1)   \arrow[d, "\phi_n"] 
\\
\Lambda_{{\pi}}^{>0}/T_{\pi}^{>0}
  \arrow[uu, "\bm H_{G}"]   
  \arrow[r,
    bend right=20,
    swap,
    "{\hat \tau_{{\mathscr{M}}_{\pi}}}"
  ] 
& 
\OGr^{\ge 0}_{n,2n} \cap \Pi_{\mathscr{M}_{\pi}}^{>0}
  \arrow[l,
    bend right=20,
    swap,
    "\hat \tau^{-1}_{\mathscr{M}_{\pi}}"
  ]
\end{tikzcd}
\]
\end{theorem}

In particular, the dotted map in~\eqref{eq:dotted} which solves the inverse problem is given by the composition
\[
 \bm H_G \circ \hat \tau^{-1}_{\mathscr{M}_{\pi}} \circ \phi_n.
\]
{
We refer the reader to Algorithm \ref{alg:ising_inverse} for the detailed nonrecursive algorithm using the ingredients of \cref{thm:intro_main_thm_inverse}.
}

\subsection*{Open problems}

We end the introduction with some open problems that we believe are interesting for further study.

\begin{problem}
It is natural to study the algebro-geometric counterparts of locally Lagrangian positroid cells, i.e., the Zariski closures of the semialgebraic sets $\Lambda_\pi^{>0}$ inside the complex Grassmannian, similarly to the positroid varieties of Knutson--Lam--Speyer~\cite{KLS}.
\end{problem}

\begin{problem}
The chamber ansatz formula~\eqref{eq:intro_s_and_c} also appears in
formula~(10) for the partial correlations in Sturmfels--Tsukerman--Williams~\cite{STW}. This suggests that the two settings
may be related and it would be very interesting to understand the precise nature of this relation.
\end{problem}

\begin{problem}
It would be interesting to find the electrical-network analogs of the locally Lagrangian positroid cells. For the maximal matching, this is a totally positive orthogonal Grassmannian; see Henriques--Speyer~\cite{HS} and George~\cite{GeorgeElectrical}. 
\end{problem}

\begin{problem}
    The target space of the inverse map $\hat{\tau}^{-1}_{\mathscr{M}_{\pi}}$ is the quotient $\Lambda_{\pi}^{> 0}/ T^{ > 0}_{\pi_0}$ of the locally Lagrangian positroid cell modulo the action of $T_{\pi}^{> 0}$. When $\pi$ is the matching $\pi_{0} \colon i \mapsto i + n~({\rm mod}~2n)$, the cell $\Lambda_{\pi_0}^{> 0}$ is the big cell ${\rm LG}^{> 0}_{n,2n}$ of the Lagrangian Grassmannian in \cite{Shevchenko}. It would be interesting to study the quotient $ \Lambda_{\pi_0}^{> 0} / T^{ > 0}_{\pi_0} = {\rm LG}^{> 0}_{n,2n} / T^{> 0}_{\pi_{0}}$ in light of \cite{PositiveConfSpace}.
\end{problem}

\subsection*{Organization of the paper}
In~\cref{sec:2}, we review the necessary background on the positive Grassmannian and its connections to the dimer model and $A$-networks. In~\cref{sec:3}, we recall the corresponding background on the positive orthogonal Grassmannian and its relation to the Ising model. In~\cref{sec:4}, we introduce $C$-networks and locally Lagrangian positroid cells. Finally, in~\cref{sec:ising_cluster_ensemble_mod_twist}, we define the maps appearing in the diagram of the main result, \cref{thm:intro_main_thm_inverse}, and prove the theorem.

{

\subsection*{Acknowledgements}
We thank the referee for carefully reading our manuscript, and for the many helpful comments and suggestions that substantially improved the exposition and the completeness of the paper.
}

   \section{Background on the positive Grassmannian, the dimer model and \texorpdfstring{$A$}{A}-networks}\label{sec:2}

    In this section we collect some background on the totally nonnegative Grassmannian which we shall need throughout this article.
    
   \subsection{The Grassmannian and positivity}
   Let $1 \leq k \leq n$ be positive integers. The real \emph{Grassmannian} $\Gr_{k,n} := \Gr_{k,n}(\R) $ is the moduli space of $k$-dimensional subspaces of the vector space $\R^n$.   A $k$-dimensional subspace $V \subset \R^{n}$ can be represented as the rowspan of a full-rank $k \times n$ matrix $M$, and the map
   \[
        \Gr_{k,n} \rightarrow \PP^{\binom{n}{k}-1} , \quad  V  = {\rm rowspan}(M) \mapsto \big(\Delta_I(V) = \det(M_I)\big)_{I \in \binom{[n]}{k}},
   \]
   is the \emph{Pl\"ucker embedding} of the Grassmannian $\Gr_{k,n}$ into the projective space $\PP^{\binom{n}{k}-1}$. Here $M_I$ denotes the $k\times k$ submatrix of $M$ consisting of the columns indexed by $I$. The Grassmannian $\Gr_{k,n}$ is a projective variety and the $\Delta_I(V)$ are called \emph{Pl\"ucker coordinates}. We will often omit the vector space in the notation and simply write $\Delta_I$. The \emph{totally positive Grassmannian} $\Gr_{k,n}^{> 0}$ is the semialgebraic subset of points in $\Gr_{k,n}$ such that $\Delta_{I} > 0$ for all $I \in \binom{[n]}{k}$. Its topological closure in $\PP^{\binom{n}{k}-1}$ is the \emph{totally nonnegative Grassmannian} consisting of points in {$\Gr_{k,n}$} such that $\Delta_{I} \geq 0$ for any $I \in \binom{[n]}{k}$. Given a point $V \in \Gr_{k,n}^{\geq 0}$, its \emph{matroid} $\mathscr{M}_V$ is defined by
   \[
   \mathscr{M}_V := \left\{I \in \binom{[n]}{k}: \Delta_I(V) \neq 0\right\}.
   \]
    In this article, we shall identify any rank $k$ matroid $\mathscr{M}$ on the ground set $[n]$ with its set of bases and write $\mathscr{M} \subset \binom{[n]}{k}$. Any matroid $\mathscr{M} \subset \binom{[n]}{k}$ such that there exists a $V \in \Gr_{k,n}^{\geq 0}$ with $\mathscr{M} = \mathscr{M}_V$ is called a \emph{positroid}. The corresponding \emph{positroid cell} is 
    \[
    \Pi_{\mathscr{M}}^{>0}:= \left\{V \in \Gr_{k,n}^{\geq 0}: \mathscr{M}_V = \mathscr{M} \right\}.
    \]
By definition of positroid cells, we have
\begin{equation}
    \Gr_{k,n}^{\geq 0} = \bigsqcup_{\substack{\text{positroids}~\mathscr{M} \\ {\text{of type } (k,n) } }} \Pi_{ \mathscr{M}}^{>0}, \label{eq:positroidstratification}
\end{equation}
where the union is taken over all positroids of rank $k$ on the ground set $[n]$.

Torus actions will play an important role in this paper and we will systematically keep track of actions on the various objects we define. In particular, we note that $\Rpos^n$ acts on $\Gr_{k,n}^{\geq 0}$ by rescaling the basis vectors of $\R^n$, i.e., by rescaling the columns of any matrix representative. More explicitly, we have
\[
    (t \cdot \Delta)_I = \left(\prod_{i \in I} t_i\right) \Delta_I, \quad \text{for } \Delta \in \Gr_{k,n}^{\geq 0} \text{ and } t = (t_1, \dots, t_n) \in \Rpos^n,
\]
and this action preserves each positroid cell $\Pi^{>0}_{\mathscr{M}}$.

   \subsection{Plabic graphs in the disk}

          A \emph{plabic graph} $\Gamma$ is a bipartite graph embedded in the disk $\disk$ such that:
          \begin{enumerate}
              \item $\Gamma$ has exactly $n$ degree-$1$ boundary vertices, labeled \( u_1^\partial, \dots, u_{n}^\partial \) in clockwise cyclic order, each of which is adjacent to exactly one internal vertex.
              \item The internal vertices are colored black or white and adjacent vertices have opposite colors.
          \end{enumerate}
        Although there is a unique choice of coloring of boundary vertices that makes $\Gamma$ bipartite, it will be more convenient to treat the boundary vertices as uncolored. We denote the vertices, edges and faces of $\Gamma$ by $V(\Gamma), E(\Gamma)$ and $F(\Gamma)$ respectively. 
        We write $B(\Gamma), W(\Gamma)$ and $\partial V(\Gamma)$ for the black, white and boundary vertices of $\Gamma$ respectively, so 
        \(
            V(\Gamma) = B(\Gamma) \sqcup W(\Gamma) \sqcup \partial V(\Gamma).
        \)
        
         We denote by $f^\partial_{i,i+1}$ the boundary face between $u_{i}^\partial$ and $u_{i+1}^\partial$. We denote by 
          \[
          k:= |W(\Gamma)|-|B(\Gamma)|+ \Big|\{b \in B(\Gamma): b~\text{is adjacent to the boundary}\}\Big|
          \]
          the \emph{helicity} of $\Gamma$ and say that $\Gamma$ is of \emph{type} $(k,n)$.

\begin{figure}[ht]
    \centering
    \scalebox{0.9}{
    \begin{tikzpicture}[scale=0.55]
        \begin{scope}[shift={(-3,0)}]
            \def\r{2};
            \fill[black!5] (0,0) circle (1.05*\r cm);
            \coordinate (n1) at (2,0);
           \coordinate[ label = left: $u_i^\partial$] (n2) at (-2,0);
            \draw[] (n1) -- (n2);
        \end{scope}
        
        \node at (0.5,0) {$\rightarrow$};	
        
        \begin{scope}[shift={(4,0)}]
            \def\r{2};
            \fill[black!5] (0,0) circle (1.05*\r cm);
            \draw[blue,line width=\lw] (-2,0) -- (45:2) 
                                       (-2,0) -- (-45:2);
             \coordinate[nvert,label = left: $u_i^\partial$] (n3) at (-2,0);
        \end{scope}

        \begin{scope}[shift={(10,0)}]
            \def\r{2};
            \fill[black!5] (0,0) circle (1.05*\r cm);
            \coordinate[bvert] (n1) at (2,0);
            \coordinate[wvert] (n2) at (-2,0);
            \draw[] (n1) -- (n2);
        \end{scope}
        
        \node at (13.5,0) {$\rightarrow$};	
        
        \begin{scope}[shift={(17,0)}]
          \def\r{2};
            \fill[black!5] (0,0) circle (1.05*\r cm);
            \coordinate[] (n1) at (135:\r);
            \coordinate[] (n2) at (225:\r);
            \draw[blue,line width=\lw] (n1) -- (-45:\r)
                                       (n2) -- (45:\r);
            \node[nvert] at (0,0) {};
        \end{scope}

        \draw[] (7,-2) -- (7,2);
        \node at (0.5,-3) {(a)};	
        \node at (13.5,-3) {(b)};	

        \begin{scope}[shift={(-3,-6)}]
            \def\r{2};
            \fill[black!5] (0,0) circle (1.05*\r cm);
            \coordinate[bvert] (n1) at (0,0);
            \draw[] (0,0) edge (2,0) edge (120:2) edge (240:2);
        \end{scope}
        \node at (0.5,-6) {$\rightarrow$};	
         \begin{scope}[shift={(4,-6)}]
          \def\r{2};
            \fill[black!5] (0,0) circle (1.05*\r cm);
            \coordinate[] (n1) at (135:\r);
            \coordinate[] (n2) at (225:\r);
      
            \draw[blue,->,line width=\lw,rounded corners=2*\rc] (20:2) --(60:1)-- (100:2);
            \draw[blue,->,line width=\lw,rounded corners=2*\rc] (120+20:2) --(120+60:1)-- (120+100:2);
            \draw[blue,->,line width=\lw,rounded corners=2*\rc] (240+20:2) --(240+60:1)-- (240+100:2);
        \end{scope}

        \begin{scope}[shift={(10,-6)}]
            \def\r{2};
            \fill[black!5] (0,0) circle (1.05*\r cm);
            \draw[] (0,0) edge (2,0) edge (120:2) edge (240:2);
            \coordinate[wvert] (n1) at (0,0);
        \end{scope}
            \node at (13.5,-6) {$\rightarrow$};	
         \begin{scope}[shift={(17,-6)}]
            \def\r{2};
            \fill[black!5] (0,0) circle (1.05*\r cm);
            \coordinate[] (n1) at (135:\r);
            \coordinate[] (n2) at (225:\r);
      
            \draw[blue,<-,line width=\lw,rounded corners=2*\rc] (20:2) --(60:1)-- (100:2);
            \draw[blue,<-,line width=\lw,rounded corners=2*\rc] (120+20:2) --(120+60:1)-- (120+100:2);
            \draw[blue,<-,line width=\lw,rounded corners=2*\rc] (240+20:2) --(240+60:1)-- (240+100:2);
        \end{scope}
        
        \draw[] (7,-8) -- (7,-4);
        \node at (0.5,-9) {(c)};	
        \node at (13.5,-9) {(d)};	
    \end{tikzpicture}}
    \caption{The construction of the oriented medial graph \( \Gamma^\times \) from \( \Gamma \) at (a) boundary edges, (b) internal edges, (c) black vertices and (d) white vertices. The vertices may be of any degree.}
    \label{fig:medial_graph_plabic}
\end{figure}
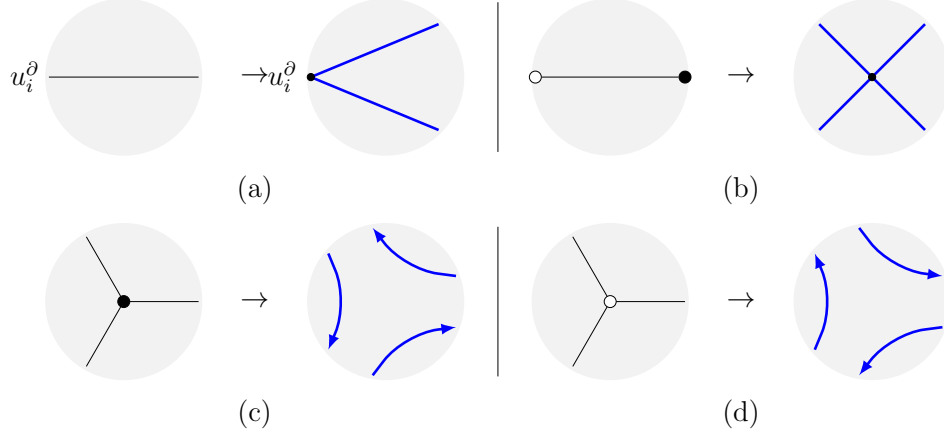

    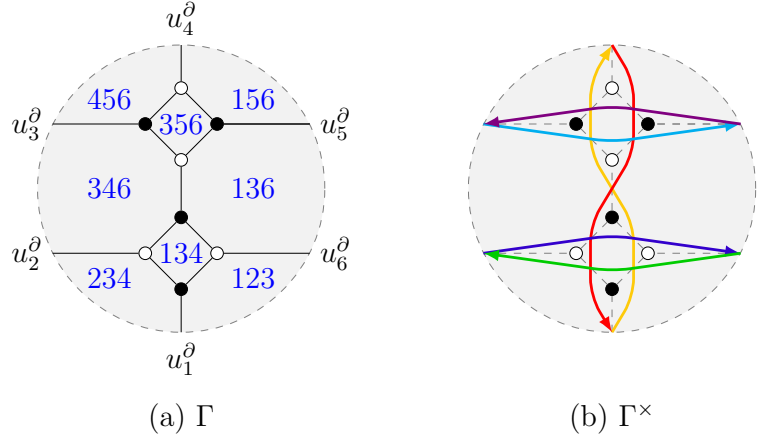
\begin{figure}[ht]
        \centering
        \scalebox{0.95}{
        \begin{tikzpicture}
        \begin{scope}[shift={(0,0)}]
            \def\r{2};
            \fill[black!5] (0,0) circle (1*\r cm);
            \draw[dashed, gray] (0,0) circle (1*\r cm);
            \coordinate[label=below:  $u_1^\partial$] (u1) at (0,-2);
            \coordinate[label=left:  $u_2^\partial$] (u2) at (-1.79,-0.9);
            \coordinate[label=left:  $u_3^\partial$] (u3) at (-1.79, 0.9);
            \coordinate[label=above:  $u_4^\partial$] (u4) at (0,2);
            \coordinate[label=right:  $u_5^\partial$] (u5) at (1.79, 0.9);
            \coordinate[label=right:  $u_6^\partial$] (u6) at (1.79, -0.9);

            \coordinate[bvert,] (v1) at (0,-1.4);
            \coordinate[bvert,] (v2) at (0,-0.4);
            \coordinate[wvert] (v3) at (-0.5,-0.9);
            \coordinate[wvert] (v4) at (0.5,-0.9);

            \coordinate[wvert] (v5) at (0,0.4);
            \coordinate[wvert] (v6) at (0,1.4);
            \coordinate[bvert] (v7) at (-0.5,0.9);
            \coordinate[bvert] (v8) at ( 0.5,0.9);
            
            \draw[] (u1) -- (v1);
            \draw[] (u2) -- (v3);
            \draw[] (u3) -- (v7);
            \draw[] (u4) -- (v6);
            \draw[] (u5) -- (v8);
            \draw[] (u6) -- (v4);

            \draw[] (v1) -- (v3);
            \draw[] (v1) -- (v4);
            \draw[] (v3) -- (v2);
            \draw[] (v4) -- (v2);
            \draw[] (v5) -- (v7);
            \draw[] (v5) -- (v8);
            \draw[] (v7) -- (v6);
            \draw[] (v8) -- (v6);
            \draw[] (v2) -- (v5);

            \node[scale=1, blue] at (-1, 1.25) {$456$};
            \node[scale=1, blue] at (-1, 0) {$346$};
            \node[scale=1, blue] at (-1, -1.25) {$234$};
            \node[scale=1, blue] at (0, 0.9) {$356$};
            \node[scale=1, blue] at (0, -0.9) {$134$};
            \node[scale=1, blue] at (1, 1.25) {$156$};
            \node[scale=1, blue] at (1, 0) {$136$};
            \node[scale=1, blue] at (1, -1.25) {$123$};
            
            \node[scale=1] at (0, -3.2) {(a) $\Gamma$};        
        \end{scope}

        \begin{scope}[shift={(6,0)}]
            \def\r{2};
            \fill[black!5] (0,0) circle (1*\r cm);
            \draw[dashed, gray] (0,0) circle (1*\r cm);
            \coordinate[] (u1) at (0,-2);
            \coordinate[] (u2) at (-1.79,-0.9);
            \coordinate[] (u3) at (-1.79, 0.9);
            \coordinate[] (u4) at (0,2);
            \coordinate[] (u5) at (1.79, 0.9);
            \coordinate[] (u6) at (1.79, -0.9);

            \coordinate[bvert] (v1) at (0,-1.4);
            \coordinate[bvert] (v2) at (0,-0.4);
            \coordinate[wvert] (v3) at (-0.5,-0.9);
            \coordinate[wvert] (v4) at (0.5,-0.9);

            \coordinate[wvert] (v5) at (0,0.4);
            \coordinate[wvert] (v6) at (0,1.4);
            \coordinate[bvert] (v7) at (-0.5,0.9);
            \coordinate[bvert] (v8) at ( 0.5,0.9);

            \draw[gray,dashed] (u1) -- (v1);
            \draw[gray,dashed] (u2) -- (v3);
            \draw[gray,dashed] (u3) -- (v7);
            \draw[gray,dashed] (u4) -- (v6);
            \draw[gray,dashed] (u5) -- (v8);
            \draw[gray,dashed] (u5) -- (v8);
            \draw[gray,dashed] (u6) -- (v4);

            \draw[gray,dashed] (v1) -- (v3);
            \draw[gray,dashed] (v1) -- (v4);
            \draw[gray,dashed] (v3) -- (v2);
            \draw[gray,dashed] (v4) -- (v2);
            \draw[gray,dashed] (v5) -- (v7);
            \draw[gray,dashed] (v5) -- (v8);
            \draw[gray,dashed] (v7) -- (v6);
            \draw[gray,dashed] (v8) -- (v6);
            \draw[gray,dashed] (v2) -- (v5);

            \draw[yellow!80!red,->,line width=\lw,rounded corners=\rc]
                (u1) -- (0.3,-1.5) -- (0.3,-0.5) -- (-0.3,0.5) -- (-0.3,1.5) -- (u4);
            \draw[red,->,line width=\lw,rounded corners=\rc]
                (u4) -- (0.3,1.5) -- (0.3,0.5) -- (-0.3,-0.5) --  (-0.3,-1.5) --  (u1); 
            \draw[blue!80!red,->,line width=\lw,rounded corners=\rc]
                (u2) --  (0,-0.65) -- (u6); 
            \draw[green!80!black,->,line width=\lw,rounded corners=\rc]
                (u6) --  (0,-1.15) -- (u2); 
            \draw[cyan,->,line width=\lw,rounded corners=\rc]
                (u3) --  (0,0.65) -- (u5); 
            \draw[violet,->,line width=\lw,rounded corners=\rc]
                (u5) --  (0,1.15) -- (u3); 
            \node[scale=1] at (0, -3.2) {(b) $\Gamma^\times$};
        \end{scope}
        \end{tikzpicture}}
        \caption{Example of the medial graph $\Gamma^\times$ and target face labels of a plabic graph $\Gamma$.}
        \label{fig:medialGraphExample}
    \end{figure}

Let $\Gamma^\times$ denote the \emph{directed medial graph} of $\Gamma$ constructed as illustrated in \cref{fig:medial_graph_plabic}. The boundary vertices of $\Gamma^\times$ naturally correspond to the boundary vertices of $\Gamma$; we therefore label them by $u_1^\partial,\dots,u_n^\partial$ as well. A \emph{strand} of \( \Gamma \) is a maximal directed path in $\Gamma^\times$. Each strand is either a closed loop or a directed path connecting two boundary vertices. The \emph{decorated permutation} $\pi_\Gamma^{:} =(\pi_\Gamma,\text{col})$ consists of:
    \begin{enumerate}
        \item A permutation \(\pi_\Gamma : [n] \rightarrow [n] \) defined by
        \[
            \pi_\Gamma(i) := j, \quad \text{if the strand starting at \( u_i^\partial \) ends at \( u_j^\partial \)}.
        \]
        \item When \( i \) is a fixed point of \( \pi_\Gamma \), a \emph{color} $\text{col}(i) \in \{ \bullet, \circ\}$. The color will not play a role for the plabic graphs arising from the Ising models considered in this paper, so we omit the details of its construction.
    \end{enumerate}

        We say that the graph $\Gamma$ is \emph{reduced} if the following hold:
        \begin{enumerate}
            \item No strand is a closed loop.
            \item No strand has a self-intersection.
            \item No two strands form a \emph{parallel bigon}, i.e., there are no two edges \( e_1, e_2 \in E(\Gamma) \) such that both strands go through \( e_1\) and \( e_2 \) with the same orientation.
        \end{enumerate}
    Henceforth, we assume that all our plabic graphs are reduced.
    
    \medskip
    
    Given a reduced plabic graph $\Gamma$ of type $(k,n)$, we define two maps
    \[
    \bm S, \bm T: F(\Gamma) \rightarrow \binom{[n]}{k},
    \]
    called \emph{source} and \emph{target labels}, respectively, as follows. Each strand of $\Gamma$ is assigned a label $i \in [n]$, determined by its source vertex or its target vertex, depending on the chosen convention. Since $\Gamma$ is reduced, there are exactly $n$ such strands labeled by $[n]$. Each face $f$ of $\Gamma$ is labeled by the subset of strand labels such that $f$ lies on the left side of the strand. It is easy to see that each face of $\Gamma$ is thus assigned a label $I \subset [n]$ of the same cardinality \( k \). 
    
    A \emph{Grassmann necklace} $\bm I = (I_1,\dots,I_n)$ is a sequence of $k$-element subsets of $[n]$ such that for each $i \in [n]$, there is a $j_i \in [n]$ such that $I_{i+1}=I_i \setminus \{i\} \cup \{j_i\}$. The \emph{Grassmann necklace} $\bm I_\Gamma = (I_1,\dots,I_n)$ of $\Gamma$ is defined by $I_{i} := \bm T(f^\partial_{i-1,i})$. Here $f^\partial_{i-1,i} \in F(\Gamma)$ denotes the boundary face of $\Gamma$ between the boundary vertices $u_{i-1}^\partial$ and $u_i^\partial$. Similarly, there is a dual object called the \emph{reverse Grassmann necklace} defined by $\bm I^{\text{rev}}_\Gamma = (I_1^\text{rev},\dots,I_n^\text{rev})$ where $I_i^{\text{rev}} := \bm S(f_{i,i+1}^\partial)$. 

   \begin{example}
   Figure~\ref{fig:medialGraphExample} shows a plabic graph $\Gamma$ with $\pi_\Gamma = 465132$ (in one-line notation) together with its medial graph $\Gamma^\times$ and target face labels, from which we obtain the Grassmann necklace 
   \[
       \bm I_\Gamma = (123,234,346,456,156,136).
   \]
   \end{example}

 A \emph{dimer cover} or \emph{almost perfect matching} of $\Gamma$ is a subset of the edges that uses every internal vertex exactly once and every boundary vertex at most once. The \emph{boundary} of a dimer cover $M$ is defined as follows:
    \[
\partial M := 
    \left\{ 
        i \in [n] \colon 
            \begin{matrix}
            u_i^\partial \text{ is adjacent to a  white vertex and covered by $M$}, \\ \text{or is adjacent to a black vertex and \emph{not} covered by $M$}
        \end{matrix} 
    \right\}.
    \]
The \emph{positroid} $\mathscr{M}_\Gamma$ of $\Gamma$ is then defined as
\[
\mathscr{M}_\Gamma := \left\{ I \in \binom{[n]}{k}: \text{there is a dimer cover $M$ with $\partial M=I$ }\right\}.
\]

The local modifications of plabic graphs shown in \cref{fig:moves:A} are called \emph{moves}. Two plabic graphs $\Gamma$ and $\Gamma'$ are said to be \emph{move-equivalent} if they are related by a sequence of moves.  

\begin{theorem}[Postnikov~\cite{Postnikov2006}] \label{thm:post_bijections}
  The maps $\Gamma \mapsto \pi_{\Gamma}^{:}$, $\Gamma \mapsto \bm{I}_{\Gamma}$ and $\Gamma \mapsto \mathscr{M}_\Gamma$ that associate to a reduced plabic graph $\Gamma$ (up to move-equivalence) its decorated permutation, Grassmann necklace and its positroid, respectively, are all bijections. In particular, decorated permutations, Grassmann necklaces and positroids all classify move-equivalence classes of reduced plabic graphs.

\end{theorem}

\subsection{Dimer models}

An \emph{edge weight} on a plabic graph $\Gamma$ is a function 
\[
\wt: E(\Gamma) \rightarrow \Rpos.
\]

We say that two edge weights $\wt_1$ and $\wt_2$ are \emph{gauge equivalent} if there exists a function 
\[
g: V(\Gamma) \rightarrow \Rpos
\]
satisfying $g(u_i^\partial)=1$ for $i \in [n]$ such that 
\[
\wt_1(e) = g(b)^{-1} \wt_2(e) g(w)
\qquad \text{for all } e = bw \in E(\Gamma).
\]
We denote the gauge equivalence class of a weight $\wt$ by $[\wt]$. For any face $f$ of $\Gamma$ delimited by a sequence of edges
    \[
    b_1 \xrightarrow{e_1} w_1 \xrightarrow{e_2} b_2 \xrightarrow{e_3} \cdots \xrightarrow{e_{2k-2}} b_k \xrightarrow{e_{2k-1}} w_k \xrightarrow{e_{2k}} b_1,
    \]
    the function
    \[
    X_f(\wt):= \frac{\wt(e_1) \cdots \wt(e_{2k-1})}{\wt(e_2) \cdots \wt(e_{2k})}
    \]
    is gauge-invariant. These functions are called \emph{$X$-cluster variables}.

    A \emph{dimer model} is a pair $(\Gamma,[\wt])$ where $\Gamma$ is a reduced plabic graph and $[\wt]$ is a gauge equivalence class of edge weights on $\Gamma$. The $X$-cluster variables provide coordinates on the space of edge weights modulo gauge equivalence up to the unique relation $\prod_{f \in F(\Gamma)}X_f=1$. Let 
    \[
    \mathcal X_\Gamma^{>0} \cong \Rpos^{|F(\Gamma)|-1}
    \]
    denote the \emph{space of dimer models} with underlying graph $\Gamma$. Each move $\Gamma \to \Gamma'$ induces a bijection between the corresponding spaces of dimer models (see \cref{fig:moves:X})
\[
 \mathcal X_\Gamma^{>0} \xrightarrow[]{\sim} \mathcal X_{\Gamma'}^{>0}.
\]
Given a positroid $\mathscr{M}$, the \emph{space of dimer models with positroid $\mathscr{M}$} is defined as
\[
\mathcal{X}_{\mathscr{M}}^{>0}
:= 
\left(\bigsqcup_{\Gamma:\mathscr{M}_\Gamma=\mathscr{M} }
\mathcal{X}_\Gamma^{>0}\right)
\!\Big/\! \text{moves}.
\]

\begin{figure}[ht]
\centering

\begin{tikzpicture}[scale=0.6]
\def\ep{0.3}

\begin{scope}[shift={(-17,0)},rotate=135]
  \def\r{2};
  \fill[black!5] (0,0) circle (1.05*\r cm);
  \coordinate[wvert] (n1) at (0:\r);
  \coordinate[wvert] (n2) at (90:\r);
  \coordinate[wvert] (n3) at (180:\r);
  \coordinate[wvert] (n4) at (270:\r);
  \coordinate[bvert] (b1) at (0:0.5*\r);
  \coordinate[bvert] (b2) at (180:0.5*\r);
  \draw[-]
    (n1) -- (b1) -- (n2) -- (b2) -- (n4) -- (b1)
    (b2) -- (n3);
  \node at (45:0.8*\r)  {$X_3$};
  \node at (135:0.8*\r) {$X_4$};
  \node at (225:0.8*\r) {$X_1$};
  \node at (315:0.8*\r) {$X_2$};
  \node at (0, 0) {$X_0$};
\end{scope}

\begin{scope}[shift={(-9.5,0)},rotate=45]
  \def\r{2};
  \fill[black!5] (0,0) circle (1.05*\r cm);
  \coordinate[wvert] (n1) at (0:\r);
  \coordinate[wvert] (n2) at (90:\r);
  \coordinate[wvert] (n3) at (180:\r);
  \coordinate[wvert] (n4) at (270:\r);
  \coordinate[bvert] (b1) at (0:0.5*\r);
  \coordinate[bvert] (b2) at (180:0.5*\r);
  \draw[-]
    (n1) -- (b1) -- (n2) -- (b2) -- (n4) -- (b1)
    (b2) -- (n3);
  \node at (45:\r)    {\scriptsize $X_2(1+X_0^{-1})^{-1}$};
  \node at (135:1.2*\r){\scriptsize $X_3(1+X_0)$};
  \node at (225:1*\r) {\scriptsize $X_4(1+X_0^{-1})^{-1}$};
  \node at (315:1.2*\r){\scriptsize $X_1(1+X_0)$};
  \node at (0, 0)     {\scriptsize $X_0^{-1}$};
\end{scope}

\node at (-14,0) {$\leftrightarrow$};
\node at (-14,-3) {(a)};

\draw[] (-5,-3) -- (-5,3);

\begin{scope}[shift={(-1,0)}, rotate=-45]
  \def\r{2};
  \fill[black!5] (0,0) circle (\r cm);
  \coordinate[wvert] (n1) at (0:\r);
  \coordinate[wvert] (n2) at (90:\r);
  \coordinate[wvert] (n3) at (180:\r);
  \coordinate[wvert] (n4) at (270:\r);
  \coordinate[bvert] (n5) at (0,0);
  \draw[] (n5)--(n1) (n5)--(n2) (n5)--(n3) (n5)--(n4);
  \node at (45:0.8*\r)  {$X_1$};
  \node at (135:0.8*\r) {$X_2$};
  \node at (225:0.8*\r) {$X_3$};
  \node at (315:0.8*\r) {$X_4$};
\end{scope}

\begin{scope}[shift={(5,0)},rotate=-45]
  \def\r{2};
  \fill[black!5] (0,0) circle (\r cm);
  \coordinate[wvert] (n1) at (0:\r);
  \coordinate[wvert] (n2) at (90:\r);
  \coordinate[wvert] (n3) at (180:\r);
  \coordinate[wvert] (n4) at (270:\r);
  \coordinate[wvert] (n5) at (0,0);
  \coordinate[bvert] (b1) at (-0.5,0.5);
  \coordinate[bvert] (b2) at (0.5,-0.5);
  \draw[-]
    (n1)--(b2)--(n4)
    (n2)--(b1)--(n3)
    (b1)--(n5)--(b2);
  \node at (45:0.8*\r)  {$X_1$};
  \node at (135:0.8*\r) {$X_2$};
  \node at (225:0.8*\r) {$X_3$};
  \node at (315:0.8*\r) {$X_4$};

\end{scope}

\node at (2,0) {$\leftrightarrow$};
\node at (2,-3) {(b)};	
\end{tikzpicture}
\caption{Transformation of $X$-cluster variables under (a) spider move and (b) contraction-uncontraction move.}
\label{fig:moves:X}
\end{figure}
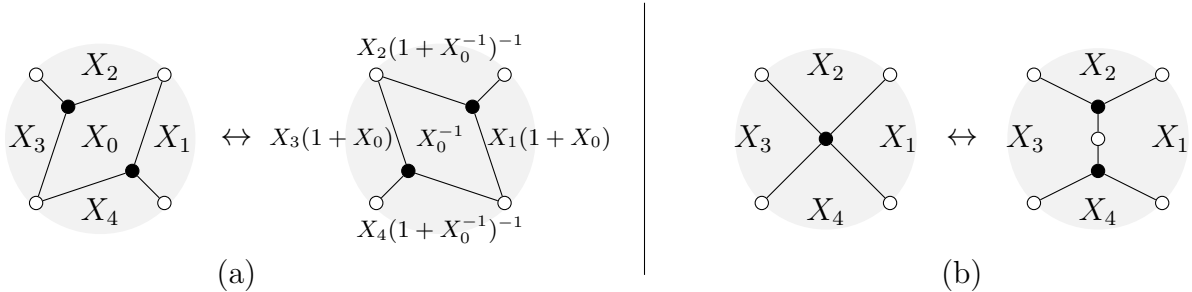

    Let $\wt$ be an edge weight. For $I \in \binom{[n]}{k}$, let 
    \[
    Z_I(\wt):=\sum_{M: \partial M=I } \prod_{e \in M}\wt(e)
    \]
     denote the \emph{partition function} for dimer covers with boundary $I$. The map 
    \[
    {\rm Meas}_\Gamma: \mathcal X_\Gamma^{>0} \rightarrow \mathbb{P}^{\binom{n}{k}-1}, \qquad [\wt] \mapsto (Z_I(\wt))_{I \in \binom{[n]}{k}}
    \]
    is called \emph{boundary measurement}. Here, $\wt$ denotes any edge weight representing the gauge equivalence class $[\wt]$. Different  representatives of $[\wt]$ yield the same vector $(Z_I(\wt))_{I \in \binom{[n]}{k}}$ modulo rescaling, and therefore the point $(Z_I(\wt))_{I \in \binom{[n]}{k}} \in \mathbb{P}^{\binom{n}{k}-1}$ is the same.

\begin{theorem}[Postnikov~\cite{Postnikov2006}] \label{thm:postnikov_boundary_meas}
    Let $\Gamma$ be a reduced plabic graph and let $\mathscr{M}=\mathscr{M}_\Gamma$ be its positroid. Then 
    \[
     {\rm Meas}_\Gamma: \mathcal X_\Gamma^{>0} \xrightarrow[]{\sim} \Pi_{\mathscr{M}}^{>0}
    \]
    is a bijection. Moreover, ${\rm Meas}_\Gamma$ is compatible with moves: for any move $\Gamma \to \Gamma'$, the diagram
\[
\begin{tikzcd}
\mathcal X_\Gamma^{>0}
  \arrow[rr, "\sim"] 
  \arrow[dr, swap, "{\rm Meas}_\Gamma"] 
& 
& 
\mathcal X_{\Gamma'}^{>0}
  \arrow[dl, "{\rm Meas}_{\Gamma'}"] 
\\
& \Pi_{\mathscr{M}}^{>0} &
\end{tikzcd}
\]
commutes. Therefore, the maps $({\rm Meas}_\Gamma)_{\Gamma:\mathscr{M}_\Gamma =\mathscr{M}}$ glue together to a bijection
\[
{\rm Meas}_{\mathscr{M}} : \mathcal X_{\mathscr{M}}^{>0} \xrightarrow[]{\sim} \Pi_{\mathscr{M}}^{>0}.
\]
\end{theorem}

Consider the following action of $\Rpos^n$ on $\mathcal X_\Gamma^{>0}$. For $\bm \lambda = (\lambda_1,\dots,\lambda_n) \in \Rpos^n$,
    \[
    \resizebox{0.98\textwidth}{!}{$
    \bm \lambda \cdot \wt(e) := \begin{cases}
    \lambda_i \wt(e), &\text{if $e$ is the boundary edge incident to $u_i^\partial$ and a white internal vertex},\\
    \lambda_i^{-1} \wt(e), &\text{if $e$ is the boundary edge incident to $u_i^\partial$ and a black internal vertex},\\
    \wt(e),&\text{otherwise}. 
    \end{cases}$}
    \]
The action is compatible with moves, and therefore gives an action of $\Rpos^n$ on $\mathcal X_\mathscr{M}^{>0}$. Moreover, the definition is such that ${\rm Meas}_\Gamma$ and ${\rm Meas}_\mathscr{M} $ are equivariant.

  \subsection{\texorpdfstring{$A$}{A}-networks}

    An \emph{$A$-network} is a pair $(\Gamma,A)$ where $\Gamma$ is a reduced plabic graph and 
    \(
    A: F(\Gamma) \rightarrow \Rpos
    \)
    is a function defined modulo rescaling by $\Rpos$. Let 
    \[
    \mathcal A_\Gamma^{>0} \cong \Rpos^{|F(\Gamma)|-1}
    \]
    denote the \emph{space of $A$-networks with underlying plabic graph $\Gamma$}.

\begin{figure}[ht]
\centering
\begin{tikzpicture}[scale=0.6]
    \def\ep{0.3}
    \begin{scope}[shift={(-10,0)},rotate=135]
      \def\r{2};
      \fill[black!5] (0,0) circle (1.05*\r cm);
      \coordinate[wvert] (n1) at (0:\r);
      \coordinate[wvert] (n2) at (90:\r);
      \coordinate[wvert] (n3) at (180:\r);
      \coordinate[wvert] (n4) at (270:\r);
      \coordinate[bvert] (b1) at (0:0.5*\r);
      \coordinate[bvert] (b2) at (180:0.5*\r);
      \draw[-]
        (n1) -- (b1) -- (n2) -- (b2) -- (n4) -- (b1)
        (b2) -- (n3);
      \node at (45:0.8*\r)  {$A_{12}$};
      \node at (135:0.8*\r) {$A_{14}$};
      \node at (225:0.8*\r) {$A_{23}$};
      \node at (315:0.8*\r) {$A_{34}$};
      \node at (0, 0) {$A_{13}$};
    \end{scope}

    \begin{scope}[shift={(-3,0)},rotate=45]
      \def\r{2};
      \fill[black!5] (0,0) circle (1.05*\r cm);
      \coordinate[wvert] (n1) at (0:\r);
      \coordinate[wvert] (n2) at (90:\r);
      \coordinate[wvert] (n3) at (180:\r);
      \coordinate[wvert] (n4) at (270:\r);
      \coordinate[bvert] (b1) at (0:0.5*\r);
      \coordinate[bvert] (b2) at (180:0.5*\r);
      \draw[-]
        (n1) -- (b1) -- (n2) -- (b2) -- (n4) -- (b1)
        (b2) -- (n3);
      \node at (45:\r)    { $A_{34}$};
      \node at (135:1.2*\r){ $A_{12}$};
      \node at (225:1*\r) { $A_{14}$};
      \node at (315:1.2*\r){ $A_{23}$};
      \node at (0, 0)     {$A_{24}$};
    
      \node at (-3,-2.8) {$A_{24}:=\frac{A_{12} A_{34}+A_{14} A_{23}}{A_{13}}$};
    \end{scope}

    \node at (-7,0) {$\leftrightarrow$};
    \node at (-7,-4) {(a)};

    \draw[] (1,-3) -- (1,3);

    \begin{scope}[shift={(5,0)}, rotate=-45]
      \def\r{2};
      \fill[black!5] (0,0) circle (\r cm);
      \coordinate[wvert] (n1) at (0:\r);
      \coordinate[wvert] (n2) at (90:\r);
      \coordinate[wvert] (n3) at (180:\r);
      \coordinate[wvert] (n4) at (270:\r);
      \coordinate[bvert] (n5) at (0,0);
      \draw[] (n5)--(n1) (n5)--(n2) (n5)--(n3) (n5)--(n4);
      \node at (45:0.8*\r)  {$A_1$};
      \node at (135:0.8*\r) {$A_2$};
      \node at (225:0.8*\r) {$A_3$};
      \node at (315:0.8*\r) {$A_4$};
    \end{scope}
    
    \begin{scope}[shift={(12,0)},rotate=-45]
      \def\r{2};
      \fill[black!5] (0,0) circle (\r cm);
      \coordinate[wvert] (n1) at (0:\r);
      \coordinate[wvert] (n2) at (90:\r);
      \coordinate[wvert] (n3) at (180:\r);
      \coordinate[wvert] (n4) at (270:\r);
      \coordinate[wvert] (n5) at (0,0);
      \coordinate[bvert] (b1) at (-0.5,0.5);
      \coordinate[bvert] (b2) at (0.5,-0.5);
      \draw[-]
        (n1)--(b2)--(n4)
        (n2)--(b1)--(n3)
        (b1)--(n5)--(b2);
      \node at (45:0.8*\r)  {$A_1$};
      \node at (135:0.8*\r) {$A_2$};
      \node at (225:0.8*\r) {$A_3$};
      \node at (315:0.8*\r) {$A_4$};
    
    \end{scope}

    \node at (8.5,0) {$\leftrightarrow$};
    \node at (9,-4) {(b)};	
    
\end{tikzpicture}
\caption{Transformation of $A$-cluster variables under (a) spider move and (b) contraction-uncontraction move.}
\label{fig:moves:A}
\end{figure}
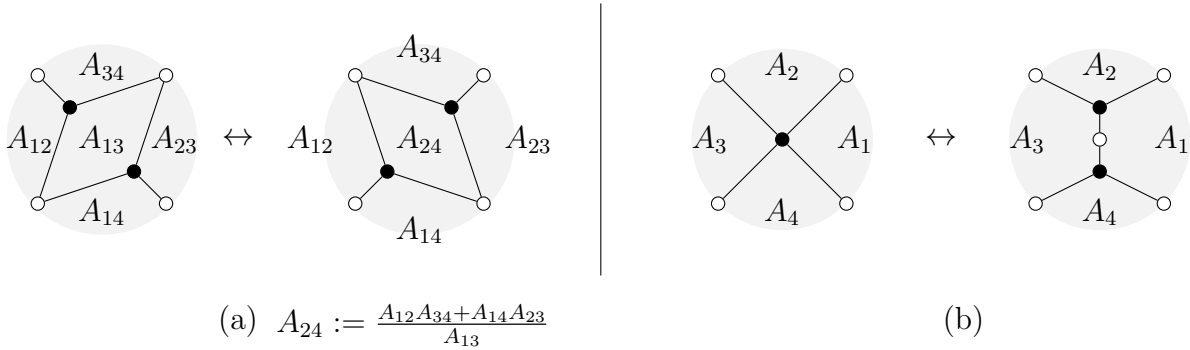

\begin{figure}[ht]
\begin{tikzpicture}[scale=0.6]
\def\ep{0.3}

\begin{scope}[shift={(-3,0)},rotate=135]
  \def\r{2};
  \fill[black!5] (0,0) circle (1.05*\r cm);
  \coordinate[wvert] (n1) at (0:\r);
  \coordinate[wvert] (n2) at (90:\r);
  \coordinate[wvert] (n3) at (180:\r);
  \coordinate[wvert] (n4) at (270:\r);
  \coordinate[bvert] (b1) at (0:0.5*\r);
  \coordinate[bvert] (b2) at (180:0.5*\r);
  \draw[-]
    (n1) -- (b1) -- (n2) -- (b2) -- (n4) -- (b1)
    (b2) -- (n3);
    

  \node at (45:0.8*\r)  {$Skl$};
  \node at (135:0.8*\r)  {$Sjk$};
  \node at (225:0.8*\r)  {$Sij$};
  \node at (315:0.8*\r)  {$Sil$};
  \node at (0, 0) {$Sik$};

\end{scope}

\begin{scope}[shift={(4,0)},rotate=45]
  \def\r{2};
  \fill[black!5] (0,0) circle (1.05*\r cm);
  \coordinate[wvert] (n1) at (0:\r);
  \coordinate[wvert] (n2) at (90:\r);
  \coordinate[wvert] (n3) at (180:\r);
  \coordinate[wvert] (n4) at (270:\r);
  \coordinate[bvert] (b1) at (0:0.5*\r);
  \coordinate[bvert] (b2) at (180:0.5*\r);
  \draw[-]
    (n1) -- (b1) -- (n2) -- (b2) -- (n4) -- (b1)
    (b2) -- (n3);

  \node at (45:0.8*\r)  {$Sil$};
  \node at (135:0.8*\r)  {$Skl$};
  \node at (225:0.8*\r)  {$Sjk$};
  \node at (315:0.8*\r)  {$Sij$};
  \node at (0, 0) {$Sjl$};
\end{scope}

\draw[] (7,-2) -- (7,2);
\node at (0.5,0) {$\leftrightarrow$};		
\node at (0.5,-3) {(a)};	
\node at (13.5,-3) {(b)};	

\begin{scope}[shift={(10,0)}, rotate=-45]
  \def\r{2};
  \fill[black!5] (0,0) circle (\r cm);
  \coordinate[wvert] (n1) at (0:\r);
  \coordinate[wvert] (n2) at (90:\r);
  \coordinate[wvert] (n3) at (180:\r);
  \coordinate[wvert] (n4) at (270:\r);
  \coordinate[bvert] (n5) at (0,0);

  \draw[] (n5) -- (n1)
          (n5) -- (n2)
          (n5) -- (n3)
          (n5) -- (n4);

\node at (45:0.8*\r)  {$Sk$};
  \node at (135:0.8*\r)  {$Sj$};
  \node at (225:0.8*\r)  {$Sl$};
  \node at (315:0.8*\r)  {$Si$};
  
\end{scope}

\node at (13.5,0) {$\leftrightarrow$};

\begin{scope}[shift={(17,0)},rotate=-45]
  \def\r{2};
  \fill[black!5] (0,0) circle (\r cm);
  \coordinate[wvert] (n1) at (0:\r);
  \coordinate[wvert] (n2) at (90:\r);
  \coordinate[wvert] (n3) at (180:\r);
  \coordinate[wvert] (n4) at (270:\r);
  \coordinate[wvert] (n5) at (0,0);

  \coordinate[bvert] (b1) at (-0.5,0.5);
  \coordinate[bvert] (b2) at (0.5,-0.5);

  \draw[-]
    (n1) -- (b2) -- (n4)
    (n2) -- (b1) -- (n3)
    (b1) -- (n5) -- (b2);

\node at (45:0.8*\r)  {$Sk$};
  \node at (135:0.8*\r)  {$Sj$};
  \node at (225:0.8*\r)  {$Sl$};
  \node at (315:0.8*\r)  {$Si$};

\end{scope}  

\end{tikzpicture}
\caption{(a) Spider move and (b) contraction-uncontraction move with face labels.}
\label{fig:moves_labels_plabic}
\end{figure}
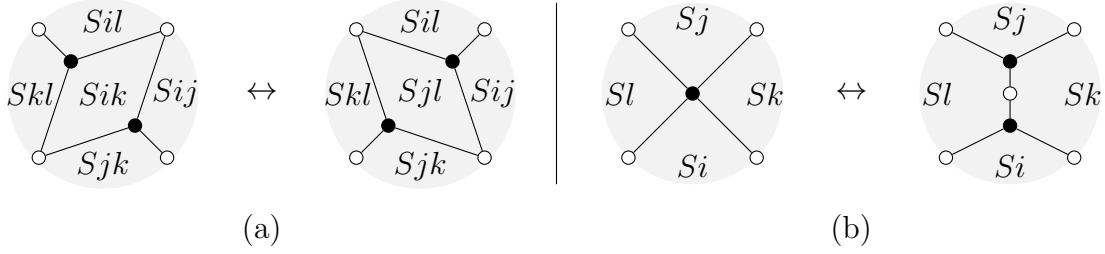

Each move $\Gamma \to \Gamma'$ induces a bijection between the corresponding spaces of $A$-networks
\[
\mathcal A_\Gamma^{>0} \xrightarrow[]{\sim} \mathcal A_{\Gamma'}^{>0}
\]
as shown in \cref{fig:moves:A}. Given a positroid $\mathscr{M}$, we define the \emph{space of $A$-networks with positroid $\mathscr{M}$} to be
\[
 \mathcal{A}_{\mathscr{M}}^{>0}
:= 
\left(\bigsqcup_{\Gamma:\mathscr{M}_\Gamma=\mathscr{M} }
\mathcal{A}_\Gamma^{>0}\right)
\!\Big/\! \text{moves}.
\]

Recall that Pl\"ucker coordinates are defined modulo rescaling. For any reduced plabic graph $\Gamma$ with $\mathscr{M}_\Gamma = \mathscr{M}$, we define a map
    \[
    \bm F_\Gamma: \Pi^{>0}_{\mathscr{M}} \rightarrow \mathcal A_\Gamma^{>0}, \qquad V \mapsto (\Delta_{\bm T(f)}(V))_{f \in F(\Gamma)}
    \]
assigning to each point in the positroid cell $\Pi^{>0}_{\mathscr{M}}$ an $A$-network.

The following is due to Scott~\cite{scott} for the uniform positroid $\mathscr{M}=\binom{[n]}{k}$ and Muller--Speyer~\cite{MullerSpeyer} in general.
\begin{theorem}\label{thm:scott_MS}
    For any reduced plabic graph $\Gamma$ with $\mathscr{M}_\Gamma = \mathscr{M}$, the map
    \[
    \bm F_\Gamma: \Pi^{>0}_{\mathscr{M}} \xrightarrow[]{\sim} \mathcal A_\Gamma^{>0}
    \]
    is a bijection. Moreover, $\bm F_\Gamma$ is compatible with moves in the following sense. For any move $\Gamma \to \Gamma'$, the diagram
\begin{equation}
    \begin{tikzcd}
\mathcal A_\Gamma^{>0}
    \arrow[rr, "\sim"]
&&
\mathcal A_{\Gamma'}^{>0}\\
& \Pi_{\mathscr{M}}^{>0}
    \arrow[ul, swap, "\bm F_\Gamma"]
    \arrow[ur, "\bm F_{\Gamma'}"]
& \\
\end{tikzcd}. \label{eq:comm:moves_A}
\end{equation}
commutes. Therefore, the maps $(\bm F_\Gamma)_{\Gamma: \mathscr{M}_\Gamma=\mathscr{M}}$ glue together to a bijection
\[
\bm F_{\mathscr{M}} : \Pi_{\mathscr{M}}^{>0} \xrightarrow[]{\sim} \mathcal A_{\mathscr{M}}^{>0}.
\]
\end{theorem}

The target face labels change under moves as shown in \cref{fig:moves_labels_plabic}. Therefore, the commutativity of \eqref{eq:comm:moves_A} is trivial for the contraction-uncontraction move and is a consequence of the three-term Pl\"ucker relation
\[
\Delta_{Sik} \Delta_{Sjl} = \Delta_{Sij} \Delta_{Skl}+\Delta_{Sil}\Delta_{Sjk}
\]
for the spider move.

The group $\Rpos^n$ acts on $\mathcal A_\Gamma^{>0}$ by 
\[
\bm \lambda \cdot A = \left( \left(\prod_{i \in \mathbf{T}(f)} \lambda_i \right) A_f \right)_{f \in F(\Gamma)}
\]
making the map $\bm F_\Gamma$ equivariant. The action is compatible with moves, and therefore the same is true for $\bm F_\mathscr{M}$.

\subsection{The twist maps}

Let $\mathscr{M}$ be a positroid, let $\bm I_{\mathscr{M}} = (I_1,\dots,I_n)$ be its Grassmann necklace, and let $\Pi_{\mathscr{M}}^{>0}$ denote the corresponding positroid cell.  
The \emph{right twist} $\rt$ is the map
\[
  \rt : \Pi_{\mathscr{M}}^{>0} \rightarrow \Pi_{\mathscr{M}}^{>0}
\]
defined as follows.  
Let $M$ be a $k \times n$ matrix representative of $V \in \Pi_{\mathscr{M}}$, and denote by $M_i$ the $i$-th column of $M$.  
Define a new matrix $\rt(M)$ by the conditions
\begin{equation}
      \langle \rt(M)_i, M_i \rangle = 1, 
  \qquad 
  \langle \rt(M)_i, M_j \rangle = 0 
  \quad \text{for } j \in I_i \setminus \{i\}. \label{eq:def_right_twist}
\end{equation}
Here $\langle \cdot , \cdot \rangle$ is the standard inner product on {$\R^k$}.
Finally, we set
\[
  \rt(V) := \operatorname{rowspan}\!\big(\rt(M)\big).
\]
The \emph{left twist} $\lt$ is defined similarly but using the reverse Grassmann necklace instead in \eqref{eq:def_right_twist}.

\begin{lemma}[Muller--Speyer~\cite{MullerSpeyer}]  \label{lem:MS}
The right twist has the following properties:
\begin{enumerate}
    \item For any $\bm \lambda \in \Rpos^n,$ $\rt(\bm \lambda \cdot V) = \bm \lambda^{-1} \cdot \rt(V)$. 
    \item For any boundary face $f_{i,i+1}^\partial$, $ \Delta_{\bm T(f_{i,i+1}^\partial)}(\rt(V))=\Delta_{\bm T(f_{i,i+1}^\partial)}(V)^{-1}$.
\end{enumerate}
\end{lemma}

Let $\mathscr{M}$ be a positroid. There is a canonical map 
\[
p_\mathscr{M}: \mathcal A_{\mathscr{M}}^{>0} \rightarrow \mathcal X_{\mathscr{M}}^{>0},
\]
called the \emph{cluster ensemble map}, defined as follows. Let $\Gamma$ be a reduced plabic graph such that $\mathscr{M}_\Gamma = \mathscr{M}$. Let $A \in \mathcal A_{\Gamma}^{>0}$, let
\[
\wt(e) :=\begin{cases} 1/{A_f A_g}, &\text{if $e$ is internal and adjacent to faces $f$ and $g$},\\
1/{A_{f_{i-1,i}^\partial}}, &\text{if $e$ is adjacent to $u_i^\partial$ and a white internal vertex},\\
1/{A_{f_{i,i+1}^\partial}}, &\text{if $e$ is adjacent to $u_i^\partial$ and a black internal vertex},\\
 \end{cases} 
\]
and define $p_\Gamma(A):=[\wt]$. If $\Gamma \to \Gamma'$ is a move, then the following diagram commutes:
\[
\begin{tikzcd}
\mathcal{A}_{\Gamma}^{>0}
  \arrow[r, "\sim"]
  \arrow[d, "p_{\Gamma}"']
& 
\mathcal{A}_{\Gamma'}^{>0}
  \arrow[d, "p_{\Gamma'}"]
\\
\mathcal{X}_{\Gamma}^{>0}
  \arrow[r, "\sim"]
&
\mathcal{X}_{\Gamma'}^{>0}
\end{tikzcd}.
\]
Thus, the maps $(p_\Gamma)_{\Gamma: \mathscr{M}_\Gamma = \mathscr{M}}$ glue together to give $p_{\mathscr{M}}$.

\begin{theorem}[Muller--Speyer~\cite{MullerSpeyer}] \label{thm:Muller_Speyer}
The right twist is a bijection whose inverse is the left twist. They fit into the following commutative diagram:
\[
\begin{tikzcd}
\mathcal A_{\mathscr{M}}^{>0} 
  \arrow[r, "p_{\mathscr{M}}"] 
& 
\mathcal X_{\mathscr{M}}^{>0} 
  \arrow[d, "{\rm Meas}_{\mathscr{M}}"] 
\\
\Pi_{\mathscr{M}}^{>0}
  \arrow[u, "\bm F_{\mathscr{M}}"]   \arrow[r,  bend right=15, swap, "\rt"]
& 
\Pi_{\mathscr{M}}^{>0}
  \arrow[l,  bend right=15, swap, "\lt"]
\end{tikzcd}.
\]
In particular, the boundary measurement map can be inverted as \[{\rm Meas}_{\mathscr{M}}^{-1}=p_{\mathscr{M}} \circ \bm F_{\mathscr{M}} \circ \lt.\]
\end{theorem}

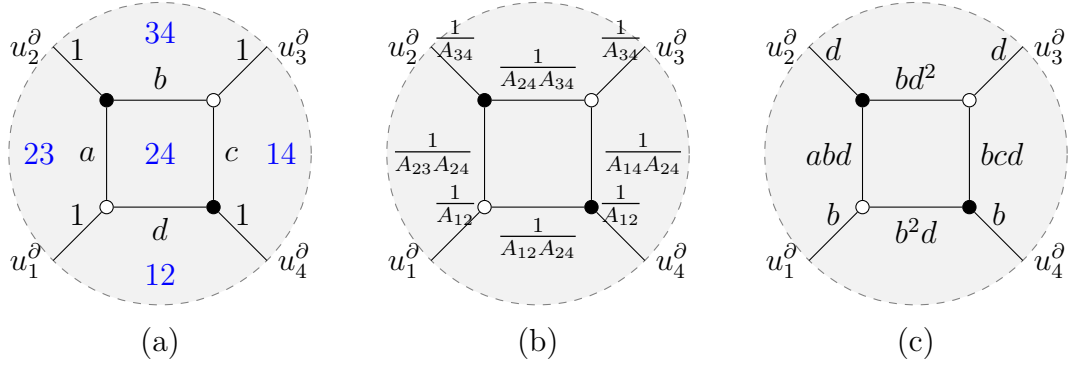
\begin{figure}[ht]
    \centering
    \begin{tikzpicture}
        \begin{scope}[shift={(0,0)}]
            \def\r{2};
            \fill[black!5] (0,0) circle (1*\r cm);
                \draw[dashed, gray] (0,0) circle (1*\r cm);
            \coordinate[wvert] (w1) at (225:0.5*\r);
            \coordinate[wvert] (w2) at (45:0.5*\r);
            \coordinate[bvert] (b1) at (135:0.5*\r);
            \coordinate[bvert] (b2) at (315:0.5*\r);

            \coordinate[label=left: $u_1^\partial$] (b1') at (225:\r);
            \coordinate[label=right: $u_3^\partial$] (b2') at (45:\r);
            \coordinate[label=left: $u_2^\partial$] (w1') at (135:\r);
            \coordinate[label=right: $u_4^\partial$] (w2') at (315:\r);

            \draw[]
            (w1) -- node[left]{$a$} (b1)
                 -- node[above]{$b$} (w2)
                 -- node[right]{$c$} (b2)
                 -- node[below]{$d$} (w1)
            (w1) -- node[above]{$1$} (b1')
            (w2) -- node[above]{$1$} (b2')
            (b1) -- node[above]{$1$} (w1')
            (b2) -- node[above]{$1$} (w2');

            \node[blue] at (0,0) {$24$};
            \node[blue] at (0:0.8*\r) {$14$};
            \node[blue] at (90:0.8*\r) {$34$};
            \node[blue] at (180:0.8*\r) {$23$};
            \node[blue] at (270:0.8*\r) {$12$};

            \node at (0,-2.5) {(a)};
        \end{scope}

        \begin{scope}[shift={(5,0)}]
            \def\r{2};
            \fill[black!5] (0,0) circle (1*\r cm);
            \draw[dashed, gray] (0,0) circle (1*\r cm);
            \coordinate[wvert] (w1) at (225:0.5*\r);
            \coordinate[wvert] (w2) at (45:0.5*\r);
            \coordinate[bvert] (b1) at (135:0.5*\r);
            \coordinate[bvert] (b2) at (315:0.5*\r);

            \coordinate[label=left: $u_1^\partial$] (b1') at (225:\r);
            \coordinate[label=right: $u_3^\partial$] (b2') at (45:\r);
            \coordinate[label=left: $u_2^\partial$] (w1') at (135:\r);
            \coordinate[label=right: $u_4^\partial$] (w2') at (315:\r);

            \draw[]
            (w1) -- node[left]{$\frac{1}{A_{23}A_{24}}$} (b1)
                 -- node[above]{$\frac{1}{A_{24}A_{34}}$} (w2)
                 -- node[right]{$\frac{1}{A_{14}A_{24}}$} (b2)
                 -- node[below]{$\frac{1}{A_{12}A_{24}}$} (w1)
            (w1) -- node[above]{$\frac{1}{A_{12}}$} (b1')
            (w2) -- node[above]{$\frac{1}{A_{34}}$} (b2')
            (b1) -- node[above]{$\frac{1}{A_{34}}$} (w1')
            (b2) -- node[above]{$\frac{1}{A_{12}}$} (w2');

            \node at (0,-2.5) {(b)};
        \end{scope}

        \begin{scope}[shift={(10,0)}]
            \def\r{2};
            \fill[black!5] (0,0) circle (1*\r cm);
                \draw[dashed, gray] (0,0) circle (1*\r cm);
            \coordinate[wvert] (w1) at (225:0.5*\r);
            \coordinate[wvert] (w2) at (45:0.5*\r);
            \coordinate[bvert] (b1) at (135:0.5*\r);
            \coordinate[bvert] (b2) at (315:0.5*\r);

            \coordinate[label=left: $u_1^\partial$] (b1') at (225:\r);
            \coordinate[label=right: $u_3^\partial$] (b2') at (45:\r);
            \coordinate[label=left: $u_2^\partial$] (w1') at (135:\r);
            \coordinate[label=right: $u_4^\partial$] (w2') at (315:\r);

            \draw[]
            (w1) -- node[left]{$abd$} (b1)
                 -- node[above]{$bd^2$} (w2)
                 -- node[right]{$bcd$} (b2)
                 -- node[below]{$b^2d$} (w1)
            (w1) -- node[above]{$b$} (b1')
            (w2) -- node[above]{$d$} (b2')
            (b1) -- node[above]{$d$} (w1')
            (b2) -- node[above]{$b$} (w2');

            \node at (0,-2.5) {(c)};
        \end{scope}
    \end{tikzpicture}
    \caption{Inverting boundary measurement via the left twist after cyclically rotating the boundary labels.}
    \label{fig:example_MS_twist_rotated}
\end{figure}

\begin{example}\label{ex:MS_twist_rotated}
Consider the reduced plabic graph $\Gamma$ in \cref{fig:example_MS_twist_rotated}(a). Its boundary measurement is
\[
\Delta_{12} = b, \quad
\Delta_{13} = 1,\quad
\Delta_{14} = c,\quad
\Delta_{23} = a,\quad
\Delta_{24} = ac+bd,\quad
\Delta_{34} = d.
\]
Since all Pl\"ucker coordinates are positive, the positroid is $\mathscr{M} = \binom{[4]}{2}$ and the positroid cell is $\Pi_{\mathscr{M}}^{>0} = \Gr^{>0}_{2,4}$. A matrix representative is given by
\[
M =
\begin{bmatrix}
1 & a & 0 & -d \\
0 & b & 1 & c
\end{bmatrix}.
\]
The reverse Grassmann necklace is
\[
\bm I^{\text{rev}}_\Gamma = (14,12,23,34),
\]
using which the left twist is computed to be
\[
\lt(M) =
\begin{bmatrix}
1 & 0 & -\frac{b}{a} & -\frac{1}{d} \\
\frac{d}{c} & \frac{1}{b} & 1 & 0
\end{bmatrix}.
\]
The twisted Pl\"ucker coordinates are given by
\begin{equation}
    A_{12} = \frac 1 b,\quad
    A_{13} = \frac{ac+bd}{ac},\quad
    A_{14} = \frac 1 c,\quad
    A_{23} = \frac 1 a,\quad
    A_{24} = \frac{1}{bd},\quad
    A_{34} = \frac 1 d.
    \label{eq:eg_MS_twist_rotated}
\end{equation}

The target face labels are shown in \cref{fig:example_MS_twist_rotated}(a), using which the map $p_\Gamma$ is computed as in \cref{fig:example_MS_twist_rotated}(b). Plugging in \eqref{eq:eg_MS_twist_rotated}, we get the edge weights shown in \cref{fig:example_MS_twist_rotated}(c), which are gauge equivalent to the edge weights in \cref{fig:example_MS_twist_rotated}(a).
\end{example}

\section{The positive orthogonal Grassmannian and the Ising model}\label{sec:3}
  Let $n \geq 1$ be a positive integer. We consider the following quadratic form on $\R^{2n}$ 
  \[
    q \colon \R^{2n} \xrightarrow[]{} \R, \quad \bm x = (x_1,\dots,x_{2n}) \mapsto \sum_{i = 1}^{2n} (-1)^{i-1} \ x_{i}^2.
  \]
   We denote by $Q := {\rm diag}(1, -1, 1, \dots 1, -1) \in \R^{2n \times 2n}$ the matrix that represents $q$ in the standard basis. The \emph{orthogonal Grassmannian} $\OGr_{n,2n}$ is the {projective subvariety of the Grassmannian $\Gr_{n,2n}$ defined as follows
   \[
        \OGr_{n, 2n} := \left\{ V \in \Gr_{n, 2n} \colon q(x) = 0 \quad \text{for all } x \in V \right\}.
   \]
   }
    It is well known that this variety has two isomorphic connected components that are irreducible {(see for example \cite[Proposition 2.4]{EM25})}. {In Pl\"ucker coordinates, these two components are cut out in $\Gr_{n,2n}$ by linear equations as follows}
    \begin{align*}
        \OGr_{n,2n}^{\pm} := \left\{ V \in \Gr_{n,2n} \colon \Delta_I(V) = \pm \Delta_{[2n] \setminus I}(V) \quad \text{for all } I \in \binom{[2n]}{n}  \right\},
    \end{align*}
    {and we have 
    \[
        \OGr_{n,2n} = \OGr_{n,2n}^{+} \sqcup \OGr_{n,2n}^{-}.
    \]
    }
    We define
    \[
    \OGr_{n,2n}^{ \geq 0} := \OGr_{n,2n} \cap \Gr_{n,2n}^{\geq 0} \quad \text{and} \quad  \OGr_{n,2n}^{ > 0} := \OGr_{n,2n} \cap \Gr_{n,2n}^{> 0} \ .
    \]
    Note that we {necessarily} have $\OGr_{n,2n}^{ \geq 0} \subset \OGr_{n,2n}^{+}$. Intersecting $\OGr_{n,2n}^{ \geq 0}$ with the positroid stratification~\eqref{eq:positroidstratification}, we get an induced stratification
\[
\OGr_{n,2n}^{\geq 0} = \bigsqcup_{\mathscr{M}} \OGr^{\geq 0}_{n,2n} \cap \Pi_{\mathscr{M}}^{>0}.
\]

 Let us denote by $P_n$ the set of fixed-point-free involutions on $[2n]$. Given $\pi \in P_n$, viewing it as a decorated permutation, we have an associated positroid $\mathscr{M}_\pi$ via the bijections in \cref{thm:post_bijections}.

\begin{theorem}[Galashin--Pylyavskyy~\cite{GalashinPylyavskyy}]
The positroid cells that have nonempty intersection with $\OGr_{n,2n}^{\geq 0}$ are precisely those of the form $\Pi_{\mathscr{M}_\pi}^{>0}$ for a fixed-point-free involution $\pi:[2n] \to [2n]$. Thus,
\[
\OGr_{n,2n}^{\geq 0} = \bigsqcup_{\pi \in P_n } (\OGr^{\geq 0}_{n,2n} \cap \Pi_{\mathscr{M}_\pi}^{>0}).
\]
\end{theorem}

\subsection{Cactus graphs}\label{sec:cactus_graphs}

Let \( \disk \) be a disk with boundary vertices labeled \( b_1^\partial, \dots, b_n^\partial \) in clockwise cyclic order. Let $P$ be a noncrossing partition of $[n]$. A \emph{cactus with shape $P$} is the topological space $\disk/P$ obtained by gluing together the points $b_1^\partial, \dots, b_n^\partial$ whose indices are in the same part of $P$. A \emph{cactus graph with shape $P$} is a graph $G$ embedded in a cactus $\disk/P$ with boundary vertices \( b_1^\partial, \dots, b_n^\partial \).

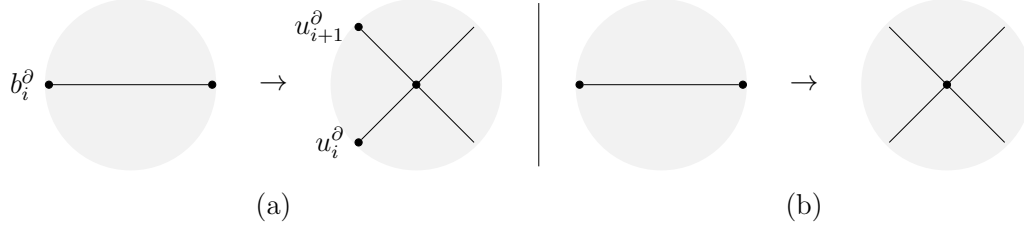
\begin{figure}[ht]
    \centering
    \scalebox{0.9}{
    \begin{tikzpicture}[scale=0.6]
        \begin{scope}[shift={(-3,0)}]
            \def\r{2};
            \fill[black!5] (0,0) circle (1.05*\r cm);
            \coordinate[nvert] (n1) at (2,0);
            \coordinate[nvert, label = left: $b_i^\partial$] (n2) at (-2,0);
            \draw[] (n1) -- (n2);
        \end{scope}
        
        \node at (0.5,0) {$\rightarrow$};	
        
         \begin{scope}[shift={(4,0)}]
            \def\r{2};
            \fill[black!5] (0,0) circle (1.05*\r cm);
            \coordinate[nvert, label = left: $u_{i+1}^\partial$] (n1) at (135:\r);
            \coordinate[nvert, label = left: $u_i^\partial$] (n2) at (225:\r);
            \node[nvert] at (0,0) {};
            \draw[] (n1) -- (-45:\r)
                    (n2) -- (45:\r);
        \end{scope}

        \begin{scope}[shift={(10,0)}]
            \def\r{2};
            \fill[black!5] (0,0) circle (1.05*\r cm);
            \coordinate[nvert] (n1) at (2,0);
            \coordinate[nvert] (n2) at (-2,0);
            \draw[] (n1) -- (n2);
        \end{scope}
        
        \node at (13.5,0) {$\rightarrow$};	
        
        \begin{scope}[shift={(17,0)}]
          \def\r{2};
            \fill[black!5] (0,0) circle (1.05*\r cm);
            \coordinate[] (n1) at (135:\r);
            \coordinate[] (n2) at (225:\r);
            \node[nvert] at (0,0) {};
            \draw[] (n1) -- (-45:\r)
                    (n2) -- (45:\r);
        \end{scope}

        \draw[] (7,-2) -- (7,2);
        \node at (0.5,-3) {(a)};	
        \node at (13.5,-3) {(b)};	
    \end{tikzpicture}}
    \caption{The construction of the medial graph \( G^\times \) from \( G \) at (a) boundary edges and (b) internal edges.}
    \label{fig:medial_graph}
\end{figure}

 Let $G$ be a cactus graph with shape $P$. The \emph{medial graph} $G^\times$ of \( G \) is the graph in $\disk$ obtained from \( G \) as follows. Place boundary vertices \( u_1^\partial,\dots,u_{2n}^\partial \) of \( G^\times \) around the boundary of the cactus $\disk/P$ so that \( b_i^\partial \) is between \( u_{2i-1}^\partial \) and \( u_{2i}^\partial \). Replace each edge of $G$ as in \cref{fig:medial_graph}. Finally, take the preimage of the resulting graph under the quotient map $\disk \rightarrow \disk/P$. By construction, each internal vertex of \(G^\times\) has degree \(4\). A \emph{strand} of \( G \) is a maximal (unoriented) path in $G^\times$ that contains the opposite pair of edges at each internal vertex. The \emph{medial pairing} of \( G \) is the fixed-point-free involution \( \pi_G : [2n] \rightarrow [2n] \) given by 
\[
\pi_G(i) = j \quad \text{if there is a strand between \( u_i^\partial\) and \(u_j^\partial\)}.
\]
If \( \pi_G(i) = j \), then we write \( j = \bar i\). We denote the strand between $i$ and $\bar i$ by $\gamma_{i,\bar i}$. We say that a cactus graph $G$ is \emph{reduced} if the following hold:
\begin{enumerate}
    \item No strand is a closed loop.
    \item No pair of strands intersects twice.
\end{enumerate}

    The \emph{Y-$\Delta$ move} is a local modification of cactus graphs shown in \cref{figyd}. We say that two cactus graphs $G$ and $G'$ are \emph{move-equivalent} if they are connected by a sequence of Y-$\Delta$ moves. Similarly to \cref{thm:post_bijections}, we have:
    \begin{proposition}
    The map
    \[
     \Big\{\text{reduced cactus graphs}\Big\}/\text{move-equivalence} \xrightarrow[] {G \mapsto \pi_G} \Big\{\text{fixed-point-free involutions of $[2n]$}\Big\}
    \]
    is well-defined and a bijection.
    \end{proposition}
    In other words, move-equivalence classes of reduced cactus graphs with $n$ boundary vertices are classified by fixed-point-free involutions of $[2n]$.

\subsection{Ising models}

An \emph{Ising model} is a pair $(G,J)$ where $G$ is a reduced cactus graph and $J: E(G) \rightarrow \Rpos$ is a function called \emph{coupling constant}. Let 
\[
\mathcal I_G^{>0} \cong \Rpos^{E(G)}
\]
be the space of Ising models with underlying graph $G$. 
Each Y-$\Delta$ move $G \to G'$ induces a bijection between the corresponding spaces of Ising models
\[
 \mathcal I_G^{>0} \xrightarrow[]{\sim} \mathcal I_{G'}^{>0}
\]
as illustrated in \cref{figyd}. The edge parameters shown in the figure are $\exp(2J_e)$ and the bijection is given in terms of these parameters by
\[
	A=\sqrt{ \frac{(abc+1)(a+bc)}{(b+ac)(c+ab)}},\quad
	B=\sqrt{ \frac{(abc+1)(b+ac)}{(a+bc)(c+ab)}},\quad
	C=\sqrt{ \frac{(abc+1)(c+ab)}{(a+bc)(b+ac)}}.
\]

Let $\pi$ be a fixed-point-free involution on $[2n]$. The \emph{space of Ising models associated to $\pi$} is defined as 
\[
\mathcal{I}_{\pi}^{>0}
:= 
\bigsqcup_{G: \pi_G=\pi }
\mathcal{I}_G^{>0} 
\!\Big/\! \text{Y-$\Delta$ moves}.
\]

A \emph{spin configuration} is a function 
\(
\sigma : V \rightarrow \{ -1,1\}
\)
assigning to each vertex $v \in V$ a spin $\sigma_v$. The probability of a spin configuration is defined as
\[
\mathbf{P}(\sigma):= \frac{1}{Z} \prod_{e = uv \in E} e^{J_e \sigma_u \sigma_v},
 \]
where 
\(
Z := \sum_{\sigma: V \rightarrow \{-1,1\} } \prod_{e = uv \in E} e^{J_e \sigma_u \sigma_v} 
\) is the \emph{partition function}. Given $i,j \in [n]$, we define the \emph{boundary correlation function} by 
\[
\langle \sigma_i \sigma_j \rangle : = \sum_{\sigma : V \rightarrow \{-1,1\}} \mathbf{P}(\sigma) \sigma_{b_i^\partial} \sigma_{b_j^\partial}. 
\]
By definition, we have
\(
\langle \sigma_i \sigma_j \rangle =\langle \sigma_j \sigma_i \rangle\) and \( \langle \sigma_i \sigma_i \rangle =1
\) for all $i,j \in [n]$. Thus, the $n \times n$-matrix
\[
M(G,J):=(\langle \sigma_i \sigma_j \rangle)_{ i,j \in [n]}
\]
of boundary correlation functions is a symmetric real matrix with $1$s on the diagonal. We denote the space of such matrices by $\operatorname{Mat}_n^{\text{sym}}(\R,1)$. The map 
    \[
    \operatorname{Corr}_G: \mathcal I_G^{>0} \rightarrow \operatorname{Mat}_n^{\text{sym}}(\R,1),  \quad J \mapsto  M(G,J)
    \]
    is called the \emph{boundary correlation map}.

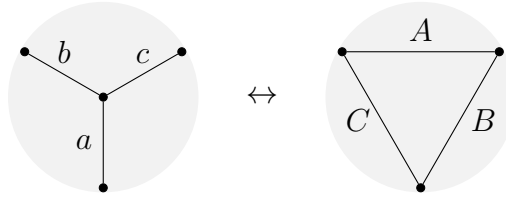
\begin{figure}[ht]
	\begin{tikzpicture}[scale=0.6]
		\def\ep{0.3}

		\begin{scope}[shift={(-3,0)}, rotate=180]
			\def\r{2};
			\fill[black!5] (0,0) circle (1.05*\r cm);
			\coordinate[nvert] (n1) at (0,0);
			\coordinate[nvert] (n2) at (90:\r);
			\coordinate[nvert] (n3) at (90+120:\r);
			\coordinate[nvert] (n4) at (-30:\r);

			\draw[-]
			(n1) edge node[left]{$a$} (n2) edge node[above]{$c$}(n3) edge node[above]{$b$}(n4)
			;

		  \end{scope}

		\node[](no) at (0.5,0) {$\leftrightarrow$};	
		
		\begin{scope}[shift={(4,0)}, rotate=180]
			\def\r{2};
			\fill[black!5] (0,0) circle (1.05*\r cm);
			
			\coordinate[nvert] (n2) at (90:\r);
			\coordinate[nvert] (n3) at (90+120:\r);
			\coordinate[nvert] (n4) at (-30:\r);

			\draw[-]
			(n2) -- node[right]{$B$} (n3) -- node[above]{$A$}(n4) -- node[left]{$C$}(n2)
			;

		  \end{scope}
	
	  \end{tikzpicture}
	  \caption{The Ising Y-$\Delta$ move.} \label{figyd}
\end{figure}

    Let 
    \[ \phi_n: \operatorname{Mat}_n^{\text{sym}}(\R,1) \rightarrow \Gr_{n,2n}\]
    be the \emph{doubling map} defined as follows: given $M = (m_{i,j}) \in \operatorname{Mat}_n^{\text{sym}}(\R,1)$, let $\phi_n(M) = \mathrm{rowspan}(\tilde M)$ where $\tilde M = (\tilde M_{i,j})$ is the $n \times 2n$-matrix defined as follows. For $i=j$, let 
\[
\tilde M_{i,2i-1} = \tilde M_{i,2i} = m_{i,i} =1,
\]
and for $i \neq j$, let 
\[
\tilde M_{i,2j-1} = -\tilde M_{i,2j} = (-1)^{i+j+ \mathbbm{1}(i<j)} m_{i,j},
\]
where $\mathbbm{1}(i<j)$ denotes $1$ if $i<j$ and $0$ otherwise. Similarly to \cref{thm:postnikov_boundary_meas}, we have:

\begin{theorem}[Galashin--Pylyavskyy \cite{GalashinPylyavskyy}]\label{thm:galashin_pylyavskyy}
Let $G$ be a reduced cactus graph and let $\pi = \pi_G$ be its fixed-point-free involution. Then
{
\[
(\phi_n \circ {\rm Corr}_G) (\mathcal{I}_G^{>0}) = \OGr_{n,2n}^{\geq 0} \cap \Pi_{\mathscr{M}_\pi}^{>0},
\]
and the map 

\[
 \mathcal{I}_G^{>0} \xrightarrow[]{\phi_n \circ {\rm Corr}_G} \OGr_{n,2n}^{\geq 0} \cap \Pi_{\mathscr{M}_\pi}^{>0}
\]
is a bijection. Moreover,} the boundary correlation maps ${\rm Corr}_G$ are compatible with Y-$\Delta$ moves, i.e., the following diagram commutes:
\[
\begin{tikzcd}
\mathcal I_G^{>0}
  \arrow[rr,"\sim"] 
  \arrow[dr, swap, "\mathrm{Corr}_G"] 
& 
& 
\mathcal I_{G'}^{>0}
  \arrow[dl, "\mathrm{Corr}_{G'}"] 
\\
& \operatorname{Mat}_n^{\mathrm{sym}}(\R,1) &
\end{tikzcd}.
\]
Therefore, the maps $(\mathrm{Corr}_G)_{G: \pi_G =\pi}$ glue together to a map
\[
\mathrm{Corr}_{\pi} : \mathcal I_{\pi}^{>0} \rightarrow \operatorname{Mat}_n^{\mathrm{sym}}(\R,1).
\]
which gives a bijection
\[
\phi_n\circ\mathrm{Corr}_{\pi}:\mathcal I_{\pi}^{>0} \xrightarrow[]{\sim} \OGr^{\geq 0}_{n,2n} \cap \Pi_{\mathscr{M}_\pi}^{>0}.
\]
\end{theorem}

 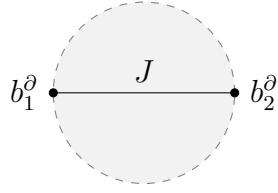
\begin{figure}[ht]
    \centering
    \begin{tikzpicture}[scale=0.6]
        \begin{scope}[shift={(-4,0)}]
      \def\r{2};
            \fill[black!5] (0,0) circle (1*\r cm);
            \draw[dashed, gray] (0,0) circle (1*\r cm);
            \coordinate[nvert, label = right: $b_2^\partial$] (n1) at (0:\r);
            \coordinate[nvert, label = left: $b_1^\partial$] (n2) at (180:\r);
            
            \coordinate (no) at (90:2);
            \coordinate (no) at (-90:2);
            
            \draw[] (n1) -- node[above]{$J$} (n2);

        \end{scope}

    \end{tikzpicture}
    \caption{An Ising model with $n=2$.}
    \label{fig:Ising_small_example}
\end{figure}

\begin{example}\label{eg:ising_n_2}
Consider the Ising model shown in \cref{fig:Ising_small_example}. Then
\[
 M = \begin{bmatrix}
    1&\langle \sigma_1 \sigma_2 \rangle\\
    \langle \sigma_1 \sigma_2 \rangle & 1
\end{bmatrix}, \qquad \text{where}~\langle \sigma_1 \sigma_2 \rangle = \frac{e^{J}-e^{-J}}{e^J+e^{-J}}=\tanh J.
\]
The point in $\OGr^{\geq 0}_{2,4}$ is the rowspan of
\[
\tilde M=\phi_2(M) = \begin{bmatrix}
    1&1&\langle \sigma_1 \sigma_2 \rangle&-\langle \sigma_1 \sigma_2 \rangle\\
    -\langle \sigma_1 \sigma_2 \rangle&\langle \sigma_1 \sigma_2 \rangle & 1 &1
\end{bmatrix}
\]
with Pl\"ucker coordinates
\begin{equation}\label{eq:plucker_ising_eg}
\Delta_{12} = \Delta_{34} = 2 \langle \sigma_1 \sigma_2 \rangle, \quad \Delta_{14}=\Delta_{23} = 1-\langle \sigma_1 \sigma_2 \rangle^2, \quad \Delta_{13} = \Delta_{24} =1+ \langle \sigma_1 \sigma_2 \rangle^2.
\end{equation}
\end{example}

\subsection{Embedding Ising models in dimer models}

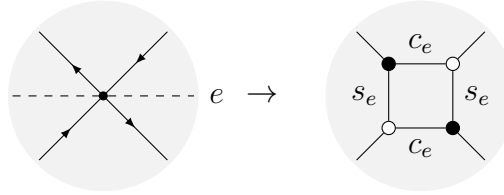
\begin{figure}[ht]
    \centering
    \begin{tikzpicture}[scale=0.6]

        \begin{scope}[shift={(10,0)}, decoration={
    markings,
    mark=at position 0.5 with {\arrow{>}}}]
           \def\r{2};
            \fill[black!5] (0,0) circle (1.05*\r cm);
            \coordinate[] (n1) at (135:\r);
            \coordinate[] (n2) at (225:\r);
            \node[nvert] at (0,0) {};
              \draw[dashed] (-2,0) -- (2,0);
            \node at (2.5,0) {$e$};
            \draw[postaction={decorate}] (0,0) -- (n1);
            \draw[postaction={decorate}]  (0,0) -- (-45:\r);
            \draw[postaction={decorate}] (n2) -- (0,0);
            \draw[postaction={decorate}]  (45:\r) -- (0,0);
            
        \end{scope}
        \node at (13.5,0) {$\rightarrow$};	
         \begin{scope}[shift={(17,0)}]
          \def\r{2};
            \fill[black!5] (0,0) circle (1.05*\r cm);
            \coordinate[wvert] (w1) at (225:0.5*\r);
            \coordinate[wvert] (w2) at (45:0.5*\r);
            \coordinate[bvert] (b1) at (135:0.5*\r);
            \coordinate[bvert] (b2) at (315:0.5*\r);

              \draw[] (w1) -- node[left]{$s_e$} (b1) --node[above]{$c_e$} (w2) --node[right]{$s_e$} (b2) --node[below]{$c_e$} (w1)
            (w1) -- (225:\r)
            (w2) -- (45:\r)
            (b1) -- (135:\r)
            (b2) -- (315:\r)
            ;
        \end{scope}

    \end{tikzpicture}
    \caption{The construction of the dimer model $(G^\square,[\wt^\square])$ from $G^\times$ and $J$. The edges of $G^\square$ with no edge weight indicated have edge weight $1$.     
    }
    \label{fig:medial_graph_to_G_square}
\end{figure}

Let $G$ be a reduced cactus graph. We orient the edges of the medial graph \( G^\times \) so that:
\begin{enumerate}
    \item The boundary vertex \( u_i^\partial \) is a source if and only if \( i \) is odd.
    \item Each internal vertex of \( G^\times \) is incident to two incoming and two outgoing edges which alternate in orientation around the vertex.
\end{enumerate}
We construct a reduced plabic graph \( G^\square \) by replacing each vertex of \( G^\times \) as shown on the left-hand-side of \cref{fig:medial_graph_to_G_square} with the bipartite graph on the right-hand-side. We refer to the degree-\(4\) faces of \( G^\square \) corresponding to edges of \( G \) as \emph{square faces}. The remaining faces of \( G^\square \) are in bijection with the vertices and faces of \( G \). Note that there is a \( 2:1 \) correspondence between strands in \(G^\square\) and strands in \(G\) as follows. The strand in \(G\) with source \(u_i^\partial\) corresponds to the pair of strands in \( G^\square\) with sources \( u_i^\partial,u_{\bar{i}}^\partial \). Moreover, we have 
\( \pi_{G^\square}=\pi_G \).

Given an Ising model \( (G,J) \), we define the two functions $s, c : E(G) \rightarrow (0,1)$ by 
\[
s_e := \operatorname{sech}(2J_e),\qquad c_e:= \operatorname{tanh}(2J_e) 
\]
for every edge $e \in E(G)$ and replace $e$ with the dimer model $(G^\square,[\wt^\square])$ shown in \cref{fig:medial_graph_to_G_square}. The map
\[
\mathcal I_{G}^{>0} \hookrightarrow \mathcal X_{G^\square}^{>0}, \quad (G,J) \mapsto (G^\square,[\wt^\square])
\]
is injective (see the proof of \cite[Theorem~5.17]{GalashinPylyavskyy}).

\begin{proposition}[Kenyon--Pemantle~\cite{KenyonPemantle1}]
Let $G \to G'$ be a Y-$\Delta$ move. Then there is a sequence of moves $G^\square \to (G')^\square$ such that the following diagram commutes:
\[
\begin{tikzcd}
\mathcal I_{G}^{>0} 
  \arrow[r, "\sim"] \arrow[d, hook] 
& 
\mathcal I_{G'}^{>0} 
  \arrow[d, hook] 
\\
\mathcal X_{G^\square}^{>0}
    \arrow[r,   "\sim"]
& 
\mathcal X_{(G')^\square}^{>0}
  \end{tikzcd}.
\]
In particular, the embeddings \(
\mathcal I_{G}^{>0} \hookrightarrow \mathcal X_{G^\square}^{>0}
\) glue together to give an embedding 
\[
\mathcal I_{\pi}^{>0}\hookrightarrow \mathcal X_{\mathscr{M}_\pi}^{>0}.
\]
\end{proposition}

\begin{theorem}[Galashin--Pylyavskyy~\cite{GalashinPylyavskyy}]\label{thm:gal_pyl_main}
 The following diagram commutes:
\[
\begin{tikzcd}
\mathcal I_{\pi}^{>0} 
  \arrow[dd, "\mathrm{Corr}_\pi"']
  \arrow[r, hook]
&
\mathcal X_{\mathscr{M}_\pi}^{>0}
  \arrow[d, "{\rm Meas}_{\mathscr{M}_\pi}"']
\\
&
\Pi_{\mathscr{M}_\pi}^{>0}
  \arrow[d, hookleftarrow, ""']
\\
\operatorname{Mat}_n^{\mathrm{sym}}(\R,1)
  \arrow[r, "\phi_n"']
&
\OGr^{\ge 0}_{n,2n} \cap \Pi_{\mathscr{M}_\pi}^{>0}
\end{tikzcd}
\]
\end{theorem}

\begin{figure}[ht]
    \centering
    \scalebox{0.8}{
    \begin{tikzpicture}
        \begin{scope}[shift={(-1,0)}]
            \def\r{2};
            \fill[black!5] (0,0) circle (1*\r cm);
            \draw[dashed, gray] (0,0) circle (1*\r cm);
            \coordinate[nvert, label=right: $b_2^\partial$] (n1) at (0:\r);
            \coordinate[nvert, label=left: $b_1^\partial$] (n2) at (180:\r);

            \draw[] (n1) -- node[above]{$J$} (n2);
            \node at (0,-2.5) {(a)};
        \end{scope}

        \node at (3,0) {$\mapsto$};

        \begin{scope}[shift={(7,0)}]
            \def\r{2};
            \fill[black!5] (0,0) circle (1*\r cm);
            \draw[dashed, gray] (0,0) circle (\r cm);
            \coordinate[wvert] (w1) at (225:0.5*\r);
            \coordinate[wvert] (w2) at (45:0.5*\r);
            \coordinate[bvert] (b1) at (135:0.5*\r);
            \coordinate[bvert] (b2) at (315:0.5*\r);

            \coordinate[label=below left: $u_1^\partial$] (b1') at (225:\r);
            \coordinate[label=above right: $u_3^\partial$] (b2') at (45:\r);
            \coordinate[label=above left: $u_2^\partial$] (w1') at (135:\r);
            \coordinate[label=below right: $u_4^\partial$] (w2') at (315:\r);

            \draw[]
            (w1) -- node[left]{$s$} (b1)
                 -- node[above]{$c$} (w2)
                 -- node[right]{$s$} (b2)
                 -- node[below]{$c$} (w1)
            (w1) -- node[above]{$1$} (b1')
            (w2) -- node[above]{$1$} (b2')
            (b1) -- node[above]{$1$} (w1')
            (b2) -- node[above]{$1$} (w2');

            \node at (0,-2.5) {(b)};
        \end{scope}
    \end{tikzpicture}}
    \caption{The Ising model from \cref{fig:Ising_small_example} and the corresponding dimer model. Here $c=\tanh(2J)$ and $s=\operatorname{sech}(2J)$.}
    \label{fig:n2_GP_commutativity}
\end{figure}
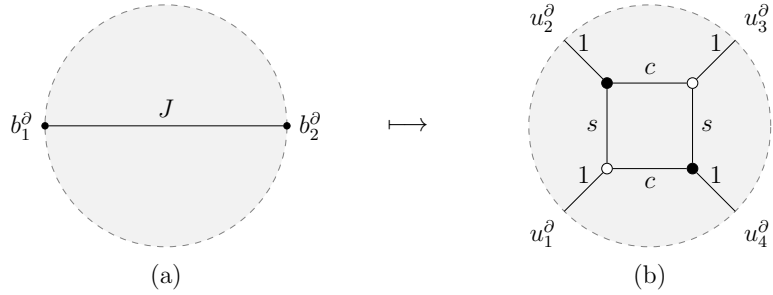

\begin{example}\label{eg:n2_GP_commutativity}
Consider the Ising model in \cref{fig:n2_GP_commutativity}(a). The associated point in $\OGr^{>0}_{2,4}$ was computed in \cref{eg:ising_n_2}. Since Pl\"ucker coordinates are projective, we may divide~\eqref{eq:plucker_ising_eg} by $1+\langle \sigma_1 \sigma_2 \rangle^2$ and write the same point as
\[
\Delta_{12}=\Delta_{34}=\frac{2\langle \sigma_1 \sigma_2 \rangle}{1+\langle \sigma_1 \sigma_2 \rangle^2},\qquad
\Delta_{13}=\Delta_{24}=1,\qquad
\Delta_{14}=\Delta_{23}=\frac{1-\langle \sigma_1 \sigma_2 \rangle^2}{1+\langle \sigma_1 \sigma_2 \rangle^2}.
\]
Using the identities
\[
\frac{2\tanh J}{1+\tanh^2 J}=\tanh(2J),
\qquad
\frac{1-\tanh^2 J}{1+\tanh^2 J}=\operatorname{sech}(2J),
\]
and setting
\[
c:=\tanh(2J),\qquad s:=\operatorname{sech}(2J),
\]
we obtain
\begin{equation}\label{eq:plucker_sing_n2}
\Delta_{12}=c,\quad
\Delta_{13}=1,\quad
\Delta_{14}=s,\quad
\Delta_{23}=s,\quad
\Delta_{24}=1,\quad
\Delta_{34}=c.
\end{equation}

On the other hand, the embedding
\(
\mathcal I_\pi^{>0}\hookrightarrow \mathcal X_{\mathscr M_\pi}^{>0}
\)
sends the Ising model in \cref{fig:n2_GP_commutativity}(a) to the dimer model in \cref{fig:n2_GP_commutativity}(b), where the edge weights are
\[
a=s,\qquad b=c,\qquad c=s,\qquad d=c
\]
in the notation of \cref{fig:example_MS_twist_rotated}. Applying the boundary measurement computed in \cref{fig:example_MS_twist_rotated} and using
\[
s^2+c^2=\operatorname{sech}^2(2J)+\tanh^2(2J)=1,
\]
we get
\[
\Delta_{12}=c,\quad
\Delta_{13}=1,\quad
\Delta_{14}=s,\quad
\Delta_{23}=s,\quad
\Delta_{24}=s^2+c^2=1,\quad
\Delta_{34}=c,
\]
which is the same as~\eqref{eq:plucker_sing_n2}, verifying \cref{thm:gal_pyl_main} for this example.
\end{example}

We give another characterization of the image 
\[
\mathcal I_{G}^{>0} \hookrightarrow \mathcal X_{G^\square}^{>0}
\]
found in~\cite{GeorgeIsing} for Ising models on a torus. The proof given there however works for arbitrary surfaces. Let $\bar{G^\square}$ be the plabic graph obtained from $G^\square$ by changing the colors of all the vertices. We define two operations:
\begin{enumerate}
    \item Color-change: We replace $G^\square$ with $\overline{G^\square}$, i.e., change colors of all vertices, but keep all the edge weights the same. 
    \item Spider moves at all square faces: We perform a spider move at each square face of $G^\square$. 
\end{enumerate}

\begin{proposition}[George~\cite{GeorgeIsing}] \label{prop:George_characterization}
    $[\wt] \in \mathcal X_{G^\square}^{>0}$ is in the image of $\mathcal I_{G}^{>0}$ if and only if color-change results in the same gauge-equivalence class of edge weights as spider moves at all square faces.
\end{proposition}
Under boundary measurement, color-change is the map $\Delta_I \mapsto \Delta_{[2n]\setminus I}$ while doing spider moves at all square faces does not change the boundary measurement. Since $\OGr^+_{n,2n}$ is defined by $\Delta_I = \Delta_{[2n]\setminus I}$ for all $I \in \binom{[2n]}{n}$, this gives an alternative proof that the composition 
\[
\mathcal I_\pi^{>0} \hookrightarrow \mathcal X^{>0}_{\mathscr{M}_\pi} \xrightarrow[]{{\rm Meas}_{\mathscr{M}_\pi}} \Pi^{>0}_{\mathscr{M}_\pi}
\]
lands in $\OGr_{n,2n}^{\geq 0}$. We will see that an identical characterization exists on the cluster $\mathcal A$ side (\cref{remark:col_change_sq}). However, the cluster ensemble and twist maps are not compatible with color-change and will need to be modified (see \cref{sec:ising_cluster_ensemble_mod_twist}).

\begin{figure}[ht]
    \centering
    \scalebox{0.8}{
    \begin{tikzpicture}
        \begin{scope}[shift={(-2,0)}]
            \def\r{2};
            \fill[black!5] (0,0) circle (1*\r cm);
            \draw[dashed, gray] (0,0) circle (\r cm);
            \coordinate[wvert] (w1) at (225:0.5*\r);
            \coordinate[wvert] (w2) at (45:0.5*\r);
            \coordinate[bvert] (b1) at (135:0.5*\r);
            \coordinate[bvert] (b2) at (315:0.5*\r);

            \coordinate[ label = below left: $u_4^\partial$] (b1') at (225:\r);
            \coordinate[ label = above right: $u_2^\partial$] (b2') at (45:\r);
            \coordinate[ label = above  left: $u_1^\partial$] (w1') at (135:\r);
            \coordinate[ label = below right: $u_3^\partial$] (w2') at (315:\r);

            \draw[] (w1) -- node[left]{$a$} (b1) --node[above]{$b$} (w2) --node[right]{$c$} (b2) --node[below]{$d$} (w1)
            (w1) -- node[above]{$1$} (b1')
            (w2) -- node[above]{$1$} (b2')
            (b1) -- node[above]{$1$} (w1')
            (b2) -- node[above]{$1$} (w2');
            \node at (0,-2.5) {(a)};
         \end{scope}
 
         \begin{scope}[shift={(5,0)}]
            \def\r{2};
            \fill[black!5] (0,0) circle (1*\r cm);
            \draw[dashed, gray] (0,0) circle (\r cm);
            \coordinate[bvert] (w1) at (225:0.5*\r);
            \coordinate[bvert] (w2) at (45:0.5*\r);
            \coordinate[wvert] (b1) at (135:0.5*\r);
            \coordinate[wvert] (b2) at (315:0.5*\r);

            \coordinate[ label = below left: $u_4^\partial$] (b1') at (225:\r);
            \coordinate[ label = above right: $u_2^\partial$] (b2') at (45:\r);
            \coordinate[ label = above  left: $u_1^\partial$] (w1') at (135:\r);
            \coordinate[ label = below right: $u_3^\partial$] (w2') at (315:\r);

            \draw[] (w1) -- node[left]{$a$} (b1) --node[above]{$b$} (w2) --node[right]{$c$} (b2) --node[below]{$d$} (w1)
            (w1) -- node[above]{$1$} (b1')
            (w2) -- node[above]{$1$} (b2')
            (b1) -- node[above]{$1$} (w1')
            (b2) -- node[above]{$1$} (w2');
            \node at (0,-2.5) {(b)};
        \end{scope}

        \begin{scope}[shift={(12,0)}]
            \def\r{2};
            \fill[black!5] (0,0) circle (1*\r cm);
            \draw[dashed, gray] (0,0) circle (\r cm);
            \coordinate[bvert] (w1) at (225:0.5*\r);
            \coordinate[bvert] (w2) at (45:0.5*\r);
            \coordinate[wvert] (b1) at (135:0.5*\r);
            \coordinate[wvert] (b2) at (315:0.5*\r);

            \coordinate[ label = below left: $u_4^\partial$] (b1') at (225:\r);
            \coordinate[ label = above right: $u_2^\partial$] (b2') at (45:\r);
            \coordinate[ label = above  left: $u_1^\partial$] (w1') at (135:\r);
            \coordinate[ label = below right: $u_3^\partial$] (w2') at (315:\r);

            \draw[] (w1) -- node[left]{$\frac{c}{ac+bd}$} (b1) --node[above]{$\frac{d}{ac+bd}$} (w2) --node[right]{$\frac{a}{ac+bd}$} (b2) --node[below]{$\frac{b}{ac+bd}$} (w1)
            (w1) -- node[above]{$1$} (b1')
            (w2) -- node[above]{$1$} (b2')
            (b1) -- node[above]{$1$} (w1')
            (b2) -- node[above]{$1$} (w2');
            
            \node at (0,-2.5) {(c)};
        \end{scope}

        \end{tikzpicture}}
    \caption{A plabic-graph $G^\square$ (a), the result of color-change (b) and spider moves at all square faces (c).}
    \label{fig:example_G_characterization}
\end{figure}
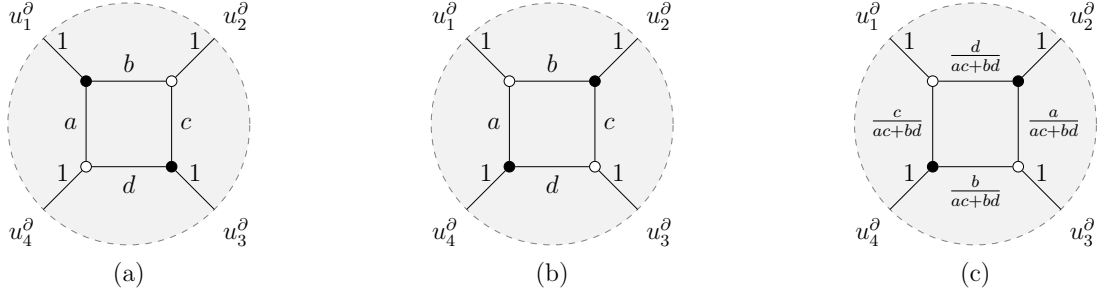

\begin{example}
Consider the dimer model $(G^\square,[\wt])$ shown in \cref{fig:example_G_characterization}(a). The results of color-change and spider moves at all square faces are shown in (b) and (c) respectively. The characterization in \cref{prop:George_characterization} says that the image of $\mathcal I_{G}^{>0}$ is defined by
\[
a = \frac{c}{ac+bd}, \quad b = \frac{d}{ac+bd},\quad c = \frac{a}{ac+bd},\quad d = \frac{b}{ac+bd}.
\]
The first and third equations imply that $a^2 = c^2$. Since $a,c>0$, this gives $a=c$. Similarly, the second and the fourth equations imply that $b=d$. Finally, any of the four equations gives $a^2+b^2=1$, so we get $s_e=a, c_e=b$.
\end{example}

    \section{\texorpdfstring{$C$}{C}-networks and locally Lagrangian positroid cells}\label{sec:4}
\subsection{\texorpdfstring{$C$}{C}-networks}

A \emph{$C$-network} is a pair $(G,C)$ where $G$ is a reduced cactus graph and 
    \[C: V(G) \sqcup F(G) \rightarrow \Rpos\]
    is a function defined modulo rescaling by $\Rpos$. Let \[ \mathcal C_G^{>0} \cong \Rpos^{|V(G)| + |F(G)|-1} \] be the \emph{space of $C$-networks with underlying cactus graph $G$.}

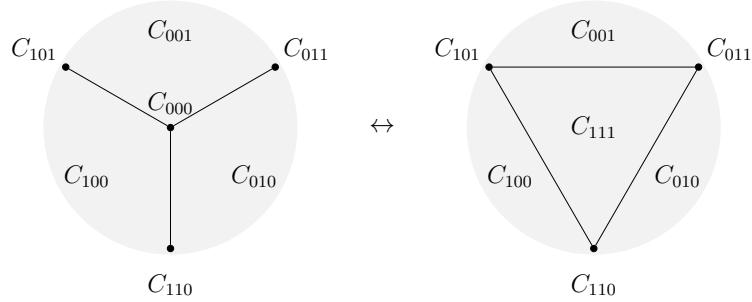
\begin{figure}[ht]
	\scalebox{0.8}{
    \begin{tikzpicture}[scale=1]
		\def\ep{0.3}
        
		\begin{scope}[shift={(-3,0)}, rotate=180]
			\def\r{2};
			\fill[black!5] (0,0) circle (1.05*\r cm);
			\coordinate[nvert] (n1) at (0,0);
			\coordinate[nvert] (n2) at (90:\r);
			\coordinate[nvert] (n3) at (90+120:\r);
			\coordinate[nvert] (n4) at (-30:\r);
    			
    		\node at (30:0.8*\r) {$C_{100}$};
            \node at (30+120:0.8*\r) {$C_{010}$};
            \node at (30-120:0.8*\r) {$C_{001}$};
    
            \node at (-30:1.3*\r) {$C_{101}$};
            \node at (-30+120:1.3*\r) {$C_{110}$};
            \node at (-30-120:1.3*\r) {$C_{011}$};
            
            \draw[-] (n1) edge (n2) edge (n3) edge (n4);
    		\node at (0,-0.4) {$C_{000}$};
		\end{scope}

		\node[](no) at (0.5,0) {$\leftrightarrow$};	
		
		\begin{scope}[shift={(4,0)}, rotate=180]
            \def\r{2};
            \fill[black!5] (0,0) circle (1.05*\r cm);
            
            \coordinate[nvert] (n2) at (90:\r);
            \coordinate[nvert] (n3) at (90+120:\r);
            \coordinate[nvert] (n4) at (-30:\r);

            \draw[-] (n2) --  (n3) -- (n4) -- (n2);
    
            \node at (30:0.8*\r) {$C_{100}$};
            \node at (30+120:0.8*\r) {$C_{010}$};
            \node at (30-120:0.8*\r) {$C_{001}$};
    
            \node at (-30:1.3*\r) {$C_{101}$};
            \node at (-30+120:1.3*\r) {$C_{110}$};
            \node at (-30-120:1.3*\r) {$C_{011}$};
    
            \node at (0,0) {$C_{111}$};
		  \end{scope}
	  \end{tikzpicture}}
	  \caption{The Y-$\Delta$ move for $C$-networks.} \label{fig:yd_C}
\end{figure}    

A Y-$\Delta$ move $G \to G'$ induces a map
\[
\mathcal C_G^{>0} \xrightarrow[]{\sim} \mathcal C_{G'}^{>0}
\]
given in the notation of \cref{fig:yd_C} by the \emph{positive Kashaev equation} 
\begin{align*}
C_{111} &= \frac{1}{C_{000}^2}\Bigg(
  2C_{100}C_{010}C_{001}
  + C_{000}(
    C_{100}C_{011}
    + C_{010}C_{101}
    + C_{001}C_{110}
  )\\
  & + 2 \sqrt{
    (C_{000}C_{011}+C_{010}C_{001})
    (C_{000}C_{101}+C_{100}C_{001})
    (C_{000}C_{110}+C_{100}C_{010})
    }\Bigg)
\end{align*}

By construction, we have a bijection
\[
V(G) \sqcup E(G) \sqcup F(G) \xrightarrow[]{\sim} F(G^\square).
\]
Note that the edges of $G$ correspond to the square faces of $G^\square$. We define an embedding
\[
\mathcal C_G^{>0} \hookrightarrow \mathcal A_{G^\square}^{>0}
\]
as follows. We let $A_x := C_x$ for $x \in V(G) \sqcup F(G)$ and for an edge $e = uv\in E(G)$ incident to faces $f,g$, we let 
\begin{equation} \label{eq:A_in_terms_of_C}
    A_e := \sqrt{C_u C_v + C_f C_g}.
\end{equation}

\begin{remark} \label{remark:col_change_sq}
We can interpret~\eqref{eq:A_in_terms_of_C} as saying that color-change gives the same $A$-network as applying spider moves at all square faces:
\begin{enumerate}
    \item Color-change: We replace $G^\square$ with $\overline{G^\square}$, i.e., change colors of all vertices, but keep the face weights the same.
    \item Spider moves at all square faces: We perform a spider move at each square face of $G^\square$, which replaces $A_e$ by $({A_u A_v+A_f A_g})/{A_e}$.  
\end{enumerate}
Setting $({A_u A_v+A_f A_g})/{A_e} = A_e$, we get~\eqref{eq:A_in_terms_of_C}.
\end{remark}

We now define the subset of $\Pi_{\mathscr{M}_\pi}$ parameterized by $C$-networks. 

{
\begin{definition}
Given a matching $\pi$ on $[2n]$ we define the space of \emph{$C$-networks with matching $\pi$ (or with positroid $\mathscr{M}_\pi$)} to be
\[
 \mathcal{C}_{\pi}^{>0}
:= 
\left(\bigsqcup_{ G \colon \pi_{G} = \pi}
\mathcal{C}_G^{>0} \right)
\!\Big/\! \text{moves}.
\]
\end{definition}
}

\begin{proposition}[Kenyon--Pemantle~\cite{KenyonPemantle1}] \label{prop:C_to_A_yd}
Let $G \to G'$ be a Y-$\Delta$ move. Then there is a sequence of moves $G^\square \to (G')^\square$ such that the following diagram commutes:
\[
\begin{tikzcd}
\mathcal C_{G}^{>0} 
  \arrow[r, "\sim"] \arrow[d, hook] 
& 
\mathcal C_{G'}^{>0} 
  \arrow[d, hook] 
\\
\mathcal A_{G^\square}^{>0}
    \arrow[r,   "\sim"]
& 
\mathcal A_{(G')^\square}^{>0}
  \end{tikzcd}.
\]
Moreover, given a matching $\pi$ on $[2n]$, the maps $\mathcal C_{G}^{>0} \hookrightarrow \mathcal A_{G^\square}^{>0}$ for $G$ such that $\pi_G = \pi$ glue to give an embedding 
\[
\mathcal C_{\pi}^{>0}\hookrightarrow \mathcal A_{\mathscr{M}_\pi}^{>0}.
\]

\end{proposition}

{
\begin{definition}[Locally Lagrangian positroid cells]
\label{def:LocallyLagrangianPositroidCells}
We denote by $\mathcal S_G^\square$ the set of all $I \in \binom{[2n]}{n}$ which appear as the target label of a square face of $G^\square$, where $G$ is a reduced cactus graph with $\pi_G=\pi$, and let 
\[
\mathcal S_\pi^\square := \bigcup_{G:\pi_G=\pi} \mathcal S_{G}^\square.
\] 
Note that any $I \in \mathcal{S}_{\pi}^{\square}$ is of the form $S i\bar i$ where:
\begin{enumerate}
    \item $\{\gamma_{i, \bar i}$, $\gamma_{j,\bar{j}} \}$ is a \emph{crossing}, i.e., a pair of strands $\{\gamma_{i, \bar i}$, $\gamma_{j,\bar{j}} \}$ with \( i<j<\bar{i}<\bar{j}\) in cyclic order around the boundary of \( \disk \).
    \item $S$ is a subset of $[2n]$ of size $(n-2)$ that contains exactly one element from each pair $\{k,\bar k\}$ for all strands $\gamma_{k,\bar k} \neq \gamma_{i,\bar i},\gamma_{j,\bar j}$.
\end{enumerate}
To such an element $I \in \mathcal{S}_{\pi}^\square$, we associate the equation 
\begin{equation}\label{eq:cond_plucker_equal}
    \Delta_{S i \bar i }(V) = \Delta_{S j \bar j}(V).
\end{equation}
For any matching $\pi \in P_n$, we define the \emph{locally Lagrangian positroid cell} $\Lambda_{\pi}^{>0} \subset \Pi_{\mathscr{M}_\pi}^{>0}$ as the subset of subspaces $V$ that satisfy \eqref{eq:cond_plucker_equal} for every $I \in \mathcal S_\pi^\square$.
\end{definition}
}

 When $\pi$ is the matching $i \mapsto i+n$ (modulo~$2n$), $\Lambda_\pi^{>0}$ is the intersection of the Lagrangian Grassmannian with $\Gr_{n,2n}^{>0}$ or equivalently the space of positive-definite Hermitian matrices~\cite[Section 5]{KenyonPemantle2}. In general, \eqref{eq:cond_plucker_equal} can be viewed as a local Lagrangian condition, but we do not know a global description of these spaces.

\begin{proposition}
\label{prop:local_lagrangian_condition}
Let $\pi$ be a matching on $[2n]$ and let $V\in \Pi_{\mathscr M_\pi}^{>0}$. Then $V \in \Lambda_\pi^{>0} $ if and only if the following condition holds:
\begin{quote}
    For every square-face label $I = S i\bar i\in \mathcal{S}_{\pi}^\square$,
let
\(
S j \bar{j}
\)
be the label obtained from the same local square after applying the spider move such that
\(
i<j<\bar{i}<\bar{j}.
\)
Set
\(
W_S:=\Span\{e_s: s \in  [2n] \setminus S\}
\subset \R^{2n},
U_{ij}:=\Span\{e_i,e_{\bar i},e_j,e_{\bar j}\},
\)
and let 
\[
\mathrm{pr}_{ij}:\R^{2n}\to U_{ij}
\]
be the coordinate projection. We equip the space \(U_{ij}\) with the symplectic form
\[
\omega_{ij} := e^\ast_{i} \wedge e^\ast_{\bar{i}}  - \epsilon_{Sij} \ e^\ast_{j} \wedge e^\ast_{\bar{j}}, \quad \text{where} \quad  \epsilon_{Sij} := (-1)^{|S \cap (i,j)|  +  |S \cap (\bar i,\bar j)|}.
\]
Here $(i,j)$ is the open interval  $(i,j) := \{i+1, \dots, j-1\}$. Then the vector space
\(
\mathrm{pr}_{ij}(V\cap W_S)
\subset U_{ij}
\)
is a Lagrangian $2$-plane with respect to $\omega_{ij}$.
\end{quote}

\end{proposition}

\begin{proof}
Let $M$ be an $n\times 2n$ matrix whose rowspan is $V$. Let $M'$ be the matrix obtained from $M$ by permuting the columns so that the columns corresponding to $S$ appear as the first $n-2$ columns. Then,
\begin{equation}\label{eq:pluck_change}
    \Delta_{S i \bar i}(M') = (-1)^{|S \cap (i,\bar i)|}  \Delta_{S i \bar i}(M), \qquad \Delta_{S j \bar j}(M') = (-1)^{|S \cap (j,\bar j)|}  \Delta_{S j \bar j}(M).
\end{equation}
Since $\Delta_{Si \bar i}(V) \neq 0$, by performing row operations, we may assume that the columns in $S i \bar i$ contain the identity matrix. Then $V \cap W_S$ is the rowspan of the matrix obtained from $M$ by deleting the rows corresponding to $S$ in the identity matrix, which leaves a $2 \times 2n$ matrix. The projection $\mathrm{pr}_{ij}(V\cap W_S)$ is then obtained by deleting the columns other than $i,\bar i, j, \bar j$. By construction, every Pl\"ucker coordinate of $\mathrm{pr}_{ij}(V\cap W_S)$ is equal to the Pl\"ucker coordinate of $M'$ obtained by adjoining $S$. Thus $\mathrm{pr}_{ij}(V\cap W_S)$ is Lagrangian in $U_{ij}$ if and only if $\Delta_{S i \bar i}(M') = \epsilon_{Sij}\Delta_{S j \bar j}(M')$. By~\eqref{eq:pluck_change}, this is equivalent to~\eqref{eq:cond_plucker_equal}.
\end{proof}

\begin{example}[$n=3$] We consider the matching $\pi$ from \cref{fig:medialGraphExample}. The cell $\Lambda^{>0}_\pi$ is cut out by the following conditions:
\begin{enumerate}
    \item $\Delta_{126} = \Delta_{345} = 0$ and  $\Delta_I > 0$ for $I \in \binom{[6]}{3} \setminus \{\{1,2,6\}, \{3,4,5\}\}$ (positroid condition),
    
    \item $\Delta_{356} = \Delta_{146}$ and $\Delta_{134} = \Delta_{236}$ (square faces condition).
\end{enumerate}
\end{example}

\begin{definition}[The map $\bm{H}_G$ from locally Lagrangian cells to $C$-networks]\label{def:H_G}
Given a reduced cactus graph $G$ with $\pi_G=\pi$, we define the map
\[
\bm H_G: \Lambda_\pi^{>0} \rightarrow \mathcal C_G^{>0}, \qquad V \mapsto (\Delta_{\bm T(x)}(V))_{x \in V(G) \sqcup F(G)},
\]
where $\bm T(x)$ denotes the target face label of the corresponding face of $G^\square$.
\end{definition}

Note that the condition~\eqref{eq:cond_plucker_equal} implies~\eqref{eq:A_in_terms_of_C} and therefore the following diagram commutes:
\begin{equation}\label{eq:C_A}
    \begin{tikzcd}
\mathcal C_{G}^{>0} 
  \arrow[r, hook]  
& 
\mathcal A_{G^\square}^{>0} 
\\
\Lambda_\pi^{>0}
    \arrow[r,hook]\arrow[u,"\bm H_G" ]
& 
\Pi_{\mathscr{M}_\pi}^{>0}   \arrow[u,"\bm F_{G^\square}"'] 
  \end{tikzcd}.
\end{equation}

\begin{lemma}
If $G \to G'$ is a Y-$\Delta$ move, the following diagram commutes:
\[
\begin{tikzcd}
\mathcal C_{G}^{>0} 
  \arrow[rr, "\sim"] 
& &
\mathcal C_{G'}^{>0} 
\\
& 
\Lambda_\pi^{>0}\arrow[ur,"{\bm H_{G'}}"' ]\arrow[ul,"{\bm H_{G}}" ]
&
  \end{tikzcd}.
\]
\end{lemma}

\begin{proof}
The proof is a simple diagram chase. By injectivity of $\mathcal C_{G'}^{>0} \hookrightarrow \mathcal A_{(G')^\square}^{>0}$, it suffices to show that
\[
\Lambda_\pi^{>0} \xrightarrow[]{\bm H_{G}}\mathcal C_{G}^{>0} \xrightarrow[]{\sim }\mathcal C_{G'}^{>0} \hookrightarrow \mathcal A_{(G')^\square}^{>0} = \Lambda_\pi^{>0} \xrightarrow[]{\bm H_{G'}}\mathcal C_{G'}^{>0} \hookrightarrow \mathcal A_{(G')^\square}^{>0}. 
\]
By \cref{prop:C_to_A_yd}, we have 
\[
\Lambda_\pi^{>0} \xrightarrow[]{\bm H_{G}}\mathcal C_{G}^{>0} \xrightarrow[]{\sim }\mathcal C_{G'}^{>0} \hookrightarrow \mathcal A_{(G')^\square}^{>0} = \Lambda_\pi^{>0} \xrightarrow[]{\bm H_{G}}\mathcal C_{G}^{>0}\hookrightarrow \mathcal A_{G^\square}^{>0}\xrightarrow[]{\sim }\mathcal A_{(G')^\square}^{>0}. 
\]
By \eqref{eq:C_A} and the corresponding diagram for $G'$, we have
\begin{align*}
\Lambda_\pi^{>0} \xrightarrow[]{\bm H_{G}}\mathcal C_{G}^{>0}\hookrightarrow \mathcal A_{G^\square}^{>0}\xrightarrow[]{\sim }\mathcal A_{(G')^\square}^{>0} &= \Lambda_\pi^{>0} \hookrightarrow \Pi^{>0}_{\mathscr{M}_\pi} \xrightarrow[]{\bm F_{G^\square}} \mathcal A_{G^\square}^{>0}\xrightarrow[]{\sim }\mathcal A_{(G')^\square}^{>0},\\
\Lambda_\pi^{>0} \xrightarrow[]{\bm H_{G'}}\mathcal C_{G'}^{>0} \hookrightarrow \mathcal A_{(G')^\square}^{>0} &= \Lambda_\pi^{>0} \hookrightarrow \Pi^{>0}_{\mathscr{M}_\pi} \xrightarrow[]{\bm F_{(G')^\square}} \mathcal A_{(G')^\square}^{>0},
\end{align*}
and these are equal by \cref{thm:scott_MS}.
\end{proof}

\begin{proposition}
    The map $\bm H_G \colon \Lambda_{\pi}^{> 0} \to \mathcal{C}_{G}^{>0}$ is a bijection.
\end{proposition}
\begin{proof}
$\bm H_G$ is injective since all the other maps in~\eqref{eq:C_A} are injective. 
For surjectivity, given an element $C \in \mathcal C_G^{>0}$, the composition $\mathcal C_{G}^{>0} \hookrightarrow \mathcal A_{G^\square}^{>0}  \xrightarrow[]{\bm F_{G^\square}^{-1}} \Pi^{>0}_{\mathscr{M}_\pi}$ gives us a point $V \in \Pi^{>0}_{\mathscr{M}_\pi}$ and it remains to show that $V \in \Lambda_\pi^{>0}$. So we need to check that this $V$ satisfies \eqref{eq:cond_plucker_equal}. The conditions hold for $I \in \mathcal S^\square_G$ by construction, and they hold for $I \in \mathcal S^\square_{G'}$ for $G'$ move-equivalent to $G$ by \cref{prop:C_to_A_yd}. 
\end{proof}

Gluing the maps $(\bm H_G)_{G: \pi_G=\pi}$, we obtain a bijection $\bm H_\pi: \Lambda_\pi^{>0} \rightarrow \mathcal C_\pi^{>0}$ such that the following diagram commutes:
\[
\begin{tikzcd}
\mathcal C_{\pi}^{>0} 
  \arrow[r, hook]  
& 
\mathcal A_{\mathscr{M}_\pi}^{>0} 
\\
\Lambda_\pi^{>0}
    \arrow[r,hook]\arrow[u,"\bm H_\pi" ]
& 
\Pi_{\mathscr{M}_\pi}^{>0}   \arrow[u,"\bm F_{\mathscr{M}_\pi}"'] 
  \end{tikzcd}.
\]

\section{The Ising cluster ensemble and twist maps} \label{sec:ising_cluster_ensemble_mod_twist}

\subsection{Ising cluster ensemble}\label{sec:5}

We now return to the Ising model and formulate the chamber ansatz from the introduction in cluster-ensemble language. 
   
\begin{definition}[Ising cluster ensemble map]\label{def:IsingClusterEnsemble}
Let $G$ be a reduced cactus graph. We denote by $q_{G}:\mathcal C_G^{>0} \rightarrow \mathcal I_G^{>0}$ the map defined as in \cref{fig:cluster_ensemble}, where
\begin{equation}\label{eq:c_and_s_from_A}
 c_e := \sqrt{\frac{C_uC_v}{C_uC_v+C_fC_g}} = \frac{\sqrt{A_u A_v}}{A_e} \quad \text{and} \quad s_e := \sqrt{\frac{C_fC_g}{C_uC_v+C_fC_g}}= \frac{\sqrt{A_f A_g}}{A_e}.   
\end{equation}
We call this map the \emph{Ising cluster ensemble map}.
\end{definition}

\begin{figure}[hb]
    \centering
    \begin{tikzpicture}[scale=0.6]
        \begin{scope}[shift={(-4,0)}]
      \def\r{2};
            \fill[black!5] (0,0) circle (1.05*\r cm);
            \coordinate[nvert, label = right: $C_v$] (n1) at (0:\r);
            \coordinate[nvert, label = left: $C_u$] (n2) at (180:\r);
            
            \coordinate[label=center: $C_f$] (no) at (90:2);
            \coordinate[label = center: $C_g$] (no) at (-90:2);
            
            \draw[] (n1) -- node[above]{$e$} (n2);
        \end{scope}
        \node at (0.5,0) {$\rightarrow$};	
         \begin{scope}[shift={(5,0)}]
            \def\r{2};
            \fill[black!5] (0,0) circle (1.05*\r cm);
            \coordinate[wvert] (w1) at (225:0.5*\r);
            \coordinate[wvert] (w2) at (45:0.5*\r);
            \coordinate[bvert] (b1) at (135:0.5*\r);
            \coordinate[bvert] (b2) at (315:0.5*\r);
            
            \draw[] (w1) -- node[left]{$s_e$} (b1) --node[above]{$c_e$} (w2) --node[right]{$s_e$} (b2) --node[below]{$c_e$} (w1)
            (w1) -- (225:\r)
            (w2) -- (45:\r)
            (b1) -- (135:\r)
            (b2) -- (315:\r)
            ;
        \end{scope}
    \end{tikzpicture}
    \caption{The cluster ensemble map for the Ising model. The edges with no weight indicated on the right-hand-side have weight $1$.}
    \label{fig:cluster_ensemble}
\end{figure}
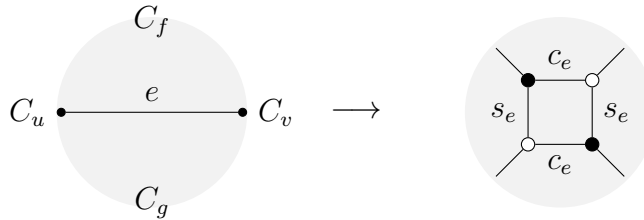

The following is a straightforward verification.

    \begin{proposition}
    If $G \to G'$ is a Y-$\Delta$ move, the following diagram commutes:
        \[
\begin{tikzcd}
\mathcal{C}_{G}^{>0}
  \arrow[r, "\sim"]
  \arrow[d, "q_{G}"']
& 
\mathcal{C}_{G'}^{>0}
  \arrow[d, "q_{G'}"']
\\
\mathcal{I}_{G}^{>0}
  \arrow[r, "\sim"]
&
\mathcal{I}_{G'}^{>0}
\end{tikzcd}.
\]
 \end{proposition}
 
Gluing the maps $q_G$ as $G$ varies over all reduced graphs with $\pi_G=\pi$, we get 
        \[
        q_\pi: \mathcal C_\pi^{>0} \rightarrow \mathcal I_\pi^{>0}.
        \]

\subsection{Ising twist map}    \label{sec:mod_twist}
Let $\mathscr{M}$ be a positroid and let $\Gamma$ be a reduced plabic graph with $\mathscr{M}_\Gamma=\mathscr{M}$. We also define a modification
\[
\hat p_{\mathscr{M}}: \mathcal A_{\mathscr{M}}^{>0} \rightarrow \mathcal X_{\mathscr{M}}^{>0}.
\]
of the cluster ensemble map as follows. Let $A \in \mathcal A_{\Gamma}^{>0}$ and for an edge $e=bw$ with incident faces $f$ and $g$, let
\begin{equation}\label{eq:IsingClusterEnsembleMap}
\wt(e) := \begin{cases}1/{A_{f} A_{g}},  &\text{if $e$ is an internal edge},\\
1/{\sqrt{A_{f} A_{g}}},  &\text{if $e$ is a boundary edge},
\end{cases}    
\end{equation}

and define $\hat p_\Gamma(A):=[\wt]$. If $\Gamma \to \Gamma'$ is a move, then the following diagram commutes:
\[
\begin{tikzcd}
\mathcal{A}_{\Gamma}^{>0}
  \arrow[r, "\sim"]
  \arrow[d, "\hat p_{\Gamma}"']
& 
\mathcal{A}_{\Gamma'}^{>0}
  \arrow[d, "\hat p_{\Gamma'}"']
\\
\mathcal{X}_{\Gamma}^{>0}
  \arrow[r, "\sim"]
&
\mathcal{X}_{\Gamma'}^{>0}
\end{tikzcd}.
\]
Thus, the maps $(\hat p_\Gamma)_{\Gamma: \mathscr{M}_\Gamma = \mathscr{M}}$ glue together to give $\hat p_{\mathscr{M}}$.

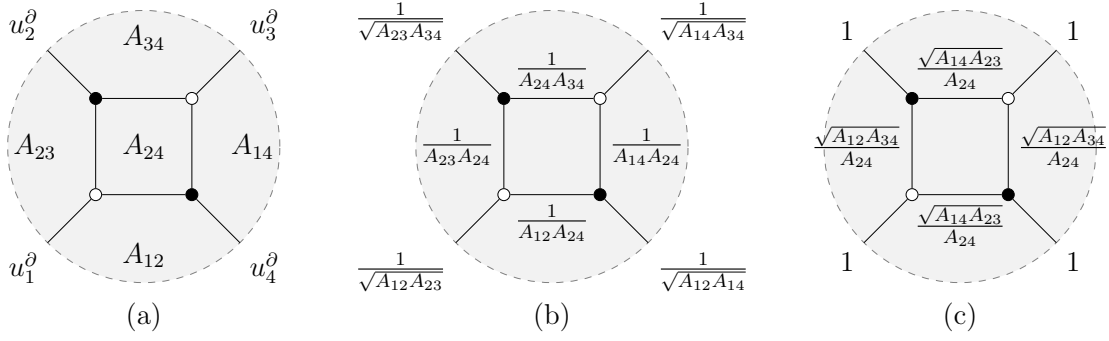
\begin{figure}[ht]
    \centering
    \scalebox{0.9}{
    \begin{tikzpicture}
            \begin{scope}[shift={(-1,0)}]
                \def\r{2};
                \fill[black!5] (0,0) circle (1*\r cm);
                      \draw[dashed, gray] (0,0) circle (\r cm);
                \coordinate[wvert] (w1) at (225:0.5*\r);
                \coordinate[wvert] (w2) at (45:0.5*\r);
                \coordinate[bvert] (b1) at (135:0.5*\r);
                \coordinate[bvert] (b2) at (315:0.5*\r);
    
                \coordinate[] (b1') at (225:\r);
                \coordinate[] (b2') at (45:\r);
                \coordinate[] (w1') at (135:\r);
                \coordinate[] (w2') at (315:\r);
    
                \draw[] (w1) --  (b1) -- (w2) -- (b2) -- (w1)
                (w1) --  (b1')
                (w2) -- (b2')
                (b1) -- (w1')
                (b2) --  (w2');
    
                \node at (0,0) {$A_{24}$};
                \node at (0:0.8*\r) {$A_{14}$};
                \node at (90:0.8*\r) {$A_{34}$};
                \node at (180:0.8*\r) {$A_{23}$};
                \node at (270:0.8*\r) {$A_{12}$};
    
                \node at (0,-2.5) {(a)};
                \coordinate[ label = below left: $u_1^\partial$] (b1') at (225:\r);
                \coordinate[ label = above right: $u_3^\partial$] (b2') at (45:\r);
                \coordinate[ label = above  left: $u_2^\partial$] (w1') at (135:\r);
                \coordinate[ label = below right: $u_4^\partial$] (w2') at (315:\r);
            \end{scope}

        \begin{scope}[shift={(5,0)}]
                \def\r{2};
                \fill[black!5] (0,0) circle (1*\r cm);
                      \draw[dashed, gray] (0,0) circle (\r cm);
                \coordinate[wvert] (w1) at (225:0.5*\r);
                \coordinate[wvert] (w2) at (45:0.5*\r);
                \coordinate[bvert] (b1) at (135:0.5*\r);
                \coordinate[bvert] (b2) at (315:0.5*\r);
    
                 \coordinate[ label = below left: $\frac{1}{\sqrt{A_{12}A_{23}}}$] (b1') at (225:\r);
                \coordinate[ label = above right: $\frac{1}{\sqrt{A_{14}A_{34}}}$] (b2') at (45:\r);
                \coordinate[ label = above left:$\frac{1}{\sqrt{A_{23}A_{34}}}$] (w1') at (135:\r);
                \coordinate[ label = below right: $\frac{1}{\sqrt{A_{12}A_{14}}}$] (w2') at (315:\r);

                \draw[] (w1) -- node[left]{$\frac{1}{A_{23} A_{24}}$} 
                        (b1) -- node[above]{$\frac{1}{A_{24} A_{34}}$} 
                        (w2) -- node[right]{$\frac{1}{A_{14} A_{24}}$} 
                        (b2) -- node[below]{$\frac{1}{A_{12} A_{24}}$} 
                        (w1)
                        (w1) --  (b1')
                        (w2) -- (b2')
                        (b1) -- (w1')
                        (b2) --  (w2');
                \node at (0,-2.5) {(b)};
        \end{scope}

        \begin{scope}[shift={(11,0)}]
                \def\r{2};
                \fill[black!5] (0,0) circle (1*\r cm);
                    \draw[dashed, gray] (0,0) circle (\r cm);
                \coordinate[wvert] (w1) at (225:0.5*\r);
                \coordinate[wvert] (w2) at (45:0.5*\r);
                \coordinate[bvert] (b1) at (135:0.5*\r);
                \coordinate[bvert] (b2) at (315:0.5*\r);
    
                \coordinate[label = below left: $1$] (b1') at (225:\r);
                \coordinate[label = above right: $1$] (b2') at (45:\r);
                \coordinate[label = above left:$1$] (w1') at (135:\r);
                \coordinate[label = below right: $1$] (w2') at (315:\r);

                \draw[] (w1) -- node[left]{$\frac{\sqrt{A_{12} A_{34}}}{A_{24}}$}              (b1) --node[above]{$\frac{\sqrt{A_{14} A_{23}}}{A_{24}}$} 
                        (w2) --node[right]{$\frac{\sqrt{A_{12} A_{34}}}{A_{24}}$} 
                        (b2) --node[below]{$\frac{\sqrt{A_{14} A_{23}}}{A_{24}}$}
                        (w1)
                        (w1) -- (b1')
                        (w2) -- (b2')
                        (b1) -- (w1')
                        (b2) -- (w2');
                
                \node at (0,-2.5) {(c)};
        \end{scope}
        \end{tikzpicture}}
    \caption{$A \in \mathcal A_{G^\square}^{>0}$ (a), $\hat p_{G^\square}(A)$ (b) and $\hat p_{G^\square}(A)$ after a gauge transformation (c). }
    \label{fig:comm_Ising_cluster_ensemble}
\end{figure}

Since we have modified $p_\mathscr{M}$, we also need to modify $\rt$ to have a commutative diagram as in \cref{thm:Muller_Speyer}.

\begin{definition}[The Ising twist map]
We define the map
\[
\hat \tau_\mathscr{M} : \Pi_\mathscr{M}^{>0} \rightarrow \Pi_\mathscr{M}^{>0}, \qquad
\hat{\tau}_{\mathscr{M}}(V) :=
\left(
\sqrt{\frac{\Delta_{I_1}(V)}{\Delta_{I_2}(V)}},
\,\dots,\,
\sqrt{\frac{\Delta_{I_{2n}}(V)}{\Delta_{I_1}(V)}}
\right)\cdot \rt(V)
\]
where $\bm I= (I_1,\dots,I_{2n})$ is the Grassmann necklace of the positroid $\mathscr{M}$. We call this map the \emph{Ising twist map}.
\end{definition}

\begin{proposition}\label{prop:modified_twist_diagram}
Let $\Gamma$ be a reduced plabic graph with $\mathscr{M}_\Gamma=\mathscr{M}$. The following diagrams commute:
    \[
\begin{tikzcd}
\mathcal A_{\Gamma}^{>0} 
  \arrow[r, "\hat p_{\Gamma}"] 
& 
\mathcal X_{\Gamma}^{>0} 
  \arrow[d, "{\rm Meas}_{\Gamma}"] 
\\
\Pi_{\mathscr{M}}^{>0}
  \arrow[u, "\bm F_{\Gamma}"]   \arrow[r,  swap, "{\hat \tau_\mathscr{M}}"]
& 
\Pi_{\mathscr{M}}^{>0}
\end{tikzcd}, \qquad \begin{tikzcd}
\mathcal A_{\mathscr{M}}^{>0} 
  \arrow[r, "\hat p_{\mathscr{M}}"] 
& 
\mathcal X_{\mathscr{M}}^{>0} 
  \arrow[d, "{\rm Meas}_{\mathscr{M}}"] 
\\
\Pi_{\mathscr{M}}^{>0}
  \arrow[u, "\bm F_{\mathscr{M}}"]   \arrow[r,  swap, "{\hat \tau_\mathscr{M}}"]
& 
\Pi_{\mathscr{M}}^{>0}
\end{tikzcd}.
\]
\end{proposition}

\begin{proof}
    We show commutativity of the first diagram by showing that ${\rm Meas}_{\Gamma}^{-1} \circ \hat \tau_\mathscr{M} =  \hat p_\Gamma \circ \bm F_\Gamma$; as usual, the second diagram is obtained by gluing. Let $\bm \lambda = (\lambda_1,\dots,\lambda_{2n}) \in \Rpos^{2n}$ be given by \(
\lambda_i := 
\sqrt{\frac{\Delta_{I_{i}}(V)}{\Delta_{I_{i+1}}(V)}}.
\) We have by definition of $\hat \tau_\mathscr{M}$ and equivariance that
    \begin{align*}
    {\rm Meas}_{\Gamma}^{-1} \circ \hat \tau_\mathscr{M} (V) &=  {\rm Meas}_{\Gamma}^{-1}(\bm \lambda \cdot  \rt(V))\\
    &=\bm \lambda \cdot {\rm Meas}_{\Gamma}^{-1}(\rt(V))\\
    &= \bm \lambda \cdot (p_\Gamma \circ \bm F_\Gamma \circ \lt) \circ \rt(V)\\
    &=\bm \lambda \cdot p_\Gamma \circ \bm F_\Gamma(V).
    \end{align*}
Therefore, it suffices to show that $ \hat p_\Gamma \circ \bm F_\Gamma(V) = \bm \lambda \cdot p_\Gamma \circ \bm F_\Gamma(V)$. Clearly this is true at internal edges. We have the following two cases for boundary edges:
\begin{enumerate}
    \item $e$ is incident to $u_i^\partial$ and a white internal vertex: $\bm \lambda \cdot p_\Gamma \circ \bm F_\Gamma$ assigns the weight \[ \sqrt{\frac{\Delta_{I_{i}}(V)}{\Delta_{I_{i+1}}(V)}} \frac{1}{\Delta_{I_i}(V)} = \frac{1}{\sqrt{\Delta_{I_i}(V) \Delta_{I_{i+1}}(V)}}.\]
    \item $e$ is incident to $u_i^\partial$ and a black internal vertex: $\bm \lambda \cdot p_\Gamma \circ \bm F_\Gamma$ assigns the weight 
    \[
    \left( \sqrt{\frac{\Delta_{I_{i}}(V)}{\Delta_{I_{i+1}}(V)}}\right)^{-1}
    \frac{1}{\Delta_{I_{i+1}}(V)} = \frac{1}{\sqrt{\Delta_{I_i}(V) \Delta_{I_{i+1}}(V)}}. \qedhere
    \]
\end{enumerate}
    
\end{proof}

Now we return to the Ising model and relate the Ising twist map with the Ising cluster ensemble map. 

\begin{lemma}\label{lem:modified_twist_equals_qG}
Let $G$ be a reduced cactus graph with $\pi_G=\pi$. The following diagrams commute:
\[
\begin{tikzcd}
\mathcal C_{G}^{>0} 
  \arrow[r, "q_G"] \arrow[d, hook] 
& 
\mathcal I_{G}^{>0} 
  \arrow[d, hook] 
\\
\mathcal A_{G^\square}^{>0}
    \arrow[r,   "\hat p_{G^\square}"]
& 
\mathcal X_{G^\square}^{>0}
  \end{tikzcd}, \qquad \begin{tikzcd}
\mathcal C_{\pi}^{>0} 
  \arrow[r, "q_\pi"] \arrow[d, hook] 
& 
\mathcal I_{\pi}^{>0} 
  \arrow[d, hook] 
\\
\mathcal A_{\mathscr{M}_\pi}^{>0}
    \arrow[r,   "\hat p_{\mathscr{M}_\pi}"]
& 
\mathcal X_{\mathscr{M}_\pi}^{>0}
  \end{tikzcd}.
\]
\end{lemma}
\begin{proof}
    See \cref{fig:comm_Ising_cluster_ensemble} for the proof of commutativity of the first diagram; the second is obtained by gluing.
\end{proof}

\subsection{Group action and inverse Ising twist} \label{sec:torus_action}
We have $\dim \mathcal I_G^{>0}=|E(G)|$ and $\dim \mathcal C_G^{>0}=|V(G)|+|F(G)|-1$. Applying Euler's formula to the cell decomposition of $\disk$ given by $G$ together with the $n$ intervals that $(b_i^\partial)_{i=1}^n$ divide $\partial \disk$ into, we have
\(
|V(G)| - (|E(G)|+n) + |F(G)|=1,
\) so 
\[
\dim \mathcal C_G^{>0}-\dim \mathcal I_G^{>0}=n.
\] 
Hence, the map $\hat \tau_{\mathscr{M}_\pi}$ cannot be a bijection. 
We will see that there is an action of an $n$-dimensional group and $\hat \tau_{\mathscr{M}_\pi}$ becomes a bijection upon taking the quotient.

Given $\pi \in P_n$, consider the subgroup 
\[ 
T^{>0}_\pi := \Big\{ \bm \lambda=(\lambda_1,\dots,\lambda_{2n}):  \lambda_{\bar{i}} = \lambda_i^{-1} \text{ for all $i \in [2n]$} \Big\} \cong \Rpos^n
\]
of $\Rpos^{2n}$ with the induced action.

\begin{lemma}
The action of the torus $T_\pi^{>0}$ stabilizes $\mathcal C_\pi^{>0} \subset \mathcal A_{\mathscr{M}_\pi}^{>0}$.
\end{lemma}
\begin{proof}
    {The locally Lagrangian cell $\Lambda_\pi^{>0}$ is cut out in $\Pi_{\mathscr{M}_\pi}$ by the equations  \eqref{eq:cond_plucker_equal}. So the action of $T_\pi^{>0}$ clearly stabilizes $\Lambda_{\pi}^{>0}$. The result then follows from the fact that $H_{\pi} \colon \Lambda_{\pi}^{>0} \to \mathcal{C}_\pi^{>0}$ is a bijection and it is equivariant under the action of $T_{\pi}^{>0}$.}
\end{proof}

\begin{lemma}
    {The map $q_\pi \colon \mathcal{C}_{\pi}^{>0} \to I_{\pi}^{>0}$ is $T^{>0}_\pi$-invariant.}
\end{lemma}
\begin{proof}
Compare the target face labels corresponding to the $A$ variables in the numerator and denominator of $c$ and $s$ in~\eqref{eq:c_and_s_from_A}.
\end{proof}

By \cref{lem:modified_twist_equals_qG} and \cref{prop:modified_twist_diagram}, the induced map $q_G: \mathcal C_G^{>0}/T_\pi^{>0} \rightarrow \mathcal I_G^{>0}$ sits in the commutative diagram:
\[
\begin{tikzcd}
\mathcal C_{G}^{>0} /T_{\pi}^{>0}
  \arrow[r, "q_G"] 
& \mathcal I_{G}^{>0} 
  \arrow[d, "\phi_n \circ \operatorname{Corr}_{G}"] 
\\
\Lambda_{\pi}^{>0}/T_\pi^{>0}
  \arrow[u, "\bm H_{G}"]   \arrow[r,  swap, "{\hat \tau_{{\mathscr{M}}_\pi}}"]
& 
\OGr^{\ge 0}_{n,2n} \cap \Pi_{\mathscr{M}_\pi}^{>0}
\end{tikzcd}
\]
and similarly, by gluing, for the induced map ${q}_\pi: \mathcal C_\pi^{>0}/T^{>0}_{\pi} \rightarrow \mathcal I_\pi^{>0} $. The goal of this section is to prove that the maps $q_G, q_\pi, \hat \tau_{\mathscr{M}_\pi}$ are bijections. We do this by constructing an inverse map $\hat \tau^{-1}_{\mathscr{M}_\pi}$ as a modification of the left twist.

Given $[\wt] \in \mathcal X_{G^\square}$, let $m_i([\wt])$ denote the alternating product of edge weights along the strand with target $u_i^\partial$ where the edge weight $\wt(bw)$ appears in the numerator if $bw$ is traversed by the strand from $b$ to $w$ and in the denominator otherwise.

\begin{lemma} \label{lem:zzpath_from_A}
Let $[\wt] = \hat p_{G^\square} (A) $ for some $A \in \mathcal A_{G^\square}^{>0}$. Then, 
\begin{equation}\label{eq:miwt}
m_i([\wt]) = \sqrt{\frac{A_{f_{i,i+1}^\partial}A_{f_{\bar i,\bar i+1}^\partial}}{A_{f_{ i-1, i}^\partial}A_{f_{\bar i-1,\bar i}^\partial}}}.
\end{equation}
\end{lemma}
{
\begin{proof}

Let $e_1,\ldots,e_r$ be the edges, in order, along the strand from the vertex $u_{\bar{i}}^\partial$ to $u_i^\partial$, and let $f_j$ (respectively $g_j$) be the face to the left (respectively right) of the edge $e_j$ for $j=1,\ldots,r$.

For $j=1,\ldots,r$, let
\[
\varepsilon_j=
\begin{cases}
-1,&\text{if $e_j$ is traversed from black to white},\\
+1,&\text{if $e_j$ is traversed from white to black}.
\end{cases}
\]
Since the colors alternate along the strand, we have
\[
\varepsilon_{j+1}=-\varepsilon_j
\qquad\text{for }j=1,\ldots,r-1.
\]
By definition of the map $\hat p_{G^\square}$, we therefore have
\begin{equation}\label{eq:strand_product}
m_i([\mathrm{wt}])
=
(A_{f_1}A_{g_1})^{\varepsilon_1/2}
\left(\prod_{j=2}^{r-1}(A_{f_j}A_{g_j})^{\varepsilon_j}\right)
(A_{f_r}A_{g_r})^{\varepsilon_r/2}.
\end{equation}
Note that
\[
    f_1=f_{\bar{i},\bar{i}+1}^\partial,\quad
    g_1=f_{\bar{i}-1,\bar{i}}^\partial, 
    \quad \text{and} \quad
    f_r=f_{i-1,i}^\partial, \quad
    g_r=f_{i,i+1}^\partial.
\]
Moreover, for any $j=1,\ldots,r-1$, if $e_j$ and $e_{j+1}$ meet at a white vertex, then $f_j=f_{j+1}$, while if they meet at a black vertex, then $g_j=g_{j+1}$. We now simplify the product along the strand by telescoping. To see how, we distinguish the following cases:

\begin{itemize}
    \item First suppose that $e_1$ and $e_2$ meet at a black vertex. Then
\[
g_1=g_2,\qquad \varepsilon_1=1,\qquad \varepsilon_2=-1,
\]
and hence
\[
(A_{f_1}A_{g_1})^{1/2}
(A_{f_2}A_{g_2})^{-1}
=
\sqrt{\frac{A_{f_1}}{A_{g_1}}}\,A_{f_2}^{-1}.
\]
Since the next vertex is white, we have $f_2=f_3$ and $\varepsilon_3=1$, so multiplying by the contribution of $e_3$ gives
\[
\sqrt{\frac{A_{f_1}}{A_{g_1}}}\,
A_{f_2}^{-1}
(A_{f_3}A_{g_3})
=
\sqrt{\frac{A_{f_1}}{A_{g_1}}}\,A_{g_3}.
\]
At the next vertex we have $g_3=g_4$ and $\varepsilon_4=-1$, and therefore
\[
\sqrt{\frac{A_{f_1}}{A_{g_1}}}\,
A_{g_3}(A_{f_4}A_{g_4})^{-1}
=
\sqrt{\frac{A_{f_1}}{A_{g_1}}}\,A_{f_4}^{-1}.
\]
Continuing in this way, the remaining terms in the product (in the middle of \eqref{eq:strand_product}) alternate between
\[
A_{f_j}^{-1}\qquad\text{and}\qquad A_{g_j}.
\]

\item 
On the other hand, suppose that $e_1$ and $e_2$ meet at a white vertex. Then
\[
f_1=f_2,\qquad \varepsilon_1=-1,\qquad \varepsilon_2=1,
\]
and
\[
(A_{f_1}A_{g_1})^{-1/2}
(A_{f_2}A_{g_2})
=
\sqrt{\frac{A_{f_1}}{A_{g_1}}}\,A_{g_2}.
\]
The next vertex is black, so $g_2=g_3$ and $\varepsilon_3=-1$. Thus
\[
\sqrt{\frac{A_{f_1}}{A_{g_1}}}\,
A_{g_2}(A_{f_3}A_{g_3})^{-1}
=
\sqrt{\frac{A_{f_1}}{A_{g_1}}}\,A_{f_3}^{-1}.
\]
Again, continuing along the strand, the remaining full power alternates between
\[
A_{g_j}\qquad\text{and}\qquad A_{f_j}^{-1}.
\]
\end{itemize}

Therefore, independently of the color of the internal vertex adjacent to $u_{\bar i}^\partial$, after successively simplifying the factors along the strand we obtain the common initial factor
\[
\sqrt{\frac{A_{f_1}}{A_{g_1}}},
\]
followed by a telescoping product until the end point factor. We now consider the two different cases to determine the endpoint factor.
\begin{itemize}
    \item If $e_{r-1}$ and $e_r$ meet at a black vertex, then
        \[
        g_{r-1}=g_r,\qquad
        \varepsilon_{r-1}=1,\qquad
        \varepsilon_r=-1.
        \]
        The factor propagated from the preceding edges is $A_{g_{r-1}}$, and hence the last step gives
        \[
        A_{g_{r-1}}(A_{f_r}A_{g_r})^{-1/2}
        =
        \sqrt{\frac{A_{g_r}}{A_{f_r}}}.
        \]
        
    \item If instead $e_{r-1}$ and $e_r$ meet at a white vertex, then
            \[
            f_{r-1}=f_r,\qquad
            \varepsilon_{r-1}=-1,\qquad
            \varepsilon_r=1.
            \]
            In this case the factor propagated from the preceding edges is $A_{f_{r-1}}^{-1}$, and hence
            \[
            A_{f_{r-1}}^{-1}(A_{f_r}A_{g_r})^{1/2}
            =
            \sqrt{\frac{A_{g_r}}{A_{f_r}}}.
            \]

\end{itemize}

Thus both possible colors at the terminal endpoint give the same contribution. It then follows that
\[
m_i([\mathrm{wt}])
=
\sqrt{\frac{A_{f_1}}{A_{g_1}}}
\sqrt{\frac{A_{g_r}}{A_{f_r}}}
=
\sqrt{\frac{A_{f_1}A_{g_r}}{A_{g_1}A_{f_r}}}.
\]
Using the identifications of the boundary faces above, we obtain~\eqref{eq:miwt} as desired.
\end{proof}
}

 \begin{remark}
    Let $G$ be a reduced graph with $\pi_G=\pi$, let $G^\square$ be the corresponding plabic graph and let $[\wt] = {\rm Meas}^{-1}_{G^\square}(V)$. Then, \eqref{eq:lambda_modified_lt} is also equal to $1/{m_i([\wt])}$, which can be proved by a similar computation as in the proof of \cref{lem:zzpath_from_A} applied to $p_\Gamma$ instead of $\hat p_\Gamma$ and using \cref{lem:MS}(2).  
 \end{remark}

{
\begin{lemma}\label{lem:compatibility_of_ratios}
    Let $\pi$ be a matching on $[2n]$ and let $\bm{I} = (I_1, \dots, I_{2n})$ be the Grassmann necklace corresponding to the positroid $\mathscr{M}_\pi$. Then for any $V \in \OGr_{n,2n}^{\geq 0} \cap \Pi_{\mathscr{M}_\pi}^{>0}$ we have
    \[
       \frac{\Delta_{I_{i+1}}(V)}{\Delta_{I_{i}}(V)} = \frac{\Delta_{I_{\bar{i}+1}}(V)}{\Delta_{I_{\bar{i}}}(V)}, \quad \text{for any } i \in [2n].
    \]
\end{lemma}
\begin{proof}
    Let $G$ be a reduced cactus graph with $\pi_G = \pi$. By Theorem~3.3,
    there exists a coupling constant $J$ on $G$ such that
    \[
        \phi_n(M(G,J)) = V.
    \]
    Let $[\mathrm{wt}]$ be the dimer model associated to $(G,J)$, equivalently
    \[
        [\mathrm{wt}]
        :=
        {\rm Meas}_{G^\square}^{-1}(V)
        \in
        \mathcal X_{G^\square}^{>0}.
    \]

    Recall that $m_i([\mathrm{wt}])$ denotes the alternating product of
    edge weights along the strand $\gamma_{\bar{i}, i}$ of $G^\square$ with target
    $u_i^\partial$. Using \cref{lem:MS}(2), a computation analogous to the proof of Lemma~5.8, applied to the cluster ensemble map $p_{G^\square}$ instead of $\hat p_{G^\square}$, gives
    \[
        \frac{\Delta_{I_{i+1}}(V)}
             {\Delta_{I_i}(V)}
        =
        \frac{1}{m_{\bar{i}}([\mathrm{wt}])},\qquad
        \frac{\Delta_{I_{\bar i+1}}(V)}
             {\Delta_{I_{\bar i}}(V)}
        =
        \frac{1}{m_{i}([\mathrm{wt}])}.
    \]

    By construction of $G^\square$, the strands with targets
    $u_i^\partial$ and $u_{\bar i}^\partial$ are the two strands of
    $G^\square$ corresponding to the same strand of $G$. We claim that
    \[
        m_i([\mathrm{wt}])=m_{\bar{i}}([\mathrm{wt}]).
    \]
    To see this, recall that the strands of $G^\square$ with targets
    $u_i^\partial$ and $u_{\bar{i}}^\partial$ are the two strands
    corresponding to the same strand of the medial graph $G^\times$.
    It is therefore enough to compare their contributions locally in
    each square of $G^\square$. Let $e$ be an edge of $G$. In the square of $G^\square$ corresponding to $e$, the two opposite edges with weight $c_e$
    have the same weight, and the two opposite edges with weight $s_e$
    have the same weight, and all the edges leaving the square have weight
    $1$. The two corresponding strands traverse opposite edges of this
    square. So, if one of the strands traverses an edge of weight
    $c_e$, then the other traverses the opposite edge, which also has
    weight $c_e$. Similarly, if one traverses an edge of weight $s_e$,
    then the other traverses the opposite edge, which also has weight
    $s_e$. 
    
    Therefore we obtain
    \[
        \frac{\Delta_{I_{i+1}}(V)}
             {\Delta_{I_i}(V)}
        =
        \frac{1}{m_{\bar{i}}([\mathrm{wt}])}
        =
        \frac{1}{m_{i}([\mathrm{wt}])}
        =
        \frac{\Delta_{I_{\bar i+1}}(V)}
             {\Delta_{I_{\bar i}}(V)},
    \]
    as desired.
\end{proof}
}

We now define the inverse of the Ising twist map.

\begin{definition}[Inverse Ising twist map] \label{def:inverseIsingTwist}
We define the \emph{inverse Ising twist map} as follows:
\[
\hat \tau^{-1}_{\mathscr{M}_\pi}:\OGr^{\ge 0}_{n,2n} \cap \Pi_{\mathscr{M}_\pi}^{>0} \rightarrow \Pi_{\mathscr{M}_\pi}^{>0}/T_{\pi}^{>0}, \quad V \mapsto \bm \lambda \cdot {\lvec{\tau}_{\mathscr{M}_\pi}}(V),
\]
where $\bm \lambda \in \Rpos^{2n}$ is any collection of numbers satisfying 
\begin{equation} \label{eq:lambda_modified_lt}
    \lambda_i \lambda_{\bar i} =  \frac{\Delta_{I_{i+1}}(V)}{\Delta_{I_{i}}(V)}  \qquad \text{for all } i \in [2n].
\end{equation}
Note that a solution to \eqref{eq:lambda_modified_lt} always exists thanks to Lemma \ref{lem:compatibility_of_ratios}. Moreover, different choices of a solution $\bm \lambda$ are related by the action of $T^{>0}_\pi$ and therefore this definition of $\hat \tau^{-1}_{\mathscr{M}_\pi}(V)$ is independent of the choice of $\bm{\lambda}$.
\end{definition}

\begin{figure}[H]
    \centering
    \scalebox{0.6}{    
    \begin{tikzpicture}
        \begin{scope}[shift={(1,2)}, scale = 0.75]
                    \def\r{2};
                    \fill[black!5] (0,0) circle (2*\r cm);
                     \draw[dashed, gray] (0,0) circle (2*\r cm);
                    \coordinate[wvert] (w1) at (225:0.5*\r);
                    \coordinate[wvert] (w2) at (45:0.5*\r);
                    \coordinate[bvert] (b1) at (135:0.5*\r);
                    \coordinate[bvert] (b2) at (315:0.5*\r);
        
                    \coordinate[] (b1') at (225:\r);
                    \coordinate[] (b2') at (45:\r);
                    \coordinate[] (w1') at (135:\r);
                    \coordinate[] (w2') at (315:\r);
        
                    \draw[] (w1) --  (b1) -- (w2) -- (b2) -- (w1)
                            (w1) --  (b1')
                            (w2) -- (b2')
                            (b1) -- (w1')
                            (b2) --  (w2') ;
        
                    \node at (0,0) {$S i \bar i$};
                    \node at (0:0.8*\r) {$S \bar i \bar j$};
                    \node at (90:0.8*\r) {$S \bar i j$};
                    \node at (180:0.8*\r) {$Sij$};
                    \node at (270:0.8*\r) {$S i \bar j$};
        
                    \draw[red,<-,line width=\lw,rounded corners=\rc] (225:4) --  (220:3) --  (235:1) -- (-45:1-\ep) -- (35:1) -- (50:3) -- (45:4); 
                    \node at (225:4.5) {$u_{\bar j}^\partial$};
    
                \begin{scope}[rotate=90]
                    \draw[yellow!80!red,->,line width=\lw,rounded corners=\rc]
                    (225:4) --  (220:3) --  (235:1) -- (-45:1-\ep) -- (35:1) -- (50:3) -- (45:4); 
                    \node at (225:4.5) {$u_{\bar i}^\partial$};
                \end{scope}
    
                \begin{scope}[rotate=180]
                    \draw[blue,<-,line width=\lw,rounded corners=\rc]
                    (225:4) --  (220:3) --  (235:1) -- (-45:1-\ep) -- (35:1) -- (50:3) -- (45:4); 
                    \node at (225:4.5) {$u_j^\partial$};
                \end{scope}
    
                \begin{scope}[rotate=270]
                    \draw[green!80!black,->,line width=\lw,rounded corners=\rc]
                    (225:4) --  (220:3) --  (235:1) -- (-45:1-\ep) -- (35:1) -- (50:3) -- (45:4); 
                    \node at (225:4.5) {$u_{ i}^\partial$};
                \end{scope}
                
                \node at (0,-5) {(a)};
        \end{scope}

        \begin{scope}[shift={(8,2)}, scale = 0.75]
                \def\r{2};
                \fill[black!5] (0,0) circle (2*\r cm);
                \draw[dashed, gray] (0,0) circle (2*\r cm);
                \coordinate[wvert] (w1) at (225:0.5*\r);
                \coordinate[wvert] (w2) at (45:0.5*\r);
                \coordinate[bvert] (b1) at (135:0.5*\r);
                \coordinate[bvert] (b2) at (315:0.5*\r);
    
                \coordinate[] (b1') at (225:\r);
                \coordinate[] (b2') at (45:\r);
                \coordinate[] (w1') at (135:\r);
                \coordinate[] (w2') at (315:\r);
    
                \draw[] (w1) --  (b1) -- (w2) -- (b2) -- (w1)
                (w1) --  (b1')
                (w2) -- (b2')
                (b1) -- (w1')
                (b2) --  (w2')
                ;
    
                \node at (0,0) {$S i \bar i$};
                \node at (0:0.8*\r) {$S \bar i \bar j$};
                \node at (90:0.8*\r) {$S \bar i j$};
                \node at (180:0.8*\r) {$Sij$};
                \node at (270:0.8*\r) {$S i \bar j$};
    
                \node at (225:4.5) {$u_{\bar j}^\partial$};
                \draw[brown,->,line width=\lw,rounded corners=\rc] (-45:4) --  (-40:3) --  (-55:1) -- (225:1-\ep) -- (180-45:1-\ep) --  (55:1) -- (40:3) -- (45:4); 
                
                \begin{scope}[rotate=90]
                    \node at (225:4.5) {$u_{\bar i}^\partial$};
                \end{scope}
    
                \begin{scope}[rotate=180]
                            \node at (225:4.5) {$u_j^\partial$};
                \end{scope}
                
                \begin{scope}[rotate=270]
                
                            \node at (225:4.5) {$u_{ i}^\partial$};
                \end{scope}
              
                \node at (0,-5) {(b)};
        \end{scope}

         \begin{scope}[shift={(1,-6)}, scale = 0.75]
                    \def\r{2};
                    \fill[black!5] (0,0) circle (2*\r cm);
                    \draw[dashed, gray] (0,0) circle (2*\r cm);
                    \coordinate[bvert] (w1) at (225:0.5*\r);
                    \coordinate[bvert] (w2) at (45:0.5*\r);
                    \coordinate[wvert] (b1) at (135:0.5*\r);
                    \coordinate[wvert] (b2) at (315:0.5*\r);
        
                    \coordinate[] (b1') at (225:\r);
                    \coordinate[] (b2') at (45:\r);
                    \coordinate[] (w1') at (135:\r);
                    \coordinate[] (w2') at (315:\r);

                    \draw[] (w1) --  (b1) -- (w2) -- (b2) -- (w1)
                    (w1) --  (b1')
                    (w2) -- (b2')
                    (b1) -- (w1')
                    (b2) --  (w2')
                    ;
        
                    \node at (0,0) {$S j \bar j$};
                    \node at (0:0.8*\r) {$S \bar i \bar j$};
                    \node at (90:0.8*\r) {$S \bar i j$};
                    \node at (180:0.8*\r) {$Sij$};
                    \node at (270:0.8*\r) {$S i \bar j$};
        
                    \draw[blue,->,line width=\lw,rounded corners=\rc] (225:4) --  (220:3) --  (235:1) -- (-45:1-\ep) -- (35:1) -- (50:3) -- (45:4); 
                    \node at (225:4.5) {$u_{\bar j}^\partial$};
        
                    \begin{scope}[rotate=90]
                        \draw[green!80!black,<-,line width=\lw,rounded corners=\rc]
                        (225:4) --  (220:3) --  (235:1) -- (-45:1-\ep) -- (35:1) -- (50:3) -- (45:4); 
                        \node at (225:4.5) {$u_{\bar i}^\partial$};
                    \end{scope}
        
                    \begin{scope}[rotate=180]
                        \draw[red,->,line width=\lw,rounded corners=\rc]
                        (225:4) --  (220:3) --  (235:1) -- (-45:1-\ep) -- (35:1) -- (50:3) -- (45:4); 
                        \node at (225:4.5) {$u_j^\partial$};
                    \end{scope}
        
                    \begin{scope}[rotate=270]
                        \draw[yellow!80!red,<-,line width=\lw,rounded corners=\rc]
                        (225:4) --  (220:3) --  (235:1) -- (-45:1-\ep) -- (35:1) -- (50:3) -- (45:4); 
                        \node at (225:4.5) {$u_{ i}^\partial$};
                    \end{scope}

                    \node at (0,-5) {(c)};
         \end{scope}
                
         \begin{scope}[shift={(8,-6)}, scale = 0.75]
                    \def\r{2};
                    \fill[black!5] (0,0) circle (2*\r cm);
                    \draw[dashed, gray] (0,0) circle (2*\r cm);
                    \coordinate[bvert] (w1) at (225:0.5*\r);
                    \coordinate[bvert] (w2) at (45:0.5*\r);
                    \coordinate[wvert] (b1) at (135:0.5*\r);
                    \coordinate[wvert] (b2) at (315:0.5*\r);
        
                    \coordinate[] (b1') at (225:\r);
                    \coordinate[] (b2') at (45:\r);
                    \coordinate[] (w1') at (135:\r);
                    \coordinate[] (w2') at (315:\r);

                    \draw[] (w1) --  (b1) -- (w2) -- (b2) -- (w1)
                    (w1) --  (b1')
                    (w2) -- (b2')
                    (b1) -- (w1')
                    (b2) --  (w2')
                    ;
        
                    \node at (0,0) {$S j \bar j$};
                    \node at (0:0.8*\r) {$S \bar i \bar j$};
                    \node at (90:0.8*\r) {$S \bar i j$};
                    \node at (180:0.8*\r) {$Sij$};
                    \node at (270:0.8*\r) {$S i \bar j$};
        
                    \node at (225:4.5) {$u_{\bar j}^\partial$};
                    \draw[brown,->,line width=\lw,rounded corners=\rc] (225:4) --  (230:3) --  (215:1)  --  (180-45+10:1) -- (180-45-5:3) -- (180-45:4); 
                    
                    \begin{scope}[rotate=90]
                        \node at (225:4.5) {$u_{\bar i}^\partial$};
                    \end{scope}
        
                    \begin{scope}[rotate=180]
                        \node at (225:4.5) {$u_j^\partial$};
                    \end{scope}
        
                    \begin{scope}[rotate=270]
                        \node at (225:4.5) {$u_{ i}^\partial$};
                    \end{scope}

                    \node at (0,-5) {(d)};
        \end{scope}        
\end{tikzpicture}}

    \caption{Configuration of strands around a square in $G^\square$ and $\bar {G^\square}$ in the proof of \cref{lem:hat_tau_inverse_well_defined}.}
    \label{fig:paths}
\end{figure}
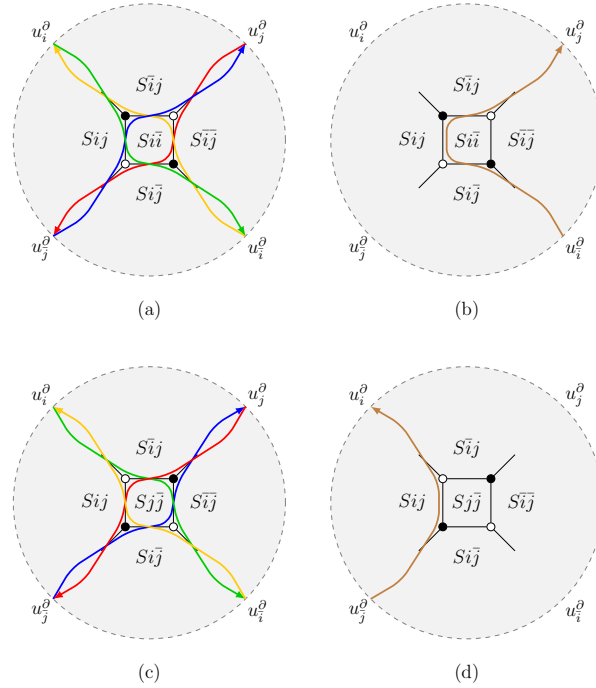

\begin{lemma} \label{lem:hat_tau_inverse_well_defined}
    We have $\hat \tau^{-1}_{\mathscr{M}_\pi} \Big(\OGr^{\ge 0}_{n,2n} \cap \Pi_{\mathscr{M}_\pi}^{>0}\Big) \subset \Lambda_\pi^{>0}/T_\pi^{>0}$.
\end{lemma}

{
\begin{proof}
Let $V \in \OGr^{\ge 0}_{n,2n} \cap \Pi_{\mathscr M_\pi}^{>0}$ and let $[U] = \hat\tau^{-1}_{\mathscr M_\pi}(V)$.
Since the action of $T_\pi^{>0}$ preserves $\Lambda_\pi^{>0}$, it is enough to check that one representative $U$ of $[U]$ satisfies \eqref{eq:cond_plucker_equal}. Fix $S i\bar i\in\mathcal S_\pi^\square$, and choose a reduced cactus graph $G$ with $\pi_G = \pi$ such that $S i\bar i$ appears as the target label of a square face of $G^\square$. Let $j,\bar j$ be the other pair of strand labels meeting at this square, as in \cref{fig:paths}.

By Theorem~3.3, there exists a coupling constant $J$ on $G$ such that
$
\phi_n(M(G,J))=V.
$
Set
$
    [\wt]:={\rm Meas}_{G^\square}^{-1}(V)
\in\mathcal X_{G^\square}^{>0}.
$
Thus $[\wt]$ is the dimer model associated to $(G,J)$ and in particular lies in the image of $\mathcal I_G^{>0}$.

We use the description of the inverse
\[
\lvec{\mathbb M}_{G^\square}=p_{G^\square}^{-1}:
\mathcal X_{G^\square}^{>0}
\rightarrow
\mathcal A_{G^\square}^{>0}.
\]
of the cluster ensemble map in terms of almost perfect matchings due to Muller--Speyer~\cite{MullerSpeyer}. For each face $f$ of
$G^\square$, let $M_f$ be the almost perfect matching with
$\partial M_f = \bm T(f)$ defined as follows. If an edge $e = bw$ is
crossed by the strands $\gamma_{\bar k, k},\gamma_{\bar l, l}$ as in
\cref{fig:loc_config_strands_e=bw}, then
\[
e\in M_f
\qquad \text{if and only if} \qquad
k\in\bm T(f)
\ \text{and}\
l\notin\bm T(f).
\]
Equivalently, $e\in M_f$ if and only if $f$ lies in the ``upstream wedge''
determined by $\gamma_{\bar k, k}$ and $\gamma_{\bar l, l}$; see \cite[Figure~4]{MullerSpeyer}. We then have
\[
\lvec{\mathbb M}_{G^\square}([\wt])
=
\left(\frac{1}{\wt(M_f)}\right)_{f\in F(G^\square)}.
\]

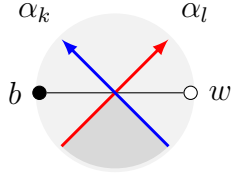
\begin{figure}[H]
\begin{center}
    \begin{tikzpicture}
     \begin{scope}[shift={(-4,0)}]
      \def\r{1};
            \fill[black!5] (0,0) circle (1.05*\r cm);
            \path[fill=black!15, draw=none] (0,0) -- (225:\r) arc[start angle=225, delta angle=90, radius=\r] -- cycle;
            \coordinate[wvert, label = right: $w$] (n1) at (0:\r);
            \coordinate[bvert, label = left: $b$] (n2) at (180:\r);
            
            \coordinate[label=center: $\gamma_{\bar k, k}$] (no) at (180-45:1.5*\r);
            \coordinate[label = center: $\gamma_{\bar l, l}$] (no) at (45:1.5*\r);
            \draw[] (n1) -- (n2);
            
          \draw[red,->,line width=\lw,rounded corners=\rc]  (225:\r) -- (45:\r);
          \draw[blue,->,line width=\lw,rounded corners=\rc]  (-45:\r) -- (180-45:\r);

        \end{scope}
        \end{tikzpicture}
\end{center}
\caption{Local configuration of strands near an edge $e=bw$.} \label{fig:loc_config_strands_e=bw}
\end{figure}

Let $A:=\lvec{\mathbb M}_{G^\square}([\wt])$. By \cref{thm:Muller_Speyer} we have
$ {\rm Meas}_{\mathscr M_\pi}^{-1} = p_{\mathscr M_\pi} \circ \bm{F}_{\mathscr M_\pi} \circ \lvec{\tau}_{\mathscr M_\pi}$.
Since $[\wt] = {\rm Meas}_{G^\square}^{-1}(V)$, the $G^\square$-component of this identity gives $A = \bm{F}_{G^\square} \bigl(\lvec{\tau}_{\mathscr M_\pi}(V)\bigr)$.
Therefore we obtain
\begin{equation}\label{eq:A_twisted_plucker}
    A_f
    =
    \Delta_{\bm T(f)}
    \bigl(\lvec{\tau}_{\mathscr M_\pi}(V)\bigr),
    \qquad
    \text{for every face } f \in F(G^\square).
\end{equation}

We now compare the matchings corresponding to the faces around the chosen square. We denote by $f_{kl}$ the face with target label $Skl$.
Since
\[
\partial M_{f_{i\bar i}}=S i\bar i,
\qquad
\partial M_{f_{ij}}=Sij,
\]
the two almost perfect matchings have the same boundary behavior at every boundary vertex except $u_{\bar i}^\partial$ and $u_j^\partial$. Therefore their symmetric difference consists of a path joining $u_{\bar i}^\partial$ to $u_j^\partial$, possibly together with some disjoint cycles.

By the definition of the matchings $M_f$, changing the target label from $S i\bar i$ to $Sij$ changes the membership of an edge only when the edge is crossed by $\gamma_{i,\bar i}$ or $\gamma_{\bar j,j}$. Along these two strands, the matchings differ precisely on the segments joining the chosen square to $u_{\bar i}^\partial$ and $u_j^\partial$. Hence their symmetric difference is exactly the path shown in \cref{fig:paths}(b), and in particular there are no additional cycle components.

Now since $A_f = 1 / \wt(M_f)$, the ratio
\[
\frac{A_{f_{ij}}}{A_{f_{i\bar i}}}
=
\frac{\wt(M_{f_{i\bar i}})}
     {\wt(M_{f_{ij}})}
\]
is the alternating product of the edge weights along this path, with the edges belonging to $M_{f_{i\bar i}}$ occurring in the numerator and those belonging to $M_{f_{ij}}$ occurring in the denominator.

We next apply the same argument to the color-changed graph $\overline{G^\square}$. Let $\overline{\wt}$ denote the same edge weights on $\overline{G^\square}$, and set
\[
\overline A
:=
\lvec{\mathbb M}_{\overline{G^\square}}
([\overline{\wt}]).
\]
Under boundary measurement, color-change sends $\Delta_I$ to $\Delta_{[2n]\setminus I}$. Since $V \in \OGr_{n,2n}^{\geq 0}$, these Pl\"ucker coordinates are equal, and therefore ${\rm Meas}_{\overline{G^\square}}([\overline{\wt}])=V$. Applying \cref{thm:Muller_Speyer} to $\overline{G^\square}$ gives $\overline A = \bm F_{\overline{G^\square}}\bigl(\lvec{\tau}_{\mathscr M_\pi}(V)\bigr)$. In particular, near the square in \cref{fig:paths}(c), we have
\[
\overline A_{f_{j\bar j}}
=
\Delta_{S j\bar j}
\bigl(\lvec{\tau}_{\mathscr M_\pi}(V)\bigr).
\]

Since $[\wt]$ lies in the image of $\mathcal I_G^{>0}$, by \cref{prop:George_characterization} color-change gives the same gauge-equivalence class of edge weights as applying spider moves at all square faces. Since the cluster ensemble map is compatible with spider moves, $\overline A$ is identified with the $A$-network obtained from $A$ by applying spider moves at all square faces. The face $f_{ij}$ is unchanged under this move, and hence $\overline A_{f_{ij}}=A_{f_{ij}}$.

Applying the same matching argument on $\overline{G^\square}$, the symmetric difference of the matchings corresponding to $f_{ij}$ and $f_{j\bar j}$ is exactly the path from $u_{\bar j}^\partial$ to $u_i^\partial$ shown in \cref{fig:paths}(d), again with no additional cycle components. Hence $\overline A_{f_{j\bar j}}/\overline A_{f_{ij}}$ is the alternating product of the edge weights along this path.

We now compare these partial strand products. Since partial strand products depend on the choice of representative of $[\wt]$, we use the representative coming from the Ising model $(G,J)$. In each square corresponding to an edge $e\in E(G)$, opposite edges have the same weight, either $c_e$ or $s_e$, and all the remaining edges have weight $1$. Moreover, the two corresponding strands traverse the square in opposite geometric directions, while the colors of the endpoints of opposite edges are reversed. Therefore an edge is traversed from black to white by one strand if and only if the corresponding opposite edge is traversed from black to white by the other strand. Thus the corresponding edges occur with the same exponent in the respective alternating products. It follows that the alternating product along a segment of $\gamma_{\bar k,k}$ is equal to the alternating product along the corresponding segment of $\gamma_{k,\bar k}$. The same is true on $\overline{G^\square}$, since color-change does not change the numerical edge weights. Therefore we deduce that
\[
\frac{\overline A_{f_{j\bar j}}}{A_{f_{i\bar i}}}
=
\frac{A_{f_{ij}}}{A_{f_{i\bar i}}}
\frac{\overline A_{f_{j\bar j}}}{\overline A_{f_{ij}}}
=
\frac{m_j([\wt])}{m_i([\wt])}.
\]
By the computation preceding \cref{lem:compatibility_of_ratios}, 
$\lambda_k\lambda_{\bar k} =  1 / m_k([\wt])$.
Finally, using~\eqref{eq:A_twisted_plucker} together with $\overline A_{f_{j\bar j}} = \Delta_{S j\bar j} \bigl(\lvec{\tau}_{\mathscr M_\pi}(V)\bigr)$, we get
\begin{align*}
\frac{\Delta_{S j\bar j}(U)}
     {\Delta_{S i\bar i}(U)}
=
\frac{\lambda_j\lambda_{\bar j}}
     {\lambda_i\lambda_{\bar i}}
\frac{\overline A_{f_{j\bar j}}}
     {A_{f_{i\bar i}}}
=
\frac{m_i([\wt])}{m_j([\wt])}
\frac{m_j([\wt])}{m_i([\wt])}
=1.
\end{align*}
Thus $\Delta_{S i\bar i}(U) = \Delta_{S j\bar j}(U)$. Since $S i\bar i\in\mathcal S_\pi^\square$ was arbitrary, this proves that $U\in\Lambda_\pi^{>0}$, and hence $[U] \in \Lambda_\pi^{>0} / T_\pi^{>0}$.
\end{proof}
}

By \cref{lem:hat_tau_inverse_well_defined}, we have a well-defined map
\[
\hat \tau^{-1}_{\mathscr{M}_\pi}: \OGr^{\ge 0}_{n,2n} \cap \Pi_{\mathscr{M}_\pi}^{>0} \rightarrow \Lambda_\pi^{>0}/T_\pi^{>0}.
\]

\subsection{Proof of the main theorem}

We are now ready to prove our main result.
\begin{theorem}\label{thm:main_thm_inverse}
    The map $\hat \tau_{{\mathscr{M}}_\pi}:\Lambda_{\pi}^{>0}/T_\pi^{>0} \rightarrow \OGr^{\ge 0}_{n,2n} \cap \Pi_{\mathscr{M}_\pi}^{>0} $ is a bijection whose inverse is $\hat \tau^{-1}_{\mathscr{M}_\pi}$. They fit into the following commutative diagram:
\[
\begin{tikzcd}
\mathcal C_{G}^{>0} /T_{\pi}^{>0}
  \arrow[r, "q_G"] 
& 
\mathcal I_{G}^{>0} 
  \arrow[d, "\phi_n \circ \operatorname{Corr}_{G}"] 
\\
\Lambda_{\pi}^{>0}/T_\pi^{>0}
  \arrow[u, "\bm H_{G}"]   
  \arrow[r,
    bend right=25,
    swap,
    "{\hat \tau_{{\mathscr{M}}_\pi}}"
  ] 
& 
\OGr^{\ge 0}_{n,2n} \cap \Pi_{\mathscr{M}_\pi}^{>0}
  \arrow[l,
    bend right=25,
    swap,
    "\hat \tau^{-1}_{\mathscr{M}_\pi}"
  ]
\end{tikzcd}
\]
In particular, the composition
\[
q_G \circ \bm H_G \circ \hat \tau^{-1}_{\mathscr{M}_\pi} \circ \phi_n = \operatorname{Corr}_{G}^{-1}
\]
is the inverse of the boundary correlation map $\operatorname{Corr}_{G}$.
\end{theorem}
\begin{proof}
The maps ${\hat \tau_{{\mathscr{M}}_\pi}}$ and $\hat \tau^{-1}_{\mathscr{M}_\pi}$ are defined by modifying the right and left twists respectively by torus actions. By \cref{lem:MS} and \cref{lem:zzpath_from_A}, these cancel out.
\end{proof}

\begin{mdframed}[
    linewidth=0.8pt,
    innertopmargin=8pt,
    innerbottommargin=8pt,
    innerleftmargin=8pt,
    innerrightmargin=8pt,
    skipabove=12pt,
    skipbelow=12pt,
    splittopskip=\topskip,
    splitbottomskip=8pt
]

\begin{algorithmic}[1]

\State \textbf{Input:} A specified reduced cactus graph $G$ with
matching $\pi$, and a boundary correlation matrix
$M\in \operatorname{Corr}_G(\mathcal I_G^{>0})$.

\State \textbf{Output:} The coupling constants
$(J_e)_{e\in E(G)}$.

\State Compute
$V\gets\phi_n(M)\in
\OGr_{n,2n}^{\geq 0}\cap\Pi_{\mathscr M_\pi}^{>0}$.

\State Let
$\bm I=(I_1,\ldots,I_{2n})$
be the Grassmann necklace of the positroid $\mathscr M_\pi$.

\State Compute the left twist
$W\gets{\lvec{\tau}_{\mathscr M_\pi}}(V)$.

\State Initialize
$\bm\lambda=(\lambda_1,\ldots,\lambda_{2n})
\in\mathbb R_{>0}^{2n}$.

\For{each matched pair $\{i,\bar i\}$ of $\pi$}
    \State Choose a representative $i$ of the pair $\{i, \bar{i}\}$.
    \State Set
    $\displaystyle
        \lambda_i\gets 1,\qquad
        \lambda_{\bar i}\gets
        \frac{\Delta_{I_{i+1}}(V)}
             {\Delta_{I_i}(V)}$.
\EndFor

\State Set
$U\gets\bm\lambda\cdot W$.

\For{each $x\in V(G)\sqcup F(G)$}
    \State Let $\bm T(x)$ be the target label of the corresponding
    non-square face of $G^\square$.
    \State Set
    $C_x\gets\Delta_{\bm T(x)}(U)$.
\EndFor

\Needspace{8\baselineskip}

\For{each edge $e=uv\in E(G)$}
    \State Let $f$ and $g$ be the two faces adjacent to $e$.
    \State Compute
    $\displaystyle
        c_e\gets
        \sqrt{\frac{C_uC_v}{C_uC_v+C_fC_g}}$.
    \State Set
    $\displaystyle
        J_e\gets\frac12{\rm arctanh}(c_e)$.
\EndFor

\State \Return $(J_e)_{e\in E(G)}$.

\end{algorithmic}

\medskip
\captionsetup{type=algorithm}
\captionof{algorithm}{Nonrecursive reconstruction of the Ising coupling constants.}
\label{alg:ising_inverse}

\end{mdframed}

\subsection{Examples}

We now compute some examples to illustrate the result of \cref{thm:main_thm_inverse}. The computations were carried out using the computer algebra system {\tt Macaulay2} \cite{M2}. The code used for our computations is available in \cite{code}.

\begin{example}[$n=2$]
Consider the $n=2$ Ising model in \cref{fig:n2_GP_commutativity}(a). We compute the composition $q_G \circ \bm H_G \circ \hat \tau^{-1}_{\mathscr{M}_\pi} \circ \phi_n$ and verify that it is the inverse of $\operatorname{Corr}_{G}$. Let 
\[
M=\begin{bmatrix}
    1&r\\
    r&1
\end{bmatrix} \in \operatorname{Mat}_2^{\text{sym}}(\R,1).
\]
Applying $\phi_2$, we obtain
\[
\tilde M = \begin{bmatrix} 1 & 1 & r&-r\\
-r&r&1&1
\end{bmatrix}
\]
with Pl\"ucker coordinates
\[
\Delta_{12} = \Delta_{34} = 2r, \quad \Delta_{14} = \Delta_{23}=1-r^2, \quad \Delta_{13}=\Delta_{24} = 1+r^2.
\]
The reverse Grassmann necklace is
\[
\bm I^{\text{rev}}_\Gamma=(14,12,23,34).
\]
Using the definition of the left twist, we get
\[
{\lvec{\tau}_{\mathscr{M}_\pi}}(\tilde M)
=
\begin{bmatrix}
\frac{1}{1-r^2} & \frac12 & -\frac{r}{1-r^2} & -\frac{1}{2r}\\[5pt]
\frac{r}{1-r^2} & \frac{1}{2r} & \frac{1}{1-r^2} & \frac12
\end{bmatrix}.
\]
We now apply the torus factor to obtain
\(
\hat \tau^{-1}_{\mathscr M_\pi}(\tilde M).
\) The Grassmann necklace is
\[
\bm I_\Gamma=(12,23,34,14),
\]
and the matching is
\[
\bar 1=3,\qquad \bar 2=4.
\]
Thus, $\bm\lambda=(\lambda_1,\lambda_2,\lambda_3,\lambda_4)$ must satisfy 
\[
\lambda_1\lambda_3
=
\frac{\Delta_{23}}{\Delta_{12}}
=
\frac{1-r^2}{2r},
\qquad
\lambda_2\lambda_4
=
\frac{\Delta_{34}}{\Delta_{23}}
=
\frac{2r}{1-r^2}.
\]
Taking
\[
\lambda_1 = \lambda_2 = 1,
\qquad
\lambda_3 = \frac{1 - r^2}{2r} \quad \text{and} \quad \lambda_4 = \frac{2r}{1-r^2},
\]
we get
\[
\hat \tau^{-1}_{\mathscr M_\pi}(\tilde M)
=
\bm\lambda\cdot {\lvec{\tau}_{\mathscr{M}_\pi}}(\tilde M)
=
\begin{bmatrix}
\frac{1}{1 - r^2} & \frac{1}{2} & -\frac{1}{2} & -\frac{1}{1 - r^2}
\\[0.7em]
\frac{r}{1 - r^2} & \frac{1}{2r} & \frac{1}{2r} & \frac{r}{1 - r^2}
\end{bmatrix}.
\]
Recall that this is an element of $\Lambda_{\pi}^{>0} / T_\pi^{>0}$. Up to the action of $T_\pi^{> 0}$, the Pl\"ucker coordinates of this point are given by:
\[
\Delta_{12}
=
\Delta_{23}
=
\Delta_{34}
=
\frac{1}{2r},
\qquad
\Delta_{13}
=
\Delta_{24}
=
\frac{1+r^2}{2r(1-r^2)},
\qquad
\Delta_{14}
=
\frac{2r}{(1-r^2)^2}.
\]
So after applying \(\bm H_G\) (see Figure~\ref{fig:comm_Ising_cluster_ensemble}), we compute $C \in \mathcal{C}_{G}^{> 0} / T_{\pi}^{> 0}$ and get:
\[
C_{b^\partial_1} = C_{f} = C_{g} = \frac{1}{2r}
\quad \text{and} \quad
C_{b^\partial_2} = \frac{2r}{(1-r^2)^2}.
\]
Finally, applying the Ising cluster ensemble map \(q_G\), we get
\[
c_e
=
\sqrt{
\frac{C_{b^\partial_1}C_{b^\partial_2}}{C_{b^\partial_1} C_{b^\partial_2} + C_f C_g}
}
=
\frac{2r}{1+r^2}, \qquad 
s_e
=
\sqrt{
\frac{C_fC_g}{C_{b^\partial_1} C_{b^\partial_2} + C_f C_g}
}
=
\frac{1-r^2}{1+r^2}.
\]
Thus, the composition
\(
q_G \circ \bm H_G \circ \hat \tau^{-1}_{\mathscr{M}_\pi} \circ \phi_2(M)
\)
is the map
\[
\begin{bmatrix}
    1&r\\
    r&1
\end{bmatrix} \mapsto (c_e,s_e) = \left( \frac{2r}{1+r^2},\frac{1-r^2}{1+r^2}  \right),
\]
which corresponds to the Ising model with coupling 
\[
J = \frac{1}{2} {\rm arctanh}(c_e)  = \frac{1}{2} {\rm arctanh}\left(\frac{2r}{1 + r^2}\right) =  \operatorname{arctanh}(r).
\] 

By \cref{eg:ising_n_2}, \(\operatorname{Corr}_G\) sends the Ising model with coupling constant $J$ to 
\[
\begin{bmatrix} 1 & \tanh J\\ \tanh J &  1
\end{bmatrix}.
\]
Hence,
\(
q_G \circ \bm H_G \circ \hat \tau^{-1}_{\mathscr{M}_\pi} \circ \phi_2
\)
is the inverse of \(\operatorname{Corr}_G\) in this example.
\end{example}

\begin{example}[$n=4$] Consider the $n = 4$ Ising model in \cref{fig:n=4Ising}. The Grassmann necklace is
\[
\bm I_{\Gamma} = (1234, 2346, 3467, 4567, 5678, 3678, 1378, 1238),
\]
and the reverse necklace is
{
\[
\bm{I}_{\Gamma}^{\mathrm{rev}} = (1578, 1258, 1238, 1234, 1245, 2456, 4567, 5678).
\]
}

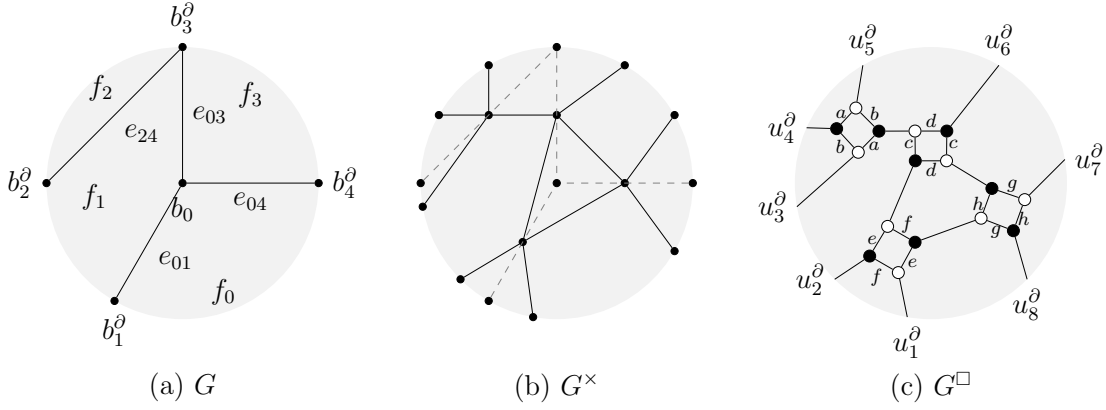
\begin{figure}[ht]
    \centering
    \scalebox{0.9}{
    \begin{tikzpicture}
        \begin{scope}[shift={(0,0)}]
            \def\r{2};
            \fill[black!5] (0,0) circle (1*\r cm);
            \coordinate[nvert, label=below: $b^\partial_1$] (b1) at (-120:\r);
            \coordinate[nvert, label=right: $b^\partial_4$] (b4) at (0:\r);
            \coordinate[nvert, label=above: $b^\partial_3$] (b3) at (90:\r);
            \coordinate[nvert, label=left: $b^\partial_2$] (b2) at (180:\r);

            \coordinate[nvert, label = below: $b_0$] (b0) at (140:0*\r);

            \draw[] (b1) -- node[below right]{$e_{01}$} (b0);
            \draw[] (b4) -- node[      below]{$e_{04}$} (b0);
            \draw[] (b3) -- node[      right]{$e_{03}$} (b0);
            \draw[] (b2) -- node[below right]{$e_{23}$} (b3);

            \node at (0.6,-1.6) {$f_0$};
            \node at (-1.3,-0.2) {$f_1$};
            \node at (-1.2,1.4) {$f_2$};
            \node at (1,1.3) {$f_3$};
            
            \node at (0,-3) {(a) $G$};

        \end{scope}

        \begin{scope}[shift={(5.5,0)}]
            \def\r{2};
            \fill[black!5] (0,0) circle (1*\r cm);
            \coordinate[nvert] (b1) at (-120:\r);
            \coordinate[nvert] (b4) at (0:\r);
            \coordinate[nvert] (b3) at (90:\r);
            \coordinate[nvert] (b2) at (180:\r);

            \coordinate[nvert] (u1) at (30:\r);
            \coordinate[nvert] (u2) at (60:\r);
            \coordinate[nvert] (u3) at (120:\r);
            \coordinate[nvert] (u4) at (150:\r);
            
            \coordinate[nvert] (u5) at ( -30:\r);
            \coordinate[nvert] (u6) at ( -100:\r);
            \coordinate[nvert] (u7) at (-135:\r);
            \coordinate[nvert] (u8) at (-170:\r);

            \coordinate[nvert] (w01) at ( -120:  0.5  * \r);
            \coordinate[nvert] (w02) at (0: 0.5 * \r);
            \coordinate[nvert] (w03) at (90: 0.5 * \r);
            \coordinate[nvert] (w34) at ( 135: 0.707 * \r);

            \coordinate[nvert] (b0) at (140:0*\r);

            \draw[gray, dashed] (b1) -- node[left]{} (b0);
            \draw[gray, dashed] (b3) -- node[right]{} (b0);
            \draw[gray, dashed] (b4) -- node[below]{} (b0);
            \draw[gray, dashed] (b2) -- node[above left]{} (b3);

            \draw[] (w01) -- node[left]{} (u6);
            \draw[] (w01) -- node[left]{} (w02);
            \draw[] (w02) -- node[left]{} (u5);
            \draw[] (w02) -- node[left]{} (u1);
            \draw[] (w02) -- node[left]{} (w03);
            \draw[] (w03) -- node[left]{} (u2);
            \draw[] (w03) -- node[left]{} (w34);
            \draw[] (w34) -- node[left]{} (u3);
            \draw[] (w34) -- node[left]{} (u4);
            \draw[] (w34) -- node[left]{} (u8);
            \draw[] (w01) -- node[left]{} (u7);
            \draw[] (w01) -- node[left]{} (w03);

            \node at (0,-3) {(b)  $G^\times$};
        \end{scope}

        \begin{scope}[shift={(11,0)}]
            \def\r{2};
            \fill[black!5] (0,0) circle (1*\r cm);
            
            \coordinate[label=right: $u_7^\partial$] (u7) at (10:\r);
            \coordinate[label=above: $u_6^\partial$] (u6) at (60:\r);
            \coordinate[label=above: $u_5^\partial$] (u5) at (120:\r);
            \coordinate[label=left: $u_4^\partial$] (u4) at (156:\r);
            \coordinate[label=below: $u_8^\partial$] (u8) at ( -45:\r);
            \coordinate[label=below: $u_1^\partial$] (u1) at ( -100:\r);
            \coordinate[label=left: $u_2^\partial$] (u2) at (-135:\r);
            \coordinate[label=left: $u_3^\partial$] (u3) at (-170:\r);

            \coordinate[wvert] (w01_1) at ( -110:0.7*\r);
            \coordinate[wvert] (w01_4) at ( -135:0.45*\r);
            \coordinate[bvert] (w01_2) at (  -105:0.45*\r);
            \coordinate[bvert] (w01_3) at ( -130:0.7*\r);
            
            \draw[] (w01_1) -- node[left]{} (w01_3);
            \draw[] (w01_1) -- node[left]{} (w01_2);
            \draw[] (w01_3) -- node[left]{} (w01_4);
            \draw[] (w01_4) -- node[left]{} (w01_2);

            \coordinate[wvert] (w02_1) at ( -10:0.7*\r);
            \coordinate[wvert] (w02_4) at (-35:0.45*\r);
            \coordinate[bvert] (w02_2) at ( -05:0.45*\r);
            \coordinate[bvert] (w02_3) at ( -30:0.7*\r);
            
            \draw[] (w02_1) -- node[left]{} (w02_3);
            \draw[] (w02_1) -- node[left]{} (w02_2);
            \draw[] (w02_3) -- node[left]{} (w02_4);
            \draw[] (w02_4) -- node[left]{} (w02_2);

            \coordinate[wvert] (w03_4) at (55:0.2*\r);
            \coordinate[wvert] (w03_1) at (107:0.4*\r);
            \coordinate[bvert] (w03_3) at ( 73:0.4*\r);
            \coordinate[bvert] (w03_2) at ( 125:0.2*\r);
            
            \draw[] (w03_1) -- node[left]{} (w03_3);
            \draw[] (w03_1) -- node[left]{} (w03_2);
            \draw[] (w03_3) -- node[left]{} (w03_4);
            \draw[] (w03_4) -- node[left]{} (w03_2);

            \coordinate[bvert] (w34_3) at ( 150: 0.8 * \r); 
            \coordinate[wvert] (w34_4) at ( 157: 0.58 * \r); 

            \coordinate[bvert] (w34_1) at ( 135: 0.54 * \r);
            \coordinate[wvert] (w34_2) at ( 135: 0.78 * \r); 
            
            \draw[] (w34_1) -- node[left]{} (w34_2);
            \draw[] (w34_2) -- node[left]{} (w34_3);
            \draw[] (w34_3) -- node[left]{} (w34_4);
            \draw[] (w34_4) -- node[left]{} (w34_1);

            \draw[] (w01_1) -- node[left]{} (u1);
            \draw[] (w01_2) -- node[left]{} (w02_4);
            \draw[] (w02_3) -- node[left]{} (u8);
            \draw[] (w02_1) -- node[left]{} (u7);
            \draw[] (w02_2) -- node[left]{} (w03_4);
            \draw[] (w03_3) -- node[left]{} (u6);
            \draw[] (w03_1) -- node[left]{} (w34_1);
            \draw[] (w34_2) -- node[left]{} (u5);
            \draw[] (w34_3) -- node[left]{} (u4);
            \draw[] (w34_4) -- node[left]{} (u3);
            \draw[] (w01_3) -- node[left]{} (u2);
            \draw[] (w01_4) -- node[left]{} (w03_2);

            \node at (129:0.66*\r) {\tiny $b$};
            \node at (158:0.72*\r) {\tiny $b$};
            \node at (142.5:0.83*\r) {\tiny $a$};
            \node at (146.5:0.5*\r) {\tiny $a$};

            \node at (90:0.46*\r) {\tiny $d$};
            \node at (90:0.10*\r) {\tiny $d$};
            \node at (120:0.33*\r) {\tiny $c$};
            \node at (60:0.33*\r) {\tiny $c$};
            
            \node at (-120:0.35*\r) {\tiny $f$};
            \node at (-120:0.8*\r)  {\tiny $f$};
            \node at (-136:0.60*\r) {\tiny $e$};
            \node at (-103:0.6*\r)  {\tiny $e$};
            
            \node at (-23:0.38*\r) {\tiny $h$};
            \node at (-21:0.73*\r) {\tiny $h$};
            \node at (-02:0.60*\r) {\tiny $g$};
            \node at (-37:0.60*\r) {\tiny $g$};

            \node at (0,-3) {(c)  $G^\square$};
        \end{scope}
    \end{tikzpicture}}
    
    \caption{An Ising model with $n=4$ boundary vertices (a), its medial graph (b) and the corresponding dimer model (c). Here the parameters $a,b,c,d,e,f,g, h$ are given in terms of the coupling constants as follows:\\
    $a = \tanh(2J_{23}), \quad c = \tanh(2J_{03}), \quad e = \tanh(2J_{01}), \quad g = \tanh(2J_{04})\\
    \quad b = \sech(2J_{23}), \quad d = \sech(2J_{03}), \quad f = \sech(2J_{01}), \quad h = \sech(2 J_{04})$.}
    \label{fig:n=4Ising}
\end{figure}

\noindent The boundary measurement map computed from the edge weights is
\[
{\rm Meas} = \resizebox{0.4\textwidth}{!}{$
\begin{bmatrix}
1 & 0 & 0 & 0 &
0 &
- \dfrac{cf}{de} &
- \dfrac{dh+f}{deg} &
- \dfrac{fh+d}{deg}
\\[3mm]
0 & 1 & 0 & 0 &
0 &
\dfrac{c}{de} &
\dfrac{dfh+1}{deg} &
\dfrac{df+h}{deg}
\\[3mm]
0 & 0 & 1 & 0 &
- \dfrac{b}{a} &
- \dfrac{1}{ad} &
- \dfrac{c}{adg} &
- \dfrac{ch}{adg}
\\[3mm]
0 & 0 & 0 & 1 &
\dfrac{1}{a} &
\dfrac{b}{ad} &
\dfrac{bc}{adg} &
\dfrac{bch}{adg}
\end{bmatrix}
$}.
\]
We compute the correlation matrix using \cite[Lemma 3.2]{GalashinPylyavskyy} and we get
\[
M :=  \phi_4^{-1}({\rm Meas}) = 
\resizebox{0.5\textwidth}{!}{$
\begin{bmatrix}
1 &
\dfrac{a(1-d)(1 - f)}{ce(b+1)} &
\dfrac{c(1-f)}{e(d+1)} &
\dfrac{e(1-h)}{g(f+1)}
\\[1.2em]
\dfrac{a(1-d)(1-f)}{ce(b+1)} &
1 &
\dfrac{1-b}{a} &
\dfrac{a(1-d)(1-h)}{gc(b+1)}
\\[1.2em]
\dfrac{c(1-f)}{e(d+1)} &
\dfrac{1-b}{a} &
1 &
\dfrac{c(1-h)}{g(d+1)}
\\[1.2em]
\dfrac{e(1-h)}{g(f+1)} &
\dfrac{a(1-d)(1-h)}{gc(b+1)} &
\dfrac{c(1-h)}{g(d+1)} &
1
\end{bmatrix}
$}.
\]
So the matrix $\tilde{M}$ is given by
\[
\tilde{M} = 
\resizebox{0.93\textwidth}{!}{$
\begin{bmatrix}
1 & 1 &
\dfrac{a(1-d)(1 - f)}{ce(b+1)} &
-\dfrac{a(1-d)(1 - f)}{ce(b+1)} &
-\dfrac{c(1-f)}{e(d+1)} &
\dfrac{c(1-f)}{e(d+1)} &
\dfrac{e(1-h)}{g(f+1)} &
-\dfrac{e(1-h)}{g(f+1)}
\\[1.2em]
-\dfrac{a(1-d)(1-f)}{ce(b+1)} &
\dfrac{a(1-d)(1-f)}{ce(b+1)} &
1 & 1 &
\dfrac{1-b}{a} &
-\dfrac{1-b}{a} &
-\dfrac{a(1-d)(1-h)}{gc(b+1)} &
\dfrac{a(1-d)(1-h)}{gc(b+1)}
\\[1.2em]
\dfrac{c(1-f)}{e(d+1)} &
-\dfrac{c(1-f)}{e(d+1)} &
-\dfrac{1-b}{a} &
\dfrac{1-b}{a} &
1 & 1 &
\dfrac{c(1-h)}{g(d+1)}&
-\dfrac{c(1-h)}{g(d+1)}
\\[1.2em]
-\dfrac{e(1-h)}{g(f+1)} &
 \dfrac{e(1-h)}{g(f+1)} &
 \dfrac{a(1-d)(1-h)}{gc(b+1)} &
-\dfrac{a(1-d)(1-h)}{gc(b+1)} &
-\dfrac{c(1-h)}{g(d+1)} &
 \dfrac{c(1-h)}{g(d+1)}&
1 & 1
\end{bmatrix}.
$}
\]
Note that the matrices ${\rm Meas}$ and $\tilde{M}$ define the same element in $\OGr(4,8) \cap \Pi^{>0}_{\mathscr{M}_\pi}$. Applying the left twist to this matrix we get the following
\[
\lt(\tilde{M}) =
\resizebox{0.85\textwidth}{!}{$
\begin{bmatrix}
\dfrac{f+1}{2f} &
\dfrac{1}{2} &
0 &
0 &
0 &
-\dfrac{ed}{2fc} &
-\dfrac{g(f+1)}{2he} &
-\dfrac{(f+1)(d+h)}{2ged}
\\[1.2em]
\dfrac{ea}{2fc(1-b)} &
\dfrac{(b+1)(df+h)}{2heca} &
\dfrac{b+1}{2b} &
\dfrac{1}{2} &
-\dfrac{a}{2b} &
0 &
0 &
-\dfrac{hca}{2gd(1-b)}
\\[1.2em]
-\dfrac{ed}{2fc} &
-\dfrac{dh+f}{2hec} &
\dfrac{a}{2b} &
\dfrac{b+d}{2da} &
\dfrac{bd+1}{2db} &
\dfrac{1}{2} &
0 &
0
\\[1.2em]
0 &
-\dfrac{gf}{2he} &
0 &
\dfrac{gcb}{2d a(1-h)} &
\dfrac{gc}{2db(1-h)} &
\dfrac{g(d+f)}{2fc(1-h)} &
\dfrac{h+1}{2h} &
\dfrac{1}{2}
\end{bmatrix}
$}.
\]
We then deduce that $\bm \lambda$ must satisfy the following equations
\begin{align*}
\lambda_1 \lambda_{6} = \frac{\Delta_{2346}}{\Delta_{1234}} = \frac{\Delta_{1378}}{\Delta_{3678}}
= \frac{cf}{de}, 
\quad 
\lambda_2 \lambda_{7} = \frac{\Delta_{3467}}{\Delta_{2346}} = \frac{\Delta_{1238}}{\Delta_{1378}} = \frac{eh}{fg},\\
\lambda_3 \lambda_{5} = \frac{\Delta_{4567}}{\Delta_{3467}} = \frac{\Delta_{3678}}{\Delta_{5678}} = \frac{b}{a},
\quad 
{
\lambda_4 \lambda_{8} = \frac{\Delta_{5678}}{\Delta_{4567}} = \frac{\Delta_{1234}}{\Delta_{1238}}
= \frac{adg}{bch}}.
\end{align*}
We can then choose 
\[
    \lambda_1 = \lambda_2 = \lambda_3 = \lambda_4 = 1 \quad \text{ and } \quad 
    \lambda_5 = \frac{b}{a}, \ \lambda_6 = \frac{cf}{de}, \ 
    \lambda_7 = \frac{eh}{fg}, \ \lambda_8 = \frac{adg}{bch}.
\]

Computing $\hat{\tau}_{\mathscr{M}_\pi}^{-1}(\tilde{M}) = \bm\lambda\cdot {\lvec{\tau}_{\mathscr{M}_\pi}}(\tilde M)$ yields the following:
\[
\hat{\tau}^{-1}_{\mathscr{M}_\pi}(\tilde{M}) = 
\resizebox{0.85\textwidth}{!}{$
\begin{bmatrix}
\dfrac{f+1}{2f} &
\dfrac{1}{2} &
0 &
0 &
0 &
-\dfrac{1}{2} &
-\dfrac{f+1}{2f} &
-\dfrac{ea(d+h)}{2hcb(1-f)}
\\[1.2em]
\dfrac{e(b+1)}{2fca} &
\dfrac{a(df+h)}{2hec(1-b)} &
\dfrac{b+1}{2b} &
\dfrac{1}{2} &
-\dfrac{1}{2} &
0 &
0 &
-\dfrac{b+1}{2b}
\\[1.2em]
-\dfrac{ed}{2fc} &
-\dfrac{dh+f}{2hec} &
\dfrac{a}{2b} &
\dfrac{b+d}{2da} &
\dfrac{bd+1}{2da} &
\dfrac{fc}{2ed} &
0 &
0
\\[1.2em]
0 &
-\dfrac{gf}{2he} &
0 &
\dfrac{(h+1)cb}{2gda} &
\dfrac{(h+1)c}{2gda} &
\dfrac{(h+1)(d+f)}{2ged} &
\dfrac{ge}{2f(1-h)} &
\dfrac{gda}{2hcb}
\end{bmatrix}
$}.
\]
We recall that this is an element of $\Lambda^{>0}_{\pi} / T^{> 0}_{\pi}$. We now compute the $C$-variables which are indexed by the vertices and faces of $G$ in \cref{fig:n=4Ising}(a). They correspond to the Pl\"ucker coordinates of $\hat{\tau}^{-1}_{\mathscr{M}_\pi}(\tilde M)$ indexed by the target labels of the non-square faces of the plabic graph in \cref{fig:n=4Ising-2}. We then obtain

\begin{table}[H]
    \[{
    \begin{array}{|c|c|c|c|c|c|c|c|c|c|}
        \hline
        x
        & b_0
        & b_1^\partial
        & b_2^\partial
        & b_3^\partial
        & b_4^\partial
        & f_0
        & f_1
        & f_2
        & f_3
        \\
        \hline
        \bm T(x)
        & 1347
        & 2346
        & 4567
        & 3678
        & 1238
        & 1234
        & 3467
        & 5678
        & 1378
        \\
        \hline
    \end{array}}
    \]
    \caption{Target label assignment to elements of $V(G) \sqcup F(G)$.}
    \label{tab:target_label_assignment}
\end{table}

{
\[
\begin{aligned}
C_{b_0}
&=
\Delta_{1347}
=
\tfrac{(h+1)(f+1)e(d+1)(b+1)}
{8gf^2da},
&
C_{b^\partial_1}
&=
\Delta_{2346}
=
\tfrac{g(f+1)(d+1)(b+1)}
{8(1-h)eda},
\\[0.6em]
C_{b^\partial_2}
&=
\Delta_{4567}
=
\tfrac{ge(d+1)(b+1)}
{8(1-h)(1-f)da},
&
C_{b^\partial_3}
&=
\Delta_{3678}
=
\tfrac{(h+1)(f+1)(d+1)(b+1)a}
{8gedb^2},
\\[0.6em]
C_{b^\partial_4}
&=
\Delta_{1238}
=
\tfrac{(h+1)g(f+1)d(b+1)a}
{8h^2e(1-d)b^2},
&
C_{f_0}
&=
\Delta_{1234}
=
\tfrac{g(f+1)(d+1)(b+1)}
{8(1-h)eda},
\\[0.6em]
C_{f_1}
&=
\Delta_{3467}
=
\tfrac{ge(d+1)(b+1)}
{8(1-h)(1-f)da},
&
C_{f_2}
&=
\Delta_{5678}
=
\tfrac{g(f+1)(d+1)(b+1)}
{8(1-h)eda},
\\[0.6em]
C_{f_3}
&=
\Delta_{1378}
=
\tfrac{(h+1)(f+1)ed(b+1)a}
{8gf^2(1-d)b^2}.
\end{aligned}
\]
}

\begin{figure}[ht]
    \centering
    \scalebox{1.4}{
    \begin{tikzpicture}[every label/.append style={font=\scriptsize, scale=0.7}]
                \def\r{2};
                \fill[black!5] (0,0) circle (1*\r cm);
                
                \coordinate[label=below:  $u_1^\partial$] (u1) at ( -100:\r);
                \coordinate[label=left:   $u_2^\partial$] (u2) at (-135:\r);
                \coordinate[label=left:   $u_3^\partial$] (u3) at (-170:\r);
                \coordinate[label=left:   $u_4^\partial$] (u4) at (156:\r);
                \coordinate[label=above:  $u_5^\partial$] (u5) at (120:\r);
                \coordinate[label=above right: $u_6^\partial$] (u6) at (60:\r);
                \coordinate[label=right:  $u_7^\partial$] (u7) at (10:\r);
                \coordinate[label=below: $u_8^\partial$] (u8) at ( -45:\r);
                
                \coordinate[wvert] (w01_1) at ( -110:0.7*\r);
                \coordinate[wvert] (w01_4) at ( -135:0.45*\r);
                \coordinate[bvert] (w01_2) at (  -105:0.45*\r);
                \coordinate[bvert] (w01_3) at ( -130:0.7*\r);
                
                \draw[] (w01_1) -- node[left]{} (w01_3);
                \draw[] (w01_1) -- node[left]{} (w01_2);
                \draw[] (w01_3) -- node[left]{} (w01_4);
                \draw[] (w01_4) -- node[left]{} (w01_2);

                \coordinate[wvert] (w02_1) at ( -10:0.7*\r);
                \coordinate[wvert] (w02_4) at (-35:0.45*\r);
                \coordinate[bvert] (w02_2) at ( -05:0.45*\r);
                \coordinate[bvert] (w02_3) at ( -30:0.7*\r);
                
                \draw[] (w02_1) -- node[left]{} (w02_3);
                \draw[] (w02_1) -- node[left]{} (w02_2);
                \draw[] (w02_3) -- node[left]{} (w02_4);
                \draw[] (w02_4) -- node[left]{} (w02_2);

                \coordinate[wvert] (w03_4) at (55:0.2*\r);
                \coordinate[wvert] (w03_1) at (107:0.4*\r);
                \coordinate[bvert] (w03_3) at ( 73:0.4*\r);
                \coordinate[bvert] (w03_2) at ( 125:0.2*\r);
                
                \draw[] (w03_1) -- node[left]{} (w03_3);
                \draw[] (w03_1) -- node[left]{} (w03_2);
                \draw[] (w03_3) -- node[left]{} (w03_4);
                \draw[] (w03_4) -- node[left]{} (w03_2);
    
                \coordinate[bvert] (w34_3) at ( 150: 0.8 * \r); 
                \coordinate[wvert] (w34_4) at ( 157: 0.58 * \r); 
    
                \coordinate[bvert] (w34_1) at ( 135: 0.54 * \r);
                \coordinate[wvert] (w34_2) at ( 135: 0.78 * \r); 
                
                \draw[] (w34_1) -- node[left]{} (w34_2);
                \draw[] (w34_2) -- node[left]{} (w34_3);
                \draw[] (w34_3) -- node[left]{} (w34_4);
                \draw[] (w34_4) -- node[left]{} (w34_1);

                \draw[] (w01_1) -- node[left]{} (u1);
                \draw[] (w01_2) -- node[left]{} (w02_4);
                \draw[] (w02_3) -- node[left]{} (u8);
                \draw[] (w02_1) -- node[left]{} (u7);
                \draw[] (w02_2) -- node[left]{} (w03_4);
                \draw[] (w03_3) -- node[left]{} (u6);
                \draw[] (w03_1) -- node[left]{} (w34_1);
                \draw[] (w34_2) -- node[left]{} (u5);
                \draw[] (w34_3) -- node[left]{} (u4);
                \draw[] (w34_4) -- node[left]{} (u3);
                \draw[] (w01_3) -- node[left]{} (u2);
                \draw[] (w01_4) -- node[left]{} (w03_2);

                \node[scale=0.5, blue] at (0, -0.25) {$1347$};
                \node[scale=0.5, blue] at (-0.55, -0.97) {$2347$};
                \node[scale=0.5, blue] at (-0.7, -1.6) {$\bm{2346}$};
                \node[scale=0.5, blue] at (-1.2, -0.4) {$\bm{3467}$};
                \node[scale=0.5, blue] at (0.4, -1.4) {$\bm{1234}$};
                \node[scale=0.5, blue] at (1.6, -0.7) {$\bm{1238}$};
                \node[scale=0.5, blue] at (1.05, -0.4) {$1348$};
                \node[scale=0.5, blue] at (0, 0.55) {$1367$};
                \node[scale=0.5, blue] at (1, 0.55) {$\bm{1378}$};
                \node[scale=0.5, blue] at (0, 1.55) {$\bm{3678}$};
                \node[scale=0.5, blue] at (-1.07, 0.80) {$4678$};
                \node[scale=0.5, blue] at (-1.6, 0.40) {$\bm{4567}$};
                \node[scale=0.5, blue] at (-1.45, 1.15) {$\bm{5678}$};

            \begin{scope}[shift={(6,0)}]
                    \def\r{2};
                    \fill[black!5] (0,0) circle (1*\r cm);
                    
                    \coordinate[label=below:  $u_1^\partial$] (u1) at ( -100:\r);
                    \coordinate[label=left:   $u_2^\partial$] (u2) at (-135:\r);
                    \coordinate[label=left:   $u_3^\partial$] (u3) at (-170:\r);
                    \coordinate[label=left:   $u_4^\partial$] (u4) at (156:\r);
                    \coordinate[label=above:  $u_5^\partial$] (u5) at (120:\r);
                    \coordinate[label=above right: $u_6^\partial$] (u6) at (60:\r);
                    \coordinate[label=right:  $u_7^\partial$] (u7) at (10:\r);
                    \coordinate[label=below: $u_8^\partial$] (u8) at ( -45:\r);
                    
                    \coordinate[wvert] (w01_1) at ( -110:0.7*\r);
                    \coordinate[wvert] (w01_4) at ( -135:0.45*\r);
                    \coordinate[bvert] (w01_2) at (  -105:0.45*\r);
                    \coordinate[bvert] (w01_3) at ( -130:0.7*\r);
                    
                    \draw[] (w01_1) -- node[left]{} (w01_3);
                    \draw[] (w01_1) -- node[left]{} (w01_2);
                    \draw[] (w01_3) -- node[left]{} (w01_4);
                    \draw[] (w01_4) -- node[left]{} (w01_2);

                    \coordinate[wvert] (w02_1) at ( -10:0.7*\r);
                    \coordinate[wvert] (w02_4) at (-35:0.45*\r);
                    \coordinate[bvert] (w02_2) at ( -05:0.45*\r);
                    \coordinate[bvert] (w02_3) at ( -30:0.7*\r);
                    
                    \draw[] (w02_1) -- node[left]{} (w02_3);
                    \draw[] (w02_1) -- node[left]{} (w02_2);
                    \draw[] (w02_3) -- node[left]{} (w02_4);
                    \draw[] (w02_4) -- node[left]{} (w02_2);

                    \coordinate[wvert] (w03_4) at (55:0.2*\r);
                    \coordinate[wvert] (w03_1) at (107:0.4*\r);
                    \coordinate[bvert] (w03_3) at ( 73:0.4*\r);
                    \coordinate[bvert] (w03_2) at ( 125:0.2*\r);
                    
                    \draw[] (w03_1) -- node[left]{} (w03_3);
                    \draw[] (w03_1) -- node[left]{} (w03_2);
                    \draw[] (w03_3) -- node[left]{} (w03_4);
                    \draw[] (w03_4) -- node[left]{} (w03_2);
        
                    \coordinate[bvert] (w34_3) at ( 150: 0.8 * \r); 
                    \coordinate[wvert] (w34_4) at ( 157: 0.58 * \r); 
        
                    \coordinate[bvert] (w34_1) at ( 135: 0.54 * \r);
                    \coordinate[wvert] (w34_2) at ( 135: 0.78 * \r); 
                    
                    \draw[] (w34_1) -- node[left]{} (w34_2);
                    \draw[] (w34_2) -- node[left]{} (w34_3);
                    \draw[] (w34_3) -- node[left]{} (w34_4);
                    \draw[] (w34_4) -- node[left]{} (w34_1);

                    \draw[] (w01_1) -- node[left]{} (u1);
                    \draw[] (w01_2) -- node[left]{} (w02_4);
                    \draw[] (w02_3) -- node[left]{} (u8);
                    \draw[] (w02_1) -- node[left]{} (u7);
                    \draw[] (w02_2) -- node[left]{} (w03_4);
                    \draw[] (w03_3) -- node[left]{} (u6);
                    \draw[] (w03_1) -- node[left]{} (w34_1);
                    \draw[] (w34_2) -- node[left]{} (u5);
                    \draw[] (w34_3) -- node[left]{} (u4);
                    \draw[] (w34_4) -- node[left]{} (u3);
                    \draw[] (w01_3) -- node[left]{} (u2);
                    \draw[] (w01_4) -- node[left]{} (w03_2);

                    \node[scale=0.5, red] at (0, -0.25) {$2568$};
                    \node[scale=0.5, red] at (-0.55, -0.97) {$2578$};
                    \node[scale=0.5, red] at (-0.7, -1.6) {$\bm{1578}$};
                    \node[scale=0.5, red] at (-1.2, -0.4) {$\bm{1258}$};
                    \node[scale=0.5, red] at (0.4, -1.4) {$\bm{5678}$};
                    \node[scale=0.5, red] at (1.6, -0.7) {$\bm{4567}$};
                    \node[scale=0.5, red] at (1.05, -0.4) {$4568$};
                    \node[scale=0.5, red] at (0, 0.55) {$1256$};
                    \node[scale=0.5, red] at (1, 0.55) {$\bm{2456}$};
                    \node[scale=0.5, red] at (0, 1.55) {$\bm{1245}$};
                    \node[scale=0.5, red] at (-1.07, 0.80) {$1248$};
                    \node[scale=0.5, red] at (-1.6, 0.40) {$\bm{1238}$};
                    \node[scale=0.5, red] at (-1.45, 1.15) {$\bm{1234}$};
            \end{scope}
            \end{tikzpicture}
            }
    \caption{The target face labels $\big({\bm T}(f), f \in  F(G^{\square}) \big) $ (left) and the source face labels $\big({\bm S}(f), f \in  F(G^{\square}) \big)$ (right).}
    \label{fig:n=4Ising-2}
\end{figure}
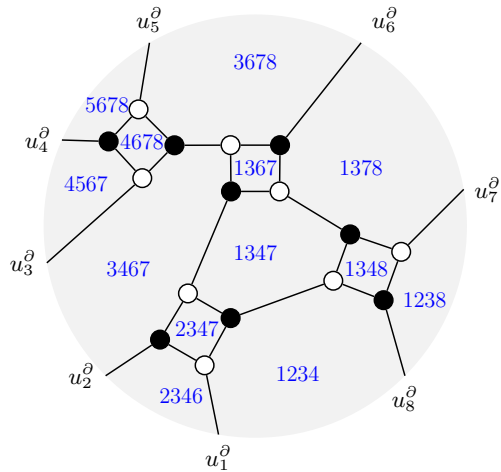
\noindent We have now computed the map $\bm H_G \circ \hat{\tau}^{-1}_{\mathscr{M}_\pi}(\tilde{M})$. To recover the edge weights in \cref{fig:n=4Ising}(c) from the C-variables, we compute $q_G \circ \bm H_G \circ \hat{\tau}^{-1}_{\mathscr{M}_\pi} (\tilde{M})$ to get:
\[
\resizebox{0.98\textwidth}{!}{$
\begin{aligned}
    a &= \sqrt{\frac{C_{b^{\partial}_2} C_{b^{\partial}_3}}{C_{b^{\partial}_2} C_{b^{\partial}_3} + C_{f_1} C_{f_2}}},
    \quad 
    c &= \sqrt{\frac{C_{b_0} C_{b^{\partial}_3}}{C_{b_0} C_{b^{\partial}_3} + C_{f_1} C_{f_3}}},
    \quad 
    e &= \sqrt{\frac{C_{b_0} C_{b^{\partial}_1}}{C_{b_0} C_{b^{\partial}_1} + C_{f_0} C_{f_1}}}, 
    \quad
    g &= \sqrt{\frac{C_{b_0} C_{b^{\partial}_4}}{C_{b_0} C_{b^{\partial}_4} + C_{f_0} C_{f_3}}},\\
    b &= \sqrt{\frac{C_{f_1} C_{f_2}}{C_{b^{\partial}_2} C_{b^{\partial}_3} + C_{f_1} C_{f_2}}}, 
    \quad
    d &= \sqrt{\frac{C_{f_1} C_{f_3}}{C_{b_0} C_{b^{\partial}_3} + C_{f_1} C_{f_3}}}, 
    \quad
    f &= \sqrt{\frac{C_{f_0} C_{f_1}}{C_{b_0} C_{b^{\partial}_1} + C_{f_0} C_{f_1}}},
    \quad 
    h &= \sqrt{\frac{C_{f_0} C_{f_3}}{C_{b_0} C_{b^{\partial}_4} + C_{f_0} C_{f_3}}}.
\end{aligned}$}
\]

The coupling constants $J_{01}, J_{03}, J_{04}$ and $J_{23}$ are then obtained as follows:

\[
J_{01} = \frac{1}{2} {\rm arctanh}(e),  \quad  
J_{03} = \frac{1}{2} {\rm arctanh}(c),  \quad 
J_{04} = \frac{1}{2} {\rm arctanh}(g),  \quad 
\text{and} \quad
J_{23} = \frac{1}{2} {\rm arctanh}(a). 
\]
\end{example}

\bibliographystyle{acm}
\bibliography{references}

\end{document}